\documentclass[aos]{imsart}

\RequirePackage{amsthm,amsmath,amsfonts,amssymb}
\RequirePackage[numbers]{natbib}
\newcommand{\mathbbm}[1]{\text{\usefont{U}{bbm}{m}{n}#1}}
\usepackage{mathrsfs} 
\usepackage{xcolor}
\usepackage{comment}
\usepackage{mathabx}
\usepackage{hyperref}

\usepackage{definitions}

\startlocaldefs
\newtheorem{theorem}{Theorem}[section]
\newtheorem{lemma}[theorem]{Lemma}

\theoremstyle{remark}

\newtheorem{remark}[theorem]{Remark}

\endlocaldefs

\begin{document}
\begin{frontmatter}
\title{Optimal high-dimensional and nonparametric distributed testing under communication constraints}
%\title{A sample article title with some additional note\thanksref{T1}}
\runtitle{Optimal distributed testing}
%\thankstext{T1}{A sample of additional note to the title.}

\begin{aug}
%%%%%%%%%%%%%%%%%%%%%%%%%%%%%%%%%%%%%%%%%%%%%%
%%Only one address is permitted per author. %%
%%Only division, organization and e-mail is %%
%%included in the address.                  %%
%%Additional information can be included in %%
%%the Acknowledgments section if necessary. %%
%%%%%%%%%%%%%%%%%%%%%%%%%%%%%%%%%%%%%%%%%%%%%%
\author[A]{\fnms{Botond} \snm{Szab\'{o}}\ead[label=e1]{botond.szabo@unibocconi.it}},
\author[B]{\fnms{Lasse} \snm{Vuursteen}\ead[label=e2]{
l.vuursteen@tudelft.nl}}
\and
\author[C]{\fnms{Harry} \snm{van Zanten}\ead[label=e3]{j.h.van.zanten@vu.nl
}}
%%%%%%%%%%%%%%%%%%%%%%%%%%%%%%%%%%%%%%%%%%%%%%
%% Addresses                                %%
%%%%%%%%%%%%%%%%%%%%%%%%%%%%%%%%%%%%%%%%%%%%%%
\address[A]{Department of Decision Sciences, Bocconi University,\\
Bocconi Institute for Data Science and Analytics (BIDSA), \printead{e1}}

\address[B]{Delft Institute of Applied Mathematics (DIAM), Delft University of Technology \printead{e2}}
\address[C]{Department of Mathematics, Vrije Universiteit Amsterdam \printead{e3}}
\end{aug}

\begin{abstract}

We derive minimax testing errors in a distributed framework where the data is split over multiple machines and their communication to a central machine is limited to $b$ bits. We investigate both the $d$- and infinite-dimensional signal detection problem under Gaussian white noise. We also derive distributed testing algorithms reaching the theoretical lower bounds. 

Our results show that distributed testing is subject to fundamentally different phenomena that are not observed in distributed estimation. Among our findings, we show that testing protocols that have access to shared randomness can perform strictly better in some regimes than those that do not. We also observe that consistent nonparametric distributed testing is always possible, even with as little as $1$-bit of communication and the corresponding test outperforms the best local test using only the information available at a single local machine. Furthermore, we also derive adaptive nonparametric distributed testing strategies and the corresponding theoretical lower bounds.
\end{abstract}

\begin{keyword}[class=MSC2020]
\kwd[Primary ]{62G10}
\kwd{62F30}
\kwd[; secondary ]{62F03}
\end{keyword}

\begin{keyword}
\kwd{Distributed methods}
\kwd{Nonparametric}
\kwd{Hypothesis testing}
\kwd{Minimax optimal}
\end{keyword}

\end{frontmatter}
%%%%%%%%%%%%%%%%%%%%%%%%%%%%%%%%%%%%%%%%%%%%%%
%% Please use \tableofcontents for articles %%
%% with 50 pages and more                   %%
%%%%%%%%%%%%%%%%%%%%%%%%%%%%%%%%%%%%%%%%%%%%%%
%\tableofcontents

%%%%%%%%%%%%%%%%%%%%%%%%%%%%%%%%%%%%%%%%%%%%%%
%%%% Main text entry area:

\section{Introduction}

% introduce SNWN model, finite dimensional

% name the problem you're studying
Distributed methods are concerned with inference in a framework where the data resides at multiple machines. Such settings occur naturally when data is observed and processed locally, at multiple locations, before sent to a central location where they are aggregated to obtain a final result. By working with smaller sample sizes locally distributed methods can substantially speed up the computation compared to centralized, classical methods. Furthermore, they reduce memory requirements and help protecting privacy by not storing all the information at a single location. For these reasons, the study of distributed methods has attracted significant attention in recent years.

In our analysis we first consider the many normal means model, {which is often used as a platform to investigate more complex statistical problems}. In the classical version of the model one obtains an observation $X$ subject to the dynamics $X = f + n^{-1/2} Z$, where $f \in \R^d$ is an unknown signal, and $Z$ an unobserved, $d$-dimensional standard normal noise vector. This is equivalent to observing $n$ independent copies of a $N_d(f,I_d)$ vector. Our focus is on testing the absence or presence of the signal component $f$ in the model. Rejecting the null hypothesis $H_0:\, f = 0$ means declaring that there is a non-zero signal underlying the observation $X$. The difficulty of distinguishing between the two hypotheses depends on signal strength, the noise ratio $n$ and dimension $d$. It is well known that the signal strength in terms of the Euclidean norm of $f$ needs to be at least of the order $d^{1/4} / \sqrt{n}$ for the hypotheses to be distinguishable, see e.g. \cite{balakrishnan_hypothesis_2019}. 

We study this signal detection problem in a distributed setting. In the distributed version of the above normal-means model, the $n$ observations are divided over $m$ machines (assuming without loss of generality that $n$ is a multiple of $m$). Equivalently, each local machine $j \in \{1,\dots,m\}$ observes
\begin{equation}\label{eq : dynamics}
X^{j}  = f + \sqrt{\frac{m}{n}}Z^{j},
\end{equation}
where $f \in \R^d$ and the noise vectors $Z^{j}$ are independent $d$-dimensional standard normal random vectors. Each machine $j$ transmits a $b$-bit transcript $Y^{j}$ to a central machine. By aggregating these $m$ local transcripts, the central machine computes a test for the hypothesis $H_0 : f = 0$. We derive, for this distributed setting, the order of the minimal signal strength $\rho$ for which the null hypothesis can be distinguished from the alternative $H_1: \|f\|_2 \geq \rho$. In the distributed setting, $\rho$ is considered as a function of the number of machines $m$ and the communication budget $b$, in addition to the dimension $d$ and noise level $n$. We allow all the parameters $b, m$ and $d$ to depend on $n$.

The transcripts generated by the machines may be either deterministic or randomized. When randomizing the transcript, we consider two different possibilities for the source of randomness. In the \emph{private coin} setup, the machines may only use their own local (independent) source of randomness. In the \emph{public coin} setup, the machines have access to a shared source of randomness in addition to their own independent source. This is akin to a situation in which the machines have access to the same random seed. We show that depending on the size of the communication budget, having access to a public coin strictly improves the distinguishability of the null- and alternative hypothesis. 

{Our results indicate that, }in the case where $b$ and $m$ are small relative to the dimension $d$ in an appropriate sense, the one-bit protocols have similar properties, in terms of separation rate, as multi-bits protocols, i.e. one can achieve the minimax optimal $b$-bit testing rates with taking the majority vote of appropriately chosen local (one-bit) test outcomes. {This is a striking difference with estimation, where for small values of $b$, increases in {the communication} budget result in (sometimes exponential) improvements in convergence rate. We find that, as $m$ increases,} the local testing problems become more difficult as the local sample size deceases, but at a certain threshold, this effect is compensated for by the increase in total communication budget $bm$. This threshold occurs when $bm$ exceeds the dimension. At this point, we find that public coin protocols start to strictly outperform private coin protocols, in the sense that smaller signals can be detected with the same amount of  transmitted bits $b$. {This is also a dissimilarity with estimation, where having access to public randomness offers no benefit, {as we show it in our paper}.} When the communication budget $b$ per machine exceeds that of the dimension $d$ of the problem, the minimax rates of the classical, non-distributed setting can be attained. %The derived results are theoretical in nature, but we deem some of the messages to be of relevance to applied settings in terms of designing distributed testing architectures.

%Our lower bound for the classical normal means model confirms a conjecture in \cite{pmlr-v125-acharya20b} for the case where $m=n$. 

We then extend our results for the $d$-dimensional Gaussian model to the nonparametric signal in white noise setting. This latter model is of interest as it serves as benchmark and starting point to investigate more complicated nonparametric models. Here, the local observations for $j=1,\dots,m$ constitute $\int_0^{\cdot} f(s)ds + \sqrt{\frac{m}{n}} W^{j}_{\cdot}$ where the $W^{j}$'s are independent Brownian motions and $f \in L_2[0,1]$ the unknown functional parameter of interest.  
%Classically, the infinite dimensional setting offers an elegant formulation of minimax testing rates. However, we find some of the principle phenomena in the distributed setting to be more transparent in the finite dimensional Gaussian model, in addition to yielding easier proofs from a technical point of view. For this reason, the main results are formulated and proven in the finite dimensional setting and the ramifications for infinite dimensional setting are then subsequently derived. 
Our results for the infinite dimensional model comes in the form of minimax rates for distributed protocols in terms of the strength of the signal in $L_2$-norm, the smoothness $s$ of the signal, the amount of bits $b$ allowed to be communicated by each machine, the signal to noise ratio $n$ and the number of machines $m$. {In contrast to nonparametric distributed estimation, we show that consistent distributed testing is always possible, even when $m$ and $b$ are small. Having a shared source of randomness results in better rates in certain regimes in the nonparametric setting, whilst we show that this is never the case for distributed estimation.} Finally, we consider the more realistic, adaptive setting where the regularity $s$ is considered to be unknown. We show that in contrast to the non-distributed setting where the cost for adaptation is a multiplicative $\log\log n$ factor, in the distributed case a more severe $\log n$ penalty is necessary. We also propose a nonparametric distributed testing procedure based on Bonferroni's correction reaching the theoretical limits (up to a $\log\log n$ factor) and observe additional, {unexpected} phase transitions compared to the non-adaptive setting.

\subsection{Related literature}

Starting a few decades ago, earlier investigations into similar topics originate in the electrical engineering community, under the names ``decentralized decision theory / the CEO problem'' e.g. \cite{tenney_detection_1981, ahlswede_hypothesis_1986, tsitsiklis_decentralized_1988, 394767, kreidl_decentralized_2011, tarighati_decentralized_2017} or ``inference under multiterminal compression'' (see \cite{te_sun_han_statistical_1998} for an overview). Motivated by applications such as surveillance systems and wireless communication, the inference problems are approached from a ``rate-distortion'' angle in this body of literature. However, these results typically consider fixed, finite sample spaces and a fixed number of machines $m$ and investigate asymptotics only in the sample size $n$.

Understanding the fundamental statistical performance of distributed methods in context of non-discrete, higher-dimensional sample spaces has been considered only recently. Most of the literature focused on estimating the parameter/signal of the model in a distributed framework. Minimax lower and (up to a possible logarithmic factor) matching upper bounds were derived for the minimax risk in terms of communication constraints in context of the many normal means and simple parametric problems, see \cite{zhang2013information,duchi_optimality_2014,shamir2014fundamental,braverman2016communication,xuraginsky2016,han2018geometric,
cai_distributed_2020,cai:distributed:adap:sigma}.  These results were extended to nonparametric models, including Gaussian white noise \cite{pmlr-v80-zhu18a}, nonparametric regression  \cite{szabo2020adaptive}, density estimation \cite{barnes2020lower} and general, abstract settings \cite{szabo:zaman:2022}. Distributed techniques for adapting to the unknown regularity of the functional parameter of interest were derived in \cite{szabo2020adaptive,szabo2020distributed,cai2022distributed}.

For distributed testing, much less is known. In \cite{9211418}, the authors consider a setting in which each machine obtains a single observation from a distribution on a finite sample space and derive lower bounds for testing uniformity of this distribution. Similar distributed uniformity testing is considered in \cite{9211522}, where matching upper bounds are exhibited for this setting. 
In \cite{szabo2022optimal}, the authors derive matching upper and lower bounds for the distributed version of the classical many normal means model (see \eqref{eq : dynamics} above) for the case that only the outcome of local tests can be communicated (e.g. $1$-bit of communication). In \cite{pmlr-v125-acharya20b} less stringent communication requirements are considered, in the special case of the model in \eqref{eq : dynamics} above with $m=n$. Questions regarding nonparametric models and adaptation in the setting of distributed testing have remained completely open thus far.

{To summarize the state of the art, the lower bounds derived in the literature so far are only optimal in case of constant communication budget in the public coin setting, i.e. $b=O(1)$. So far no lower bound results are available in the public coin setting if $b$ can tend to infinity as $n$ increases. Furthermore, there is a lack of any lower bound result in the private coin setup. The traditional methods based on
mutual information and Taylor expansion as considered in \cite{szabo2022optimal} and \cite{pmlr-v125-acharya20b}, respectively, do not extend to the setting of multiple bits or private
coin protocols. In this article we fill this gap and derive the first rigorous minimax lower bounds for distributed testing procedures in the normal means model for arbitrary communication budget $b$ both for private and public coin settings. In order to prove the lower bounds, we provide a novel Bayesian testing argument based on a Brascamp-Lieb type inequality with distributed version of testing lower bounding techniques.} 

The upper bounds derived in \cite{pmlr-v125-acharya20b} are more complete for both the private and public coin settings and go beyond the above described restrictive setting in which the lower bounds were derived, but do not cover all possible cases. {For instance, in \cite{pmlr-v125-acharya20b} it is assumed that the separation distance between the null and alternative hypotheses is bounded from above by one, which does not cover the case $\sqrt{dm}\gg n$. Also, only the $m=n$ case was considered in the preceding paper. Therefore, in certain regimes new testing procedures and proof techniques had to be derived for full treatment of the problem (e.g. our novel test $T_{\text{III}}$ in the high-budget private coin case, see Section \ref{ssec : private coin high-budget}).}

{The literature on distributed testing has so far solely focused on finite dimensional models. We provide the first results for distributed testing in nonparametric models. Besides deriving lower and matching upper bounds we also derive an adaptive testing procedure, not depending on the typically unknown regularity of the underlying functional parameter of interest.}

{\subsection{Overview of our results and organization}

For a quick overview, the main contributions of this article are:
\begin{itemize}
    \item Sharp minimax upper- and lower bounds for all values of $n,m,d,b$ for the $d$-dimensional distributed-signal-in-white-noise model, for both private and public coin settings (Section \ref{sec: main}), with accompanying methods achieving these rates (Theorem \ref{thm : detection lb} and Theorem \ref{thm : detection ub}).
    \item We extend the $d$-dimensional distributed-signal-in-white-noise model to the nonparametric setting where the signal is a Sobolev regular functional parameter of known regularity and establish the minimax rates within this setting for all values of $n,m,b$ for both the private and public coin settings (Theorem \ref{thm : nonparametric SNWN minimax rate}). 
    \item We consider the nonparametric setting in which the regularity of underlying signal is unknown and derive adaptive private and public coin procedures. Furthermore, we establish private and public coin lower bounds for the adaptive setting that are tight up to a $\log \log n$ factor for all values of $n,m,b$ (Theorem \ref{thm:adapt1} and Theorem \ref{thm:adapt2}).
\end{itemize}

The remainder of the paper is organized as follows. In Section \ref{sec:model} we describe the distributed-signal-in-white-noise model with $d$-dimensional signal $f\in\mathbb{R}^d$ and formalize the distributed testing problem both for private and public coin protocols. In Section \ref{sec: main} we provide the minimax lower and matching upper bounds for both testing protocols. Section \ref{sec : sketch lb} gives a sketch of the proof of the lower bound. We exhibit constructive algorithms that achieve matching upper bounds in Section \ref{sec : upper bounds}. We extend our results to the nonparametric distributed-signal-in-white-noise model with Sobolev regular functional parameter in Section \ref{sec : nonparametric part}. Here, we also compare distributed testing and estimation rates and highlight the similarities and differences between them both in the private and public coin settings. In Section \ref{sec : adaptation}, we consider adaptation to the unknown regularity level in the nonparametric setting and present theoretical lower and matching upper bounds. In Section \ref{sec: proof:adapt}, we derive constructive algorithms achieving these upper bounds. The detailed proof of the lower bound for the $d$-dimensional signal is deferred to Section \ref{sec : lower bound} and a key technical lemma described in Section \ref{sec : proof gaussian kernel lemma}. Detailed proofs for this lemma, as well as some of the technical details of the other main results and various auxilliary results, have been deferred to the Supplementary Material \cite{szabo2022nonparametric_supplement} to this manuscript.  Results, equations and sections in the Supplementary Material are indexed by capital letters as opposed to numerals, as is used in the article.}

%The principle phenomena we encounter in distributed testing are radically differing from those typically found in distributed estimation under communication constraints. This is perhaps not too suprising. Firstly, the goal in distributed testing is essentially a binary output, whereas estimation of a high dimensional parameter should sensibly require a large amount of communication. 

 \subsection{Notation}
We write $a \wedge b = \min\{a, b\}$ and $a \vee b = \max\{a, b\}$. For two positive sequences $a_n,b_n$ we use the notation $a_n\lesssim b_n$ if there exists a universal positive constant $C$ such that $a_n\leq C b_n$. We write $a_n\asymp b_n$ which holds if $a_n\lesssim b_n$ and $b_n\lesssim a_n$ are satisfied simultaneously. We shall use $a_n \gg b_n$ to denote $b_n / a_n \to 0$. The Euclidean norm of a vector $v \in \mathbb{R}^d$ is denoted by $\|\cdot\|_2$. For absolutely continuous probability measures $P\ll Q$, we denote by $D_{KL}(P\| Q)=\int \log\frac{dP}{dQ}dP$ and $D_{\chi^2}(P\| Q) = \int (\frac{dP}{dQ}-1)^2dP$ their Kullback-Leibler and Chi-square divergences, respectively. Throughout the whole text we use for convenience the abbreviation rhs and lhs for right-hand-side and left-hand-side, respectively, and cdf for the cumulative distribution function.

\section{Problem formulation and setting}\label{sec:model}

We consider testing in the distributed-signal-in-white-noise model. In this section, we provide the formulation of the distributed setup for data coming from the finite dimensional model. Except for obvious modifications to the sample space, the same setup is considered when the local data is from the infinite dimensional distributed-signal-in-white-noise model, which is formulated in Section \ref{sec : nonparametric part}. For $j=1,\dots,m$ machines, the local observations constitute $X^{j}$ taking values in $\cX\subset\R^d$, subject to {dynamics \eqref{eq : dynamics} under $P_f$.} Each machine $j$ communicates a $b$-bit transcript $Y^{j}$ to a central machine. That is, the transcript $Y^{j}$ takes values in some space $\cY^{j}$ with $|\cY^{j}| \leq 2^{b}$ for $b \in \N$. Let $Y = \left(Y^{1},\dots,Y^{m}\right)$ denote the aggregated data in the central machine. The central machine computes a test $T(Y)$, where $T$ is a map from $\cY :=\bigotimes_{j=1}^m \cY^{j}$ to $\{0,1\}$ that has to distinguish between the null hypothesis $f = 0$ and the alternative hypothesis. As alternative hypothesis, we consider whether
\begin{equation*}
f \in H_{\rho}  :=  \left\{ f\in \R^d : \|f\|_2 \geq \rho \right\}, %\; \text{ or } \; H_{\rho,\infty} :=  \left\{ f\in \R^d : \|f\|_\infty \geq \rho \right\}.
\end{equation*}
for some appropriately chosen $\rho= \rho_{m,n,d,b}$.

We distinguish two mechanisms through which the local machines $j=1,\dots,m$ can generate their transcripts $Y^{j}$. In the first setup, machines can use only their local observation $X^{j}$ when generating $Y^{j}$, possibly in combination with a local source of randomness. In the second setup, we allow the machines to access a common source of randomness $U$, which is independent of the data $X := ( X^{1},\ldots,X^{m})$. In the latter setup, which we call the public coin setting, the machines may use both local randomness, the observed draw of $U$ and their local observation $X^{j}$ when generating their transcript $Y^{j}$. The setup where only local randomness is available shall be referred to as the private coin setting. A formal definition of these two setups is as follows.
 \begin{itemize}
    \item A \emph{private coin distributed testing protocol} consists of a map $T : \cY \to \{0,1\}$ and a collection of Markov kernels $K^j : 2^{\cY^{j}} \times \cX^{j} \to [0,1]$, $j=1,\dots,m$, and the transcript satisfies $Y^{j} | X^{j} \sim K^j(\cdot | X^{j})$.
    \item A \emph{public coin distributed testing protocol} consists of a map $T : \cY \to \{0,1\}$, a random variable $U$ taking values in a probability space $(\cU, \mathscr{U},\P^{U})$ and a collection of Markov kernels $K^j : 2^{\cY^{j}} \times \cX^{j} \times \cU \to [0,1]$, $j=1,\dots,m$, such that $Y^{j} | (X^{j},U) \sim K^j(\cdot | X^{j},U)$.
 \end{itemize}
The choices for the kernels induce the conditional distribution of $Y = (Y^{1},\dots,Y^{m})$, which we will denote $K := {\bigotimes}^m_{j=1}K^j$. For the joint distribution of $X$, $Y$ and $U$ we shall write $\P_{f,K} \equiv \P_f$, where the $f$ subscript indicates the dynamics underlying $X$ and the subscript $K$ is used to stress that the conditional distribution of $Y$ induced by the choice of kernels. {Furthermore, we denote by $\mathbb{P}_f^{X}$ the corresponding marginal distribution of X, i.e. $\mathbb{P}_f^{X}=P_f$.} Our distributed architecture in the public coin case then follows the following Markov chain structure at each local machine $j=1,...,m$
\begin{equation} \label{eq : markov chain structure}
 \begin{matrix}
  &\qquad  &\qquad  U &\qquad \put(0,3){\vector(3,-1){20}} \qquad& \\
  &\qquad   &\qquad &\qquad  &\quad Y^{j}. \\
f  &\qquad \put(0,3){\vector(1,0){20}} &\qquad X^{j} &\qquad \put(0,3){\vector(3,1){20}}   \qquad&
\end{matrix} 
\end{equation}

Note that any private coin testing protocol can effectively be considered a public coin testing protocol for which $U$ has degenerate distribution, i.e. $U = u \in \cU$ almost surely. In our proofs below, for the sake of compactness, we consider without loss of generality that the private coin setting implies $U$ has a degenerate distribution. When no confusion can arise, we will refer to a distributed testing protocol as ``distributed test'', and we will refer to the tuple $(T, \{ K^1,\dots,K^m \}, \P^U)$ by $T$ for ease of notation. We use $\cT_{priv}(b)$ and $\cT_{pub}(b)$ to denote the classes of all private and public coin distributed tests, respectively, each with communication budget $b$ per machine. 

We define the testing risk of a distributed test $T \equiv (T, K, \P^U)$ for the alternative hypothesis $H_\rho$ as the sum of the Type I and Type II errors, i.e.
\begin{equation}\label{eq : def testing risk}
\cR(H_\rho , T) := \P_0 \left( T(Y) = 1 \right) + \underset{f \in H_\rho}{\sup} \; \P_f \left(T(Y) = 0\right). 
\end{equation}
% rho as a function of budget, d, n m

% introduce density testing model??

% NOTATION SECTION

% I_d as d \times d indentity matrix

% equivalences

\section{Minimax upper and lower bounds in the normal means model}\label{sec: main}

Our main results come in the form of two theorems. The first establishes the lower bounds for the detection threshold for both the public- and private coin distributed tests. We provide the proof of this theorem in Section \ref{sec : lower bound}. The second theorem establishes the optimality of the lower bound posed in the first theorem by providing distributed tests in both the public and private coin cases which attain the respective rates posed by the lower bounds. These optimal distributed testing procedures are described in Section \ref{sec : upper bounds}. We note that our results are not asymptotic in nature as they hold for every combination of $b,n,m$ and $d$, hence going beyond the classical parametric framework.

\begin{theorem}\label{thm : detection lb}[Distributed testing lower bound]
 For each $\alpha \in (0,1)$ there exists a constant $c_\alpha > 0$ (depending only on $\alpha$) such that if
\begin{equation}\label{eq : pub coin rho lb}
\rho^2 < c_\alpha \frac{\sqrt{d}}{n} \left( \sqrt{\frac{d }{b \wedge d}} \bigwedge \sqrt{m}  \right),
\end{equation}
then in the public coin protocol case
\begin{equation*}
\underset{T \in \cT_{pub}(b)}{\inf} \; \; \cR(H_{\rho} , T) > \alpha \; \text{ for all } \; n,m,d,b \in \N.
\end{equation*}
Similarly, for
\begin{equation}\label{eq : priv coin rho lb}
\rho^2 < c_\alpha \frac{\sqrt{d}}{n} \left( \frac{d}{ b \wedge d} \bigwedge \sqrt{m}  \right),
\end{equation}
we have under the private coin protocol that
\begin{equation*}
\underset{T \in \cT_{priv}(b)}{\inf} \; \; \cR(H_{\rho} , T) > \alpha \; \text{ for all } \; n,m,d,b \in \N.
\end{equation*}
\end{theorem}

The approach to proving the lower bound theorem can be summarized as follows. We start out by lower bounding the testing risk by a type of Bayes risk, where the parameter $f$ is drawn from {an adversarial} prior distribution $\pi$. By taking $\pi$ to be Gaussian, we can exploit the conjugacy of the model in order to show that optimal transcripts are either invariant to the prior or ``Gaussian'' in an appropriate sense. After this, the results follow by data processing arguments that are geometric in nature. {We defer a more elaborate sketch of the proof to Section \ref{sec : sketch lb} and the detailed proof to Section \ref{sec : lower bound}. The techniques used in this work are novel and drastically different than those used in \cite{pmlr-v125-acharya20b, szabo2022optimal}, {which provide tight bounds only} in the $1$-bit case.}

The above theorem implies that if \eqref{eq : pub coin rho lb} holds, no consistent public coin distributed testing protocol with communication budget $b$ bits per machine exists for the hypotheses $H_0: f = 0$ versus the alternative $H_1: \|f\|_2 \geq \rho$. In other words, no public coin distributed test manages to consistently distinguish all signals from $0$ if the signals are smaller than the rhs of \eqref{eq : pub coin rho lb}. When considering only private coin distributed testing protocols, the detection threshold \eqref{eq : priv coin rho lb} is more stringent than the public coin threshold \eqref{eq : pub coin rho lb} for certain values of $d$, $m$ and $b$. Theorem \ref{thm : detection ub} below affirms that, in these cases, the best private coin protocol have a strictly worse performance compared to the best public coin protocol.

\begin{theorem}\label{thm : detection ub}
 For each $\alpha \in (0,1)$ there exists a constant $C_\alpha > 0$ (depending only on $\alpha$) such that if
\begin{equation*}
\rho^2 \geq C_\alpha \frac{\sqrt{d}}{n} \left( \sqrt{\frac{d }{b \wedge d}} \bigwedge \sqrt{m}  \right),
\end{equation*}
there exists $T \in \cT_{pub}(b)$ such that 
\begin{equation*}
 \cR(H_{\rho} , T) \leq \alpha \; \text{ for all } \; n,m,d,b \in \N.
\end{equation*}
Similarly, for
\begin{equation*}
\rho^2 \geq  C_\alpha \frac{\sqrt{d}}{n} \left( \frac{d }{ b \wedge d} \bigwedge \sqrt{m}  \right)
\end{equation*}
there exists $T \in \cT_{priv}(b)$ such that 
\begin{equation*}
\cR(H_{\rho} , T) \leq \alpha \; \text{ for all } \; n,m,d,b \in \N.
\end{equation*}
\end{theorem}

The achievability of arbitrarily small testing risk is shown using a constructive proof, see Section \ref{sec : upper bounds}. That is, we derive distributed testing protocols that distinguish the null hypothesis from any $f \in \R^d$ in the alternative class. 

Theorem \ref{thm : detection lb} together with Theorem \ref{thm : detection ub} establish the minimax distributed testing rate for public and private coin protocols. As a sanity check, note that when $m = 1$, we obtain the non-distributed minimax testing rate $\rho^2=\sqrt{d}/n$. Furthermore, when $b\gtrsim d$, enough information about the coefficients can be communicated to obtain the non-distributed minimax rate also, for both the public coin and private coin distributed protocols. When the communication budget is smaller than the dimension ($b = o(d)$), the class of public coin protocols starts to exhibit strictly better performance than the private coin ones in scenarios as long as $d = o(mb)$. That is, as long as the total communication budget $mb$ of the system exceeds the dimension $d$ of the parameter, public coin protocols achieve a strictly better rate than private coin ones. This remarkable phenomenon disappears when the dimension is larger than the total communication budget (i.e. $mb = o(d)$), at which point there exists a one-bit private coin protocol achieving the optimal rate of $\rho^2 \asymp \frac{\sqrt{md}}{n}$ in both cases. {Consistent distributed testing turns out to be possible even for small values of $b$ and $m$, as long as $n$ is large enough compared to $d$. This stands in contrast to estimation in the $d$-dimensional Gaussian mean model, where consistent estimation is not possible when $mb = o(d)$, regardless of sample size $n$ (see e.g. \cite{
cai_distributed_2020}). Furthermore, as long as $mb = o(d)$ in the public coin case or $mb^2 = o(d^2)$ in the private coin case, an increase in communication budget does not lead to a better rate. This stands in stark contrast to estimation, where for small budgets an increase can lead to an exponential improvement in convergence rate.}

\section{Distributed testing protocols achieving the lower bound in the many normal means model}\label{sec : upper bounds}

In this section, we exhibit three distributed testing procedures achieving the rates posed by the lower bound. The first distributed testing procedure $T_{\text{I}}$ communicates only a single bit per machine and can detect signals with a squared Euclidean norm of larger or equal order than $\frac{\sqrt{dm}}{n}$ and does not need a public coin. As a second procedure, we consider a test satisfying the public coin protocol $T_{\text{II}}$ that achieves the rate $\frac{d }{n \sqrt{b\wedge d}}$. The third procedure satisfies the private coin protocol and achieves the corresponding slower rate $\frac{d }{n (b\wedge d)}$. Note that, depending on the values of $n,m,d$ and $b$, the existence of such distributed testing protocols proves Theorem \ref{thm : detection ub} and implies that the lower bounds in Theorem \ref{thm : detection lb} are in fact tight.

A common denominator in the construction of the three protocols is that the transcripts $Y^{j}$ are generated as vector of $p_f^{j}$-Bernoulli random variables taking values in $\{0,1\}^b$ where $p_f^{j} \in [0,1]^b$ depends on the underlying signal $f$, with $p_f^{j} = (1/2,\dots,1/2)$ under the null hypothesis {($f = 0$)}. {The concentration inequality for groups of Bernoulli random variables given in Lemma \ref{lem : binomial testing} provides a recipe for the construction of a central test for each of the three regimes.} The Type I error can be controlled since the distribution under the null hypothesis is known. The Type II error is small  whenever the vectors of probabilities $p_f^{1},\dots,p_f^{m}$ are sufficiently separated from $(1/2,\dots,1/2)$ in Euclidean norm. 

\begin{lemma}\label{lem : binomial testing}
Consider for $k,l \in \N$, $l \geq 2$, independent random variables $\{B^j_i : i=1,\dots,k,\;j=1,\dots,l\}$ with $B^j_i \sim \text{Ber}(p_i)$. If $p_i = 1/2$ for $i=1,\dots,k$, it holds that for all $\alpha \in (0,1)$ there exists $\kappa_\alpha > 0$ {such that
\begin{equation*}
T := \mathbbm{1} \left\{ \bigg| \frac{1}{\sqrt{k} l} \underset{i=1}{\overset{k}{\sum}}  \left( \underset{j=1}{\overset{l}{\sum}} (B_i^j - \frac{1}{2} ) \right)^2 - \sqrt{k}/4 \bigg| \geq \kappa_\alpha \right\} 
\end{equation*}
satisfies $\E T \leq \alpha/2$. On the otherhand, if 
\begin{equation}\label{eq : binomial l2 divergence lb}
\eta_{p,l,k}:= \frac{l-1}{2\sqrt{k}} \underset{i=1}{\overset{k}{\sum}} \left( p_i - \frac{1}{2} \right)^2 \geq \kappa_\alpha,
\end{equation}
 it holds that
\begin{equation}\label{eq : binomial lem type II bound}
\E (1 - T) \leq  \frac{1/2+16\eta_{p,l,k}/\sqrt{k}}{\eta_{p,l,k}^2} .
\end{equation}}
\end{lemma}

The proof of the lemma can be found in Section A.2 of the Supplementary Material where it is restated as Lemma \ref{lem : binomial testing supplement}. We also provide a version of this lemma (Lemma \ref{lem : binomial testing supplement_version2} in the Supplement) used in the high-budget private coin protocol case.

\subsection{Low communication budget: construction of $T_{\text{I}}$}\label{ssec : private coin low-budget} The protocol presented here is similar to the one given in \cite{szabo2022optimal}, with some adjustment allowing the application of Lemma \ref{lem : binomial testing} for a simpler proof. 

We first compute the local test statistic $S_{\text{I}}^{j} = (n/m)\|X^{j}\|_2^2$ at every machine $j=1,...,m$. Under the null hypothesis, $S_{\text{I}}^{j}$ follows a chi-square distribution with $d$ degrees of freedom, i.e. $S_{\text{I}}^{j}\sim \chi^2_d$. Letting $F_{\chi^2_d}$ denote $\chi^2_d$-cdf, the quantity $F_{\chi^2_d} \left(S_{\text{I}}^{j} \right)$ can be seen as the p-value for the local test statistic $S_{\text{I}}^{j}$. Based on these ``local p-values'', we then generate the randomized transcript $Y_{\text{I}}^{j}$ for every $j$ using Bernoulli random variables:
\begin{equation*}\label{eq : low budget transcript generation}
Y_{\text{I}}^{j} | S_{\text{I}}^{j} \sim \text{Ber}\left(F_{\chi^2_d} \left(S_{\text{I}}^{j} \right)\right).
\end{equation*}
For a given $\alpha \in (0,1)$, we can construct the test
\begin{equation}\label{eq : small budget test}
T_{\text{I}} = \mathbbm{1} \Big\{  \Big|\frac{1}{m} \Big( \underset{j=1}{\overset{m}{\sum}} (Y_{\text{I}}^{j}- 1/2) \Big)^2 - {1/4}\Big| \geq  {\kappa}_\alpha  \Big \}
\end{equation}
at the central machine. In applications, one could set for instance $\kappa_\alpha$ such that $\P_0 T_{\text{I}} \approx \alpha$ by considering that $\sum_{j=1}^{m} Y_I^{j}$ is $(m,1/2)$-binomially distributed under the null. {Lemma \ref{lem : T_I proof} in the Supplementary Material yields that for each $\alpha \in (0,1)$, there exist constants $\kappa_\alpha, C_\alpha, M_\alpha, D_0 >0$ such that for $m \geq M_\alpha$ and $d \geq D_0$ it holds that $\cR(H_\rho,T_{\text{I}}) \leq \alpha$, whenever $\rho^2 \geq C_\alpha \frac{\sqrt{md}}{n}$.}

The case $m \leq M_\alpha$ corresponds essentially to the non-distributed setting and is treated separately for technical reasons. In practice, one would simply use the test given in \eqref{eq : small budget test} also for $m \leq M_\alpha$. Furthermore, if one allows for a slightly larger amount of bits (e.g. $\log_2(n)$ bits), one could opt to transmit an (approximation of) the test statistics $S_{\text{I}}^{j}$ themselves, see e.g. Lemma 2.3 in \cite{szabo2020distributed}, for which it is easy to prove that the rate of $\frac{\sqrt{md}}{n}$ is {achieved  without} requiring any assumptions on $m$. For the sake completeness: by considering $\rho^2 \geq C_\alpha \sqrt{M_\alpha} \frac{\sqrt{d}}{n}$, we see that the optimal rate of $\frac{\sqrt{md}}{n}$ can be achieved in the $m \leq M_\alpha$ case by simply taking
\begin{equation}\label{eq : standard local chi-square test}
T_{\text{I}}' := {Y_{\text{I}}^{1}} := {\mathbbm{1} \left \{ \frac{1}{\sqrt{d}} \left(S_{\text{I}}^{1} - d \right) \geq \kappa_\alpha \right\} }
\end{equation}
for an appropriately large choice of the constant $\kappa_\alpha$. Similarly, the requirement that $d$ is larger than some constant $D_0$ (which is independent of $\alpha$) appears for technical reasons. The case where $d \leq D_0$ is covered by the private coin protocol $T_{\text{III}}$ in Section \ref{ssec : private coin high-budget}.

\subsection{Public coin, high communication budget: construction of $T_{\text{II}}$}\label{ssec : public coin test} We now switch our attention to exhibiting a testing procedure that is optimal when $bm \gtrsim d$ in the public coin case. {That a shared source of randomness in distributed settings can be strictly better than private ones in terms of \emph{communication complexity}, is an idea that goes back to \cite{yao_distributed_shared_randomness}. Essentially, the use of shared randomness allows for the machines coordinate their efforts in ``covering'' each of the $d$ dimensions of the data even though all communication happens in just one round. See also e.g. Chapter 3 in \cite{rao2020communication} for an extensive treatment of this phenomenon.} We adopt ideas proposed by \cite{pmlr-v125-acharya20b}, who consider the setting where $m = n$ with asymptotics in $m$. We exhibit this testing protocol below and provide a full proof covering also the case where $m \neq n$. To that extend, let $U$ be a random rotation, i.e. $U$ is drawn from the Haar measure (see e.g. Theorem F.13 in \cite{anderson_guionnet_zeitouni_2009}) on the set of orthonormal matrices in $\R^{d \times d}$. At each machine, for $b \leq d$, we can compute the $b$-bit transcript $Y^{j}_{\text{II}} \in \{0,1\}^b $ conditionally on the shared public coin draw $U$, where each of the $1 \leq i \leq b$ components is defined through 
\begin{equation*}
(Y^{j}_{\text{II}})_i|U,X^{j} = \mathbbm{1} \left \{ \left( \sqrt{n/m} U X^{j} \right)_i > 0 \right\},
\end{equation*}
where $\left( v \right)_i$ denotes the projection onto the $i$-th coordinate of the vector $v \in \R^d$. {The random rotation fulfills a similar purpose as the random reweighting algorithm proposed in \cite{szabo2022optimal}, but leads to an easier proof in the $d$-dimensional case because of rotational invariance of the Gaussian distribution.} 

Centrally, after transmitting $(Y^{1},\dots,Y^{m})$, we compute the aggregated test statistics  $S_{\text{II}} = \sum_{j=1}^{m} Y^{j}_{\text{II}}$ and define the corresponding test as
\begin{equation}\label{eq : public coin dist test}
T_{\text{II}} = \mathbbm{1} \Big\{ \Big|\frac{1}{ \sqrt{b} m} \underset{i=1}{\overset{b}{\sum}} \Big( (S_{\text{II}})_i - \frac{m}{2} \Big)^2 - \sqrt{b}/4\Big|  > \kappa_\alpha  \Big\}.
\end{equation}
{Lemma \ref{lem : public coin test risk ub} in the Supplementary Material shows that this test achieves the public coin lower bound when $mb \gtrsim d$ and $m\geq M_{\alpha}$.}

\subsection{Private coin, high total communication budget: constructing $T_{\text{III}}$}\label{ssec : private coin high-budget} Finally, we consider the case of not having access to a public coin, but having a relatively large communication budget ($b^2 m \gtrsim d^2$).  Note that we can assume without loss of generality that $m \geq M_\alpha {d^2}/{b^2}$ for a constant $M_\alpha>0$,  as otherwise the optimal rate is $\sqrt{md}/n$, obtained by the $1$-bit private coin test described by \eqref{eq : small budget test} (see Section \ref{ssec : private coin low-budget}). This case is the most involved one and we construct a test consisting two sub-tests optimal in different sub-regimes.

The most obvious approach in this case is to divide the communication budget of each machine over the $d$ coordinates as uniformly as possible. That is to say, to partition the coordinates $\{ 1,\ldots,d \}$ into approximately $d/b$ sets of size $b$ (we assume without loss of generality that $b\leq d$, as we can always throw away excess budget and $b = d$ bits suffices for achieving the minimax rate). The machines are then  equally divided over each of these partitions and communicate the coefficients corresponding to their partition. More formally, such a strategy entails taking sets $\cI_i \subset \{1,\dots,m\}$ such that $| \cI_i | = \lfloor \frac{mb}{d} \rfloor$ and each $j \in \{1,\dots,m\}$ is in $\cI_i$ for $b$ different indexes $i \in \{1,\dots,d\}$. For $i = 1,\dots,d$ and $j \in \cI_i$, generate the transcripts according to
\begin{equation}\label{eq : private coin high-budget transcript}
Y^{j}_i|X_i^{j}  = \mathbbm{1} \{ X_i^{j} > 0 \}.
\end{equation}
Centrally, a natural test based on these transcripts is
\begin{equation}\label{eq : large budget priv coin test I def}
T_{\text{III}}^{1} :=\mathbbm{1} \Big\{ \Big|\frac{1}{{| \cI_1 |}\sqrt{d}} \underset{i=1}{\overset{d}{\sum}} \Big(  \underset{ j \in \cI_i}{\overset{}{\sum}} (Y_i^{j} - 1/2) \Big)^2 - \sqrt{d}/4 \Big| > \kappa_\alpha\Big\}.
\end{equation}

It turns out that such a test does not cover all regimes where $m \gtrsim {d^2}/{b^2}$, because, there is a certain amount of information loss due to the nonlinearity of the quantization step \eqref{eq : private coin high-budget transcript}, i.e. the test induces soft thresholding for the signal components which is sub-optimal for (relatively) large signal components. For the exact statement on the testing error of this test, see Lemma \ref{lem : private coin test risk ub I} below. 

For detecting signals including large coordinates we propose an adaptation of test $T_{\text{III}}^1$. We start by assuming that $b\geq 2\log (d+1)$ otherwise we do not construct the test. Then for $i=1,...,d$ and $j=1,...,m$, let us generate
\begin{equation*}
B^{j}_{li} \stackrel{iid}{\sim} \text{Ber}\Big(F_{\chi^2_1} \Big( \big(\sqrt{n/m}X_i^{j} \big)^2 \Big)\Big), \qquad l\in\{1,\dots,C_{b,d}=\lfloor 2^b/(d+1) \rfloor\}. 
\end{equation*}
Note that $C_{b,d}\geq 1$ by assumption. Then machine $j$ communicate the transcripts 
\begin{equation}\label{def:Nj}
N^{j}= \sum_{l=1}^{C_{b,d}}\underset{i=1}{\overset{d}{\sum}} B^{j}_{li} \in \{0,1, \dots, C_{b,d}d\},
\end{equation}
which can be done using $\log_2(C_{b,d}d+1)\leq b$ bits in total. Based on these transcripts, we compute the test
\begin{equation*}
T_{\text{III}}^2 = \mathbbm{1} \left\{ \bigg| \frac{1}{dmC_{b,d}} \left(\underset{j=1}{\overset{m}{\sum}}  (N^{j}- Ld/2) \right)^2 - \frac{1}{4} \bigg| \geq \kappa_\alpha \right\}
\end{equation*}
 centrally. The testing risk bound for the above test is given in Lemma \ref{lem : priv coin high-budget additional test II} below.

Finally, we construct our test by combining the above ones. We construct both partial tests $T_{\text{III}}^1$ and $T_{\text{III}}^2$ if $b\geq 2 \log (d+1)$ by transmitting $b'=\lfloor b/2\rfloor$ bits per machine for each, otherwise we just construct $T_{\text{III}}^1$. Then we merge them by taking
\begin{equation}
T_{\text{III}} = T_{\text{III}}^{1} \vee T_{\text{III}}^2\mathbbm{1}_{\{b \geq 2\log(d+1)\}} ,\label{def:test3}
\end{equation}
where the indicator should be understood to rule out cases in which the transcripts for $T_{\text{III}}^2$ cannot necessarily be communicated. This case, as shown below, is covered by the first test $T_{\text{III}}^1$. {Lemma \ref{lem: private:large:budget} in the Supplementary Material shows that $T_{\text{III}}$ has sufficiently small testing risk in all cases where $m\geq M_{\alpha} d^2/b^2$.}

{\section{A sketch of proof for the testing lower bound (Theorem \ref{thm : detection lb})}\label{sec : sketch lb}

In this section, we provide a sketch of proof of Theorem \ref{thm : detection lb}, of which the full details are given in Section \ref{sec : lower bound}. The proof starts out the same way for both the private and public coin cases, but bifurcates later on. We consider for the time being a generic collection of $b$-bit distributed testing protocols $\cT(b)$.

As a first step, we introduce a prior distribution $\pi$ on $\R^d$ and lower bound the testing risk by a type of Bayes risk and the mass of $\pi$ that resides outside of the alternative hypothesis $H_\rho$, akin to e.g. \cite{ingster_nonparametric_2003}. Recall that $\P_f$ denotes the joint distribution of $Y$, $U$ and $X$ where $X^{j}$ follows \eqref{eq : dynamics} and $Y \sim \E_f^{X,U} K(\cdot | X,U) =: \P_{f,K}^Y=\P_{f}^Y$. For $\pi$ a given a distribution on $\R^d$, define the mixture distribution $\P_\pi^X = P_\pi$ on $\R^{md}$ by $P_\pi(A) = \int P_f(A) d\pi(f)$, {where we recall the notational convention $\mathbb{P}_{f}^{X}=P_f$ from Section \ref{sec:model}}.

Through the Markov chain relation $f \to X \to Y$ this defines a distribution $\P^Y_\pi=\P^Y_{\pi,K}$ on $\cY$ and let us denote by $\E^Y_\pi$ the corresponding expectation. Lemma \ref{lem : dist le cam bound} in the Supplementary Material lower bounds the infimum testing risk $\inf_{T \in \cT} \mathcal{R}(H_\rho,T)$ using a version of Le Cam's lemma adapted to the distributed setting. The lemma yields that, for any distribution on $U$, 
\begin{equation*}
\underset{T \in \cT}{\inf}  \Big( \E_{0}^Y T(Y)  + \underset{f \in H_\rho}{\sup} \, \E_{f}^Y ( 1 - T(Y)) \Big) \geq \underset{K}{\inf} \, \underset{\pi}{\sup}  \Big( 1 - \|\P_{0,K}^Y - \P_{\pi,K}^Y \|_{TV} - \pi(H_\rho^c) \Big),
\end{equation*}
where the infimum on the rhs is over all kernels on $\cY$.

Using that the measure $d\P_f^Y$ disintegrates as $d\P_f^{Y|U=u}d \P^{U}_f(u)$, and the fact that $U$ is independent of the prior $\pi$, we find by Jensen's inequality that
\begin{equation*}
 \|\P_{0,K}^Y - \P_{\pi,K}^Y \|_{TV} \leq \int \|\P_{0,K}^{Y|U=u} - \P_{\pi,K}^{Y|U=u}  \|_{TV} d\P^U(u).
\end{equation*}
By Pinsker's second inequality and the fact that $\log(x) \leq x - 1$, we obtain that
\begin{equation}\label{eq : sup inf public coin divergence lb}
\underset{T \in \cT(b)}{\inf} \cR(H_\rho,T) \geq 1 - \sup_{K } \inf_\pi   \Big( \int \sqrt{2D_{\chi^2}(\P^{Y|U=u}_{0,K} ; \P^{Y|U=u}_{\pi,K}) }d\P^U(u) + \pi(H_\rho^c)  \Big),
\end{equation}
where 
\begin{equation}\label{eq : protocol induced chi-sq div}
D_{\chi^2}(\P^{Y|U=u}_{0,K}  ; \P^{Y|U=u}_{\pi,K}) = \E_{0,K}^{Y|U=u} \left( \frac{\P_{\pi,K}^{Y|U=u}}{\P_{0,K}^{Y|U=u}}(Y) \right)^2 - 1.
\end{equation}}

{From hereon, the proof can be broken down into two steps. We provide the skeleton of the proof here and defer the full details to Section \ref{sec : lower bound}.}
\begin{enumerate}
    \item The first term on the rhs of \eqref{eq : protocol induced chi-sq div} can be expressed in terms of a conditional expectation of the likelihood of $X$:
    \begin{equation}\label{eq : quantity pre factorization}
    \E_0^{Y|U=u} \left( \E_0 \left[ \int \underset{j=1}{\overset{m}{\Pi}} \frac{d\P^{X^{j}}_f}{d \P^{X^{j}}_0} ( X^{j}) d\pi(f) \bigg| Y, U=u \right]^2 \right),
    \end{equation}
    which we compare to the quantity
    \begin{equation}\label{eq : quantity post factorization}
     \underset{j=1}{\overset{m}{\Pi}}\E_0^{Y^{j}|U=u} \left( \E_0 \left[  \int \frac{d\P^{X^{j}}_f}{d \P^{X^{j}}_0} ( X^{j}) d\pi(f) \bigg| Y^{j}, U=u \right]^2 \right),
    \end{equation}
which corresponds to the product of the first terms of the local chi-square divergences. {In particular, we compare the ratio of the expressions in the above two displays and show that when $\pi$ is taken to be centered Gaussian, this ratio is maximized when the protocol's kernel $K: L_2(\cY) \to L_2(\cX)$ with Hilbert space adjoint $K^*$ satisfies that $K^*K: L_2(\cX) \to L_2(\cX)$ is Gaussian in an appropriate sense. This is the content of Lemma \ref{lem : key gaussian kernel maximizer}, which forms the crux of our proof. This lemma, on which we expound in Section \ref{sec : proof gaussian kernel lemma}, exploits the conjugacy between the prior and the model which enables the use of techniques applied in \cite{lieb_gaussian_1990}. Consequently, we obtain that the first term on the rhs of \eqref{eq : protocol induced chi-sq div} is bounded from above by a multiple of
     \begin{equation}% there was a typo here
         \underset{j=1}{\overset{m}{\Pi}} \E_0^{} \left[ \mathscr{L}_\pi \left({X}^{j}\right)^2 \right] \int \exp \left( f^\top \Xi_u g \right) d (\pi \times \pi)(f,g) d\P^U(u),\label{eq:help:new1_main}
    \end{equation}
    where
    \begin{equation}\label{eq : Xi_u matrix lb}
     \Xi_u := \underset{j=1}{\overset{m}{\sum}} \E_0^{Y^{j}|U=u} \E_0\left[ {X}^{j} \bigg| Y^{j}, U=u \right] \E_0\left[ {X}^{j} \bigg| Y^{j}, U=u \right]^\top.
    \end{equation}}
    \item The final step combines data processing techniques with what is essentially a geometric argument. The first term {in \eqref{eq:help:new1_main}}  is handled using classical, non-distributed techniques, i.e. decoupling argument of the measure and the moment generating function of the Gaussian chaos, see e.g. \cite{vershynin_high-dimensional_2018}. In the second term {in \eqref{eq:help:new1_main}} the $d \times d$ matrix $\Xi_u$ geometrically captures how well $Y$ allows to ``reconstruct'' the compressed sample $X$. {The information lost by compressing a $d$ dimensional observation $X^{j}$ into a $b$-bit transcript $Y^{j}$ is captured in a data processing inequality for the matrix $\Xi_u$ and its trace,} which comes in the form of Lemma \ref{lem : strict DPI} and Lemma \ref{lem : trace of fisher info}. From hereon out, the proof of the private and the public coin cases separate. Recalling the order of the supremum, infimum and expectation with respect to the public coin {in \eqref{eq : sup inf public coin divergence lb}}, we see that in the private coin case, $\pi$ can be chosen with knowledge of $\Xi_u$, as $U$ is degenerate in this case. To obtain the stricter lower bound in the private coin case, we choose $\pi$'s covariance in order to exploit the ``weakest directions'' of the protocol $Y$ and the proof is finished by matrix algebra arguments. 
\end{enumerate}

\section{{Nonparametric testing with known regularity}}\label{sec : nonparametric part}

A natural extension of the above finite dimensional signal in Gaussian noise setting is the infinite dimensional signal in white noise model. Here, the $j=1,\dots,m$ machines observe iid $X^{j}$ taking values in $\cX \subset L_2[0,1]$ and subject to the stochastic differential equation
\begin{equation}\label{eq : nonparametric model dynamics}
dX^{j}_t = f(t)dt + \sqrt{\frac{m}{n}}dW_t^{j}
\end{equation}
under $P_f$, with $W^{1},\dots,W^{m}$ iid Brownian motions and $f \in L_2[0,1]$. Besides the difference in the local observations, the distributed setup considered for this model remains exactly the same. The results derived for the alternatives $H_{\rho}$ in the finite dimensional model translate to testing in the infinite dimensional model against the alternative hypotheses
 \begin{equation*}
f \in H_\rho^{s,R} := \{ f \in \cH^{s,R}[0,1] : \|f\|_{L_2} \geq \rho \text{ and } \|f\|_{\cH^s} \leq R \}.
 \end{equation*}
Here, $\cH^{s,R}=\cH^{s,R}([0,1])$ denotes the Sobolev ball of radius $R$ in the space of $s$-smooth Sobolev functions and $\|\cdot\|_{\cH^s}$ the Sobolev norm, see Section \ref{sec: wavelets} for recalling the definitions. The smoothness parameter $s > 0$ determines the difficulty of the classical (non-distributed, $m=1$) nonparametric testing problem as considered in e.g. \cite{ingster_nonparametric_2003}. The asymptotic minimax rate for the non-distributed case is $\rho^2 \asymp n^{-\frac{2s}{2s + 1/2}}$ for the $s$-smooth Sobolev alternative class. 

We allow for asymptotics in $b$ and $m$ in the sense that they can depend on $n$. Consequently, we consider the separation rate $\rho$ in the nonparametric problem to be a sequence of positive numbers in both $n$, $m$ and the budget $b$. A distributed test $T$ in the nonparametric setting is called $\alpha$-consistent for $\alpha \in (0,1)$ if $\cR( H_{\rho}^{s,R} , T) \leq \alpha$ for all $n$ large enough.

The distributed setting for the nonparametric model remains unchanged in comparison with the finite dimensional model introduced in Section 2, except of course for the sample space in which the observations $X^{j}$ take values. This becomes $L_2[0,1]$ instead of $\R^d$. The following theorem describes the minimax rate for the nonparametric distributed problem.

\begin{theorem}[Nonparametric signal in white noise minimax rate]\label{thm : nonparametric SNWN minimax rate} Take $f\in H^{s,R}$ for some $s,R>0$ and let $b\equiv b_n$ and $m \equiv m_n$ be sequences of natural numbers and take $ \rho\equiv \rho_{n,b,m,s}$ be a sequence of positive numbers satisfying
\begin{equation}\label{eq : rho condition public coin nonparametric}
\rho^2 \asymp \begin{cases} 
 n^{-\frac{2s}{2s+1/2}}, &\text{ if } b \geq n^{\frac{1}{2s+1/2}}, \\
 \left({ \sqrt{b} n}\right)^{-\frac{2s}{2s+1}}, &\text{ if }  n^{\frac{1}{2s+1/2}}/m^{\frac{2s+1}{2s+1/2}}  \leq  b < n^{\frac{1}{2s+1/2}} ,  \\
(n/\sqrt{m})^{-\frac{2s}{2s+1/2}}, &\text{ if }  b < n^{\frac{1}{2s+1/2}} /m^{\frac{2s+1}{2s+1/2}}. 
 %\\ \frac{1}{n} &\text{ if } \rho > 1.
\end{cases}
\end{equation}
In the public coin protocol case the minimax testing rate is $\rho^2$ given in \eqref{eq : rho condition public coin nonparametric}, i.e. for all $\alpha \in (0,1)$ there exist constants $C_\alpha,c_\alpha >0$  depending only on $\alpha$, $s$ and $R$ such that for all $n$ large enough,
\begin{equation*}
\underset{T \in \cT_{pub}(b)}{\inf} \; \; \cR(H_{c_\alpha \rho}^{s,R} , T) >1- \alpha \;\; \text{ and } \;\; \underset{T \in \cT_{pub}(b)}{\inf} \; \; \cR(H_{C_\alpha \rho}^{s,R} , T) \leq \alpha.
\end{equation*}
Similarly, in the private coin protocol case $ \rho\equiv \rho_{n,b,m}$ given below
\begin{equation}\label{eq : rho condition private coin nonparametric}
\rho^2 \asymp   \begin{cases} 
n^{-\frac{2s}{2s+1/2}} &\text{ if } b \geq n^{\frac{1}{2s+1/2}}, \\
 \left({ b n}\right)^{-\frac{2s}{2s+3/2}} &\text{ if }  n^{\frac{1}{2s+1/2}}/m^{\frac{s+3/4}{2s+1/2}}  \leq b < n^{\frac{1}{2s+1/2}}  ,  \\
(n/\sqrt{m})^{-\frac{2s}{2s+1/2}} &\text{ if }  b <  n^{\frac{1}{2s+1/2}}/m^{\frac{s+3/4}{2s+1/2}}, 
 %\\ \frac{1}{n} &\text{ if } \rho > 1,
\end{cases}
\end{equation}
provides the minimal testing rate, i.e. for all $\alpha \in (0,1)$ there exist constants $C_\alpha,c_\alpha >0$ depending only on $\alpha$ and $R$ such that for all $n$ large enough,
\begin{equation*}
\underset{T \in \cT_{priv}(b)}{\inf} \; \; \cR(H_{c_\alpha \rho}^{s,R} , T) > 1-\alpha \;\; \text{ and } \;\; \underset{T \in \cT_{priv}(b)}{\inf} \; \; \cR(H_{C_\alpha \rho}^{s,R} , T) \leq \alpha.
\end{equation*}
\end{theorem}

The proof of the theorem is given in Section \ref{sec: proof:nonparam}. The theorem reveals the relationship between the signal-to-noise-ratio $n$, communication budget per machine $b$, the number of machines $m$ and the smoothness of the signal $s$. Before providing the proof we briefly discuss the connection with distributed minimax estimation rates.

The distributed minimax estimation rates {under private coin protocol} were established in Corollary 2.2 of \cite{szabo2020adaptive} or Theorem 3.1 in \cite{pmlr-v80-zhu18a}. A slight reformulation of the latter yields that
\begin{equation}\label{eq : estimation nonparametric rate}
\underset{(\hat{f},\cL(Y)) \in \cE_{priv}(b)}{\inf} \; \underset{f \in \cH^{s,R}}{\sup} \E_f^Y \| \hat{f}(Y) - f \|^2_{L_2} \asymp \begin{cases} 
 n^{-\frac{2s}{2s+1}}, &\text{ if } b \geq n^{\frac{1}{2s+1}}, \\
{(b n)}^{-\frac{2s}{2s+2}}, &\text{ if } (n/m^{2+2s})^{\frac{1}{2s+1}}\leq b \leq n^{\frac{1}{2s+1}} ,  \\
 \left(bm \right)^{-2s}, &\text{ if } b \leq  (n/m^{2+2s})^{\frac{1}{2s+1}},
 %\\ \frac{1}{n} &\text{ if } \rho > 1.
\end{cases}
\end{equation}
where $\cE_{priv}(b)$ is the class of all distributed estimators based on $b$-bit transcripts $Y= \left(Y^{1},\dots,Y^{m}\right)$. 

A first observation is that consistent testing is possible in any regime of $b \geq 1$ and $m$, whereas this is not the case in estimation. Consider for instance the regime where $m$ and $b$ are fixed. In nonparametric distributed estimation, the $L_2$-risk does not improve once the sample size is large enough. In fact, even when allowing for asymptotics in $b$ and $m$ (but assuming that $(n/m^{2+2s})^{\frac{1}{2s+1}}{\geq b}$)  one is better off performing the estimation locally using just one of the machines with local signal-to-noise-ratio $n/m$, attaining the locally optimal rate $(n/m)^{-\frac{2s}{2s+1}}$.

In the case of nonparametric testing, not only can we consistently test for any fixed $m$ and $b$, the distributed testing rate is bounded from above by $(n/\sqrt{m})^{-2s/(2s+1/2)}$ (regardless of the communication budget $b$), which is significantly smaller (for large $m$) than the minimax testing rate based on the local signal-to-noise-ratio $(n/m)^{-2s/(2s+1/2)}$, which can be achieved by using only a single local machine. One possible explanation for this discrepancy is that in nonparametric estimation, the output of the inference is a high-dimensional object, which requires a large total communication budget to be reconstructed with sufficient granularity. In testing, the output of our inference is binary.

A perhaps less surprising difference is that a larger budget is needed for testing at the non-distributed minimax testing rate compared to estimation. That is, in order to obtain the non-distributed minimax rate of $\rho^2 \asymp n^{-\frac{2s}{2s+1/2}}$, the communication budget needs to satisfy $b \gtrsim n^{\frac{1}{2s+1/2}}$. On the other hand, the non-distributed minimax estimation rate $n^{-\frac{2s}{2s+1}}$ requires only $b \gtrsim n^{\frac{1}{2s+1}}$. This follows from the fact that the $L_2$ testing rate is faster than the estimation rate and hence to achieve this faster rate one has to collect information about the signal at higher frequency level as well (up to the $O(n^{\frac{1}{2s+1/2}})$ coefficients in the spectral decomposition).

Increasing $m$ decreases the local signal-to-noise-ratio. When the total budget $bm$ grows at a similar or faster rate than the ``effective dimension'' of the model, the rate that can be achieved no longer depends on $m$ in both estimation and testing settings. In this regime, this effect is offset by the total number of bits being received by the central machine. What is different in testing problem, however, is that having access to shared randomness strictly improves the performance (until the local communication budget $b$ reaches the effective dimension $n^{\frac{1}{2s+1/2}}$ as after that both method reaches the minimax non-distributed testing rate $n^{-\frac{2s}{2s+1/2}}$). One might wonder whether having access to a public coin improves the rate in the estimation setting also. It turns out that this is not the case. We show in Theorem \ref{thm : public coin estimation} in the Supplementary Material that under the public coin protocol the distributed minimax estimation rate does not improve compared to the private coin protocol.

\section{Adaptation in nonparametrics}\label{sec : adaptation}
In the previous section we have derived minimax lower and matching upper bounds for the nonparametric distributed testing problem in context of the Gaussian white noise model. The proposed tests, however, depend on the regularity hyper-parameter $s$ of the functional parameter of interest $f$. Typically, the regularity of the function is not known in practice and one has to use data driven methods to find the best testing strategies. In this section we derive distributed tests adapting to this unknown regularity. We derive both lower and upper bounds and observe surprising, additional phase transition in the small budget regime which was not present in the non-adaptive setting.

First, we note that even in the non-distributed setting, we have to pay an additional $\log\log n$ factor as a price for adaptation (see e.g. Theorem 2.3 in \cite{spokoiny_adaptive_1996} or Section 7 in \cite{ingster_nonparametric_2003}). More concretely, if $\rho_s \asymp n^{-s/(2s+1/2)}$, it holds that for any $s_{\min}<s_{\max}$,
\begin{equation*}
\underset{s \in [s_{min} , s_{max}]}{\sup}  \cR( H^{s,R}_{c_n M_{n,s}\rho_s}, T ) \to 1,
\end{equation*}
for all tests $T$, ${M_{n,s}}=(\log\log n)^{\frac{s/4}{2s+1/2}}$ and any $c_n=o(1)$ whilst there exists a test $T$ satisfying 
\begin{equation*}
\underset{s \in [s_{min} , s_{max}]}{\sup}  \cR(  H^{s,R}_{C M_{n,s} \rho_s},T ) \to 0.
\end{equation*}
for large enough constant $C>0$.

The distributed testing problem is more complicated as we have to consider different regimes based on the number of transmitted bits, see Theorem \ref{thm : nonparametric SNWN minimax rate}. These regimes, however, depend on the unknown regularity hyper-parameter and require different testing procedures to achieve consistent testing. The transcripts transmitted require a larger communication budget to attain the same performance as in Theorem \ref{thm : nonparametric SNWN minimax rate}. {Theorem \ref{thm:adapt1} and \ref{thm:adapt2} below capture this increased difficulty in terms of lower- and upper bounds on the detection rate {(tight up to a log-log factor)}. In the proof of the theorem,} we derive such an adaptive distributed testing method which adapts to {the smoothness. These methods are in principle based on taking a $1/\log n$ grid of the regularity interval $[s_{\min},s_{\max}]$, constructing optimal tests for each of the grid points and combining them using Bonferroni's correction. This results in loosing a logarithmic factor in the intermediate case as the budget has to be divided over $O(\log n)$ tests, each capturing a different possible level of smoothness. 

This additional incurred cost in the distributed setting due to additional communication budget required is fundamental, as our accompanying lower bound shows. This additional difficulty translates to a $\sqrt{\log (n)}$ and $\log (n)$ factor more observations required in the intermediate budget regimes for the public and private coin settings, respectively. In the small budget regime, such a loss is incurred when the {local communication budget $b$} is of smaller order than $\log(n)$. When $b \gtrsim \log(n)$ in the small budget regime, the same rate as in Theorem \ref{thm : nonparametric SNWN minimax rate} can be obtained, up to the $\log \log (n)$ factor incurred by the Bonferroni correction.} 

The above described results are split over two theorems. The first, Theorem \ref{thm:adapt1}, concerns the case where $b \gtrsim \log(n)$. In the second, Theorem \ref{thm:adapt2}, the case where $b \lesssim \log(n)$ (both theorems coincide when $b\asymp \log(n)$). The case where $b =O(1)$ is of special interest, as $b = 1$ means each machine's local transcript forms a test itself and the global test can be seen as a ``meta-analysis'' on the basis of these $m$ tests. {The proofs of the upper bounds in both theorems are given in
Section  \ref{sec: proof:adapt}, while the proofs of the lower bound are deferred to Section \ref{sec : adaptation lower bounds} in the supplement.}

\begin{theorem}\label{thm:adapt1}
Let us consider some $0<s_{\min}<s_{\max}<\infty$, $R>0$, let $b\equiv b_n$ such that $b\gg \log n$ and $m \equiv m_n$ be sequences of natural numbers and take a sequence of positive numbers $ \rho_s\equiv \rho_{n,b,m,s}$ satisfying 
\begin{equation}\label{eq : rho condition public coin adaptive lb}
\rho^2_s \asymp  \begin{cases} 
 n^{-\frac{2s}{2s+1/2}}, &\text{ if } \; b \geq {{\log(n)} n^{\frac{1}{2s+1/2}}}, \\
 \left(\frac{ \sqrt{b} n}{\sqrt{\log(n)}}\right)^{-\frac{2s}{2s+1}}, &\text{ if } \; {\log(n)} \left(\frac{n^{\frac{1}{2s+1/2}}}{m^{\frac{2s+1}{2s+1/2}}} \bigvee 1 \right) \leq  b < {{\log(n)}n^{\frac{1}{2s+1/2}}} ,  \\
\left(\frac{n}{\sqrt{m}}\right)^{-\frac{2s}{2s+1/2}}, &\text{ if } \; \log(n) \leq b < {\log(n)} \left(\frac{n^{\frac{1}{2s+1/2}}}{m^{\frac{2s+1}{2s+1/2}}} \bigvee 1 \right).
 %\\ \frac{1}{n} &\text{ if } \rho > 1.
\end{cases}
\end{equation}
in the public coin case, and 
\begin{equation}\label{eq : rho condition private coin adaptive lb}
\rho^2_s \asymp   \begin{cases} 
n^{-\frac{2s}{2s+1/2}} &\text{ if } \; b \geq {{\log(n)}n^{\frac{1}{2s+1/2}}}, \\
 \left(\frac{ b n }{\log(n)}\right)^{-\frac{2s}{2s+3/2}} &\text{ if } \; \log(n) \left( \frac{n^{\frac{1}{2s+1/2}}}{m^{\frac{s+3/4}{2s+1/2}}} \bigvee 1 \right)  \leq b < {{\log(n)}n^{\frac{1}{2s+1/2}}},  \\
\left(\frac{n  }{\sqrt{m }}\right)^{-\frac{2s}{2s+1/2}} &\text{ if } \; \log(n) \leq  b <  \log(n) \left( \frac{n^{\frac{1}{2s+1/2}}}{m^{\frac{s+3/4}{2s+1/2}}} \bigvee 1 \right).
 %\\ \frac{1}{n} &\text{ if } \rho > 1,
\end{cases}
\end{equation}
in the case of a private coin. Then, there exits a sequence of distributed testing procedures in the respective setups such that
\begin{equation*}
\underset{s \in [s_{\min}, s_{\max}]}{\sup} \cR(H_{M_n \rho_s}^{s,R} , T) \to 0,
\end{equation*}
for arbitrary $M_n\gg \big(\log\log(n)\big)^{1/4}$. Similarly, for all distributed testing procedures in the respective setups, we have that for all $\alpha \in (0,1)$ there exists $c_\alpha > 0$ such that 
\begin{equation*}
\underset{s \in [s_{\min}, s_{\max}]}{\sup} \cR(H_{c_\alpha \rho_s}^{s,R} , T) > \alpha.
\end{equation*}
\end{theorem}

The above theorem recovers {(up to log-factors)} the three rates corresponding to the three regimes also found in Theorem \ref{thm : nonparametric SNWN minimax rate}, the different regimes corresponding to different testing strategies. Since the true smoothness is unknown, these different distributed testing strategies are to be conducted simultaneously.

{{We note that for} $m \geq n^{\frac{1}{2s_{\min}+1}}$ or $ m \geq n^{\frac{1}{s_{\min}+3/4}}$ in the public and private coin cases, respectively, the small budget regime no longer occurs.} The reason for this is that, even though $b$ could be relatively small, the total communication budget $bm$ is large enough to warrant the strategy for the intermediate and high budget regimes. {Furthermore,} whenever  $b > {{\log(n)}n^{\frac{1}{2s+1/2}}}$, the budget is large enough to recover the non-distributed regime rate.

For $b \lesssim \log(n)$ the separation rate is different from the non-adaptive low budget regime. Depending on the interplay between $n$ and $m$ either the minimax rate corresponding to the intermediate case applies or an additional $(\log (n)/b)^\delta$ factor is present compared to the non-adaptive low budget regime, both in the private and public coin settings. This results in an additional phase transition at $b=\log n$. The reason for this, is that in order to cover approximately $\log(n)$ different levels of smoothness using less than $\log(n)$ bits, each of the machines can no longer send an adequate amount of information on all of the relevant smoothness levels. Instead, an optimal strategy is to divide the different machines over each of the smoothness levels, where each machines foregoes sending information regarding certain smoothness levels all together.

\begin{theorem}\label{thm:adapt2}
Assume the conditions of Theorem \ref{thm:adapt1} with $b \lesssim \log(n)$ and assume $bm \gg \log(n)$. {Let us consider}
\begin{equation}\label{eq : rho condition public coin adaptive lb 2}
\rho^2_s \asymp  \begin{cases} 
 \left(\frac{ \sqrt{b} n}{\sqrt{\log(n)}}\right)^{-\frac{2s}{2s+1}}, &\text{ if } \; m \geq n^{\frac{1}{2s+1}},  \\
\left(\frac{\sqrt{b} n }{\sqrt{m \log(n)}}\right)^{-\frac{2s}{2s+1/2}}, &\text{ if } \; m < n^{\frac{1}{2s+1}},
 %\\ \frac{1}{n} &\text{ if } \rho > 1.
\end{cases}
\end{equation}
in the public coin case {and}
\begin{equation}\label{eq : rho condition private coin adaptive lb 2}
\rho^2_s \asymp   \begin{cases} 
\left(\frac{ b n }{\log(n)}\right)^{-\frac{2s}{2s+3/2}} &{ if } \; m \geq n^{\frac{2}{2s+3/2}} \left( \frac{b}{\log(n)} \right)^{\frac{s-1/4}{2s+3/2}},  \\
\left(\frac{n\sqrt{b} }{\sqrt{m \log(n)}}\right)^{-\frac{2s}{2s+1/2}} &\text{ if } \;  \; m < n^{\frac{2}{2s+3/2}} \left( \frac{b}{\log(n)} \right)^{\frac{s-1/4}{2s+3/2}}.
 %\\ \frac{1}{n} &\text{ if } \rho > 1,
\end{cases}
\end{equation}
in the private coin case. Then, there exits a sequence of distributed testing procedures in the respective setups such that
\begin{equation*}
\underset{s \in [s_{\min}, s_{\max}]}{\sup} \cR(H_{M_n \rho_s}^{s,R} , T) \to 0,
\end{equation*}
for arbitrary $M_n\gg \big(\log\log(n)\big)^{1/4}$. Similarly, for all distributed testing procedures in the respective setups, we have that for all $\alpha \in (0,1)$ there exists $c_\alpha > 0$ such that 
\begin{equation*}
\underset{s \in [s_{\min}, s_{\max}]}{\sup} \cR(H_{c_\alpha \rho_s}^{s,R} , T) > \alpha.
\end{equation*}
\end{theorem}

\begin{remark}
{Both theorems together cover all cases where $mb \gg \log (n)$. The cases where $mb \lesssim \log (n)$ are excluded for technical reasons, as well as the fact that when $mb \lesssim \log (n)$, the optimal rate in \eqref{eq : rho condition public coin adaptive lb 2}-\eqref{eq : rho condition private coin adaptive lb 2} (up to at most a $\sqrt{\log \log (n)}$ factor) is attained by using a standard non-distributed method using just the data of one machine (see e.g. \cite{spokoiny_adaptive_1996}). Similarly, in order to contain the level of technicality, we have foregone the $(\log \log (n))^{1/4}$ additional factor in the lower bound which we esteem also to be present in the distributed setting. We refer the reader to the argument of Theorem 2.3 in \cite{spokoiny_adaptive_1996} for how to obtain the $(\log \log (n))^{1/4}$ factor in the lower bound in addition to the $\sqrt{\log(n)}$ and ${\log(n)}$ factors in the public and private coin cases respectively.}
\end{remark}

\section{Adaptive tests attaining the adaptation bounds in Theorem \ref{thm:adapt1} and \ref{thm:adapt2}}\label{sec: proof:adapt}

{Let us consider the smooth orthonormal wavelet basis $\{ \psi_{li}\, : l \in \N_0, \, i=0,1,\dots,2^l-1 \}$. See Section \ref{sec: wavelets} for a brief introduction of wavelets and collection of properties used in this proof. For $L=L \in \N$, let $V_L = \{ \psi_{li} : l \leq L, \; i = 0, 1, \dots, 2^l - 1\}$. For $f \in L_2[0,1]$, let $f^L$ denote the projection of $f$ onto $V_L$, i.e.
\begin{equation}\label{eq : wavelet projection f}
f^L = \underset{l=0}{\overset{L}{\sum}} \underset{i=0}{\overset{2^l-1}{\sum}}  \tilde{f}_{li}\psi_{li}   
\end{equation}
with $\tilde{f}_{li} := \int f \psi_{li}$. We denote the wavelet coefficients of $X^{j}$ by $\tilde{X}_{li}^{j} := \int_0^1 \psi_{li} dX_t^{j}$. For the coefficients at resolution level $L$, write $\tilde{X}_L^{j} = (\tilde{X}_{L0}^{j}, \dots, \tilde{X}_{L(2^L-1)}^{j}) \in \R^{2^L}$ and let $\tilde{X}_{L':L}^{j}$ denote the concatenated coefficients from resolution level $L' < L$ up to resolution level $L$, i.e. $\tilde{X}_{L':L}^{j} = (\tilde{X}_{L'}^{j}, \dots, \tilde{X}_{L}^{j}) \in \R^{2^{L+1}-2^{L'+1}}$. The vector $\tilde{X}^{j}_{0:L} := (\tilde{X}_{0}^{j},\tilde{X}_{1}^{j},\dots,\tilde{X}_{L}^{j})$ follows the dynamics
\begin{equation}\label{eq : wavelet coefficients X}
\tilde{X}^{j}_{0:L} = \tilde{f}^L + \sqrt{\frac{m}{n}} Z^{j},
\end{equation}
where $Z^{j} \sim^{iid} N(0,  I_{2^{L+1}-1})$, $j=1,\dots,m$, and $\tilde{f^L} := (\tilde{f}_{li})_{l=0,...,L;\, i=0,...,2^{l}-1}$.}

Let $\nu_L = 2^{L+1}-1$ and let us introduce the notations ${L_{s}}=\lfloor s^{-1}\log (1/\rho_s)\rfloor\vee 1$, and for shorthand write $L_{\min}=L_{s_{\max}}$ and $L_{\max}=L_{s_{\min}}$ and note that $L_s\in \mathcal{C}:=\{L_{\min},...,L_{\max}\}$ for all $s\in[s_{\min},s_{\max}]$. Note that  $|\mathcal{C}|\leq \log n$.

For each regularity hyper-parameter $s$, we distinguish low-budget ($2^{L_s} \gtrsim mb$ in the public coin, and $2^{\frac{3}{2}L_s} \gtrsim mb$ in the private coin setting) and {high-budget (corresponding to  $2^{L_s} \lesssim mb$ in the public coin and $2^{\frac{3}{2}L_s} \lesssim mb$ in the private coin setting)  cases}. Since $m$ and $b$ are known for any given regularity $s$ we know which regime it falls and is sufficient to construct that test. For notational convenience, without loss of generality, for each $s$ we construct both the high-budget and the low-budget optimal tests using all the $m$ machines (and do not split them between these two cases).

\subsection{Proof of the upper bound {in the} low-budget regime}
First we deal with the low-budget case (where the total budget is small compared to the effective dimension), which coincides in both setups. For each $L\in \mathcal{C}$ we take a subset of machines $M_L\subset\{1,...,m\}$ such that $|M_L|=m':=\frac{m (\log(n) \wedge b)}{\log(n)}$ and each machine appears in at most $b$ such subsets. We note that this is possible since $m'|\mathcal{C}|\leq mb$. Then for each $j\in M_L$, $L\in\mathcal{C}$ we communicate
\begin{equation}\label{eq : adaptive Bj}
Y_I^{j}(L) | X^{j} \sim \text{Ber}\left(  \chi^2_{\nu_L}\left(\sqrt{n/m}\|\tilde{X}_{0:L}^{j}\|_2^2 \right) \right)
\end{equation}
and at the central machine, we can compute
\begin{equation*}
S_I(L) = \frac{1}{\sqrt{m'}}\underset{j\in M_L}{\sum} (2Y_I^{j}(L) - 1).
\end{equation*}
Then we consider the following adaptive test based on Bonferroni's correction
\begin{equation*}
T_I^{adapt} = \mathbbm{1} \Big\{\underset{L\in \mathcal{C} }{\max} S_I(L) \geq 2 \sqrt{\log \log n} \Big\}.
\end{equation*}
Since for $L\in\mathcal{C}$, it holds that $L \asymp \log(n)$, {the above $\sqrt{\log \log n}$ blow up suffices to guarantee that the test has asymptotically vanishing Type I error control, i.e. $\E_0 T_I^{adapt} =o(1)$ by Lemma \ref{lem : petrov iterated logarithm bound} in the Supplementary Material (as the random variables $2Y_I^{j}(L) - 1$ are iid Rademacher under $\mathbb{P}_0$).}

For the Type II error note that
\begin{equation*}
\E_f (1 - T_I^{adapt}) \leq \P_f \Big( S_I(L_s) < 2 \sqrt{\log \log n} \Big)
\end{equation*}
and aim to apply Lemma \ref{lem : binomial testing supplement}. In view of Lemma \ref{lem : T_I proof}, (with $\|f\|_2$ replaced by $\|\tilde{f}^{L_s}\|_2$ and $d=\nu_{L_s}$), noting that by triangle inequality {$\|\tilde{f}^{L_s}\|_2^2 \geq \|f\|_2^2/2-2^{-2L_s s}R^{2}$} (see also Section \ref{sec: proof:nonparam} in the Supplementary Material), we get for  $\|f\|_2^2\geq C_0^2\sqrt{\log\log (n)}\rho_s^2\geq C_0^2\sqrt{\log\log (n)} { \frac{\sqrt{2^{L_s} m\log(n)} }{n \sqrt{b \wedge \log(n)}}}$, that for $m$ large enough
\begin{align*}
\eta_{p,m',1}\gtrsim  (m'-1) \Big(\frac{n\| \tilde{f}^{L_s} \|_2^2}{m2^{L_s/2}}\wedge \frac{1}{2} \Big)^2\gtrsim  m' \Big((\tilde{C} \frac{  \log\log n}{m'} ) \wedge (1/4)\Big),
 \end{align*}
 with $\tilde{C}=C_0^2/2-R^2$. By the assumption that $bm \gg \log(n)$, $m'$ can be taken larger than {arbitrary constant $M_0 >0$}. This means that, in view of Lemma \ref{lem : binomial testing supplement} with $c_{\alpha,n}=4\log\log n$ and large enough constant $C_0$ (depending on $R$), the Type II error is bounded by $\alpha$.

\subsection{Proof of the {upper bound in the public coin, high budget regime}}

 We use similar arguments as before, applying a Bonferroni-type of correction. First let us consider the public coin setting and take a one-to-one mapping $\xi_L$ from $\{1,\dots,\nu_L\}$ to $\{(l,i) : l=0,\dots,L, \; i = 0,1,\dots,2^{l}-1\}$. Let us define the test
\begin{equation}\label{eq : public coin adaptive transcript}
(Y_{\text{II}}^{j}(L))_{i} | U_L = \mathbbm{1} \left\{ \left( \sqrt{n/m} U_L \tilde{X}_{\xi_L(i)}^{j}\right)_{i} > 0 \right\},
\end{equation}
where the random variable $U_L \in \R^{\nu_L \times \nu_L}$ is drawn from the Haar measure on the rotation group on $\R^{\nu_L}$. Similarly to before for each $L$ we take a subset of machines $M_L\subseteq\{1,...,m\}$ such that $|M_L|=m':=\frac{m (b \wedge \log(n))}{\log(n)}$, and each machine appears at most in $b$ such sets. 

Then machine $j\in M_L$, $L\in\mathcal{C}$, transmits the bits $(Y_{\text{II}}^{j}(L))_{i}$, $i=1,...,b':= \frac{mb }{m'|\mathcal{C}|}\wedge \nu_L$  to the central machine, where these local test statistics are aggregated, similarly to \eqref{eq : public coin dist test}, as
\begin{equation}\label{eq : pub coin adaptive test stat}
S_{\text{II}}(L) = \frac{1}{\sqrt{b'} m'} \underset{i=1}{\overset{b'}{\sum}} \left[ \left(\underset{j\in M_L}{\sum} \left[ (Y_{\text{II}}^{j}(L))_{i} - 1/2\right] \right)^2 - \frac{m'}{4} \right].
\end{equation}
 In view of {Lemma \ref{lem : petrov iterated logarithm bound}} the Type I error of the test
\begin{equation*}
T_{\text{II}}^{pub,adapt} := \mathbbm{1} \Big\{ \underset{L\in\mathcal{C}}{\max} S_{\text{II}}(L) \geq 2  \sqrt{\log \log n}  \Big\}
\end{equation*}
is $o(1)$.
For the Type II error note that
\begin{equation*}
\E_f (1-T_{\text{II}}^{pub,adapt}) \leq \E_f \mathbbm{1} \left\{  S_{\text{II}}(L_s) < 2 \sqrt{\log \log n}  \right\}. 
\end{equation*}
{By Lemma \ref{lem:adaptive_public_coin_proof}, the above display is $o(1)$ whenever $\rho^2 \gtrsim M_n {\frac{2^{L_s}  }{n \sqrt{ \frac{b}{\log(n)} \wedge 2^{L_s}}}}$, which, for the choice of $L_s{=\lfloor s^{-1}\log(1/\rho_s)\rfloor\vee 1}$ yields the rates of Theorem \ref{thm:adapt1} and \ref{thm:adapt2}.}

\subsection{Proof of the {upper bound in the private coin, high-budget regime}}  We proceed by adapting the test $T_{\text{III}}$ provided in Section \ref{ssec : private coin high-budget} to the nonparametric setting with unknown regularity using again a Bonferroni type correction to achieve adaptation. For simplicity we again apply the map $\xi_L$ introduced previously to move between the single and double index notations of the sequence model.

For all $L\in \mathcal{C}$, similarly to the previous cases we consider a collection of machines $M_L$ with $|M_L|=m'= \frac{m (b \wedge \log(n))}{\log(n)}$ and similarly to Section \ref{ssec : private coin high-budget} let us use the notation $\cI_i(L)\subset M_L$ for the collection of machines corresponding the $i$th coordinate. We note that without loss of generality we can assume that $m' \geq M_{\alpha}\sqrt{\log\log n} 2^{2L_s}/(b')^2$, for some large enough constant $M_{\alpha}$, otherwise the test $T_I^{adapt}$ above covers the corresponding range. Then we modify the test given in \eqref{eq : large budget priv coin test I def} by increasing the threshold with the Bonferroni correction, i.e.
\begin{align*}
T_{\text{III}}^{priv, adapt,1} &=  \mathbbm{1} \Big\{\underset{L\in\mathcal{C}}{\max} S^{\text{III},1}(L) \geq 2\sqrt{\log\log n} \Big\},\quad\text{where}\\
S^{\text{III},1}(L)&= \Big|\frac{1}{{| \cI_1(L) |}2^{L/2}} \underset{i=1}{\overset{\nu_L}{\sum}} \Big(  \underset{ j \in \cI_i(L)}{\overset{}{\sum}} (Y_i^{j} - 1/2) \Big)^2 - 2^{L/2}/4 \Big|,\quad Y_{i}^{j}|\tilde{X}_{\xi_L(i)}^{j}=\mathbbm{1}_{\tilde{X}_{\xi_L(i)}^{j}>0}.
\end{align*}
To deal with large signal components, similarly to \eqref{eq : large budget priv coin test I def} (with $d=\nu_L$ and including the Bonferroni correction in the threshold), we propose the test,
\begin{align*}
T_{\text{III}}^{priv,adapt,2} &= \mathbbm{1} \Big\{\underset{L\in\mathcal{C},  2\log(L)\leq b}{\max} S^{\text{III},2}(L) \geq \kappa_\alpha\sqrt{\log\log n} \Big\},\quad \text{where}\\
 S^{\text{III},2}(L)&=\bigg| \frac{1}{dm'C_{b,L}} \left(\underset{j=1}{\overset{m'}{\sum}}  (N^{j}- C_{b,L}2^{L-1}) \right)^2 - \frac{1}{4} \bigg|,
\end{align*}
with $C_{b,L}=2^{b-L}$ and $N^{j}$ given in \eqref{def:Nj}. Finally, we aggregate these tests by taking
\begin{equation*}
T_{\text{III}}^{priv,adapt} = T_{\text{III}}^{priv,adapt,1}  \vee T_{\text{III}}^{priv,adapt,2}.
\end{equation*}

In view of the law of Lemma \ref{lem : petrov iterated logarithm bound} the Type I error tends to zero for both tests. Therefore it remained to show that the Type II error is bounded by $\alpha$. Similarly to the previous cases, note that 
$$E_f (1-T_{\text{III}}^{priv,adapt}) \leq \E_f \Big(\mathbbm{1} \left\{  S^{\text{III},1}(L_s) < 2 \sqrt{\log \log n}  \right\}\wedge \mathbbm{1} \left\{  S^{\text{III},2}(L_s) < 2\sqrt{\log \log n}  \right\}  \Big).$$

Following the proofs of Lemmas \ref{lem: private:large:budget}, \ref{lem : private coin test risk ub I} and \ref{lem : priv coin high-budget additional test II} (with $d=\nu_{L_s}$, $f$ taken to be the $\nu_{L_s}$ dimensional vector $\tilde{f}^{L_s}$, $b$ replaced by $b'$, and $M_{\alpha}$ replaced by $M_0\sqrt{\log\log n}$, for some large enough $M_0>0$), noting that for $C_0^2>4R^2${
\begin{align*}
\|\tilde{f}^{L_s}\|_2^2&\geq \|f\|_2^2/2 -R^2 2^{-2L_s s}\gtrsim
 C_0\sqrt{\log\log (n)} \rho_s^2 \\
& = \frac{C_02^{3L_s/2}\sqrt{\log\log n}}{2n (\frac{b}{\log(n)} \wedge 2^{L_s})}\gtrsim  \frac{C_02^{L_s}\sqrt{\log\log n}}{n b' \frac{m'}{m}},
\end{align*}
 } and applying Lemmas \ref{lem : normal cdf divergence lb} and \ref{lem : binomial testing supplement} with $c_{n,\alpha}=2\sqrt{\log\log n}$, we get that the Type II error of $T_{\text{III}}^{priv,adapt} $ is bounded from above by $\alpha/2$.

Finally, we combine the above tests by taking
\begin{align*}
T^{priv,adapt}=T_{\text{III}}^{priv,adapt} \vee  T_{\text{I}}^{priv,adapt}\quad \text{and}\quad T^{pub,adapt}=T_{\text{II}}^{pub,adapt} \vee  T_{\text{I}}^{pub,adapt}.
\end{align*}
Note that both of the above tests still have vanishing Type I error, while the Type II errors are bounded by the prescribed level $\alpha$ in view of taking the union of the above optimal tests.

\section{Proof of the testing lower bound}\label{sec : lower bound}

\noindent We provide the details for Steps 1 and 2 as outlined in Section \ref{sec : sketch lb}. We shall write $\mathscr{L}_\pi(x)=\int \mathscr{L}_f(x) d \pi(f)$ with $\mathscr{L}_f(x) := \frac{dP_f}{dP_0}(x)$ and $P_f=\P_f^{X}$.

\emph{Step 1.} 
In view of the Markov chain structure given in \eqref{eq : markov chain structure}, the probability measure $d\P_\pi(x,u,y)$ disintegrates as $d\P_K^{Y|(X,U)=(x,u)}d\P^X_f(x)d\P^U(u)d\pi(f)$. Using the Markov chain structure, the first term on the rhs of \eqref{eq : protocol induced chi-sq div} can be seen to equal
\begin{equation}\label{eq : conditional exp. form of chi-sq div}
  \underset{y \in \mathcal{Y}}{\overset{}{\sum}} \P^{Y|U=u}_0(y) \left(  \int \mathscr{L}_\pi(x) \frac{K(y|x,u)}{\P^{Y|U=u}_0(y)} dP_0(x) \right)^2 = \E_0^{Y|U=u} \E_0\left[ \mathscr{L}_\pi(x) \bigg| Y, U=u \right]^2.
\end{equation}
Decoupling the square in $X$ and using Fubini's theorem we can write the above display as
\begin{equation}\label{eq : numerator key ratio q notation}
\int \mathscr{L}_\pi(x_1) \mathscr{L}_\pi(x_2) q_u(x_1,x_2) d(P_0 \times P_0) (x_1,x_2),
\end{equation}
where by independence between the transcripts,
\begin{equation*}
q_u(x_1,x_2) := \underset{y \in \mathcal{Y}}{\overset{}{\sum}} \frac{K(y|x_1,u) K(y|x_2,u) }{ \P_0^{Y|U=u}(y) } = \underset{j=1}{\overset{m}{\Pi}} \left( \underset{y^{j} \in \mathcal{Y}^{j}}{\overset{}{\sum}}   \frac{K^j(y^{j}|x_1^{j},u) K^j(y^{j}|x_2^{j},u) }{ \P_0^{Y^{j}|U=u}(y^{j}) } \right).
\end{equation*}
Note that in the above display, $x_i^{j}$ and $y^{j}$ denote the projection of $x_i$ and $y$ on the coordinates indexed by $\{ (j-1)d + 1,\dots, jd \}$, respectively. In addition, let us denote by $\prod_{j=1}^m q_u^{j}(x_1^{j},x_2^{j})$ the rhs of the preceding display. Since $K$ is a Markov kernel, the function $q_u \in L_2(\R^{2dm}, P_0 \times P_0)$ is bounded and nonnegative. Furthermore,
\begin{align*}
\int q_u(x_1,x_2) \; dP_0 (x_1)= \underset{y \in \mathcal{Y}}{\overset{}{\sum}} \frac{ K(y|x_2,u) }{ \P_0^{Y|U=u}(y) }\int K(y|x_1,u)  \; dP_0 (x_1)=\underset{y \in \mathcal{Y}}{\overset{}{\sum}} K(y|x_2,u)=1,
\end{align*}
similarly $\int q_u(x_1,x_2) \; dP_0 (x_2) = 1$,
\begin{equation}\label{eq : q mean condition}
\int x_i q_u(x_1,x_2) d(P_0 \times P_0) (x_1,x_2) = \int x_i  dP_0 (x_i) =   0 \in \R^{md}
\end{equation}
 for $i=1,2$, and 
\begin{equation}\label{eq : q cov condition}
\int \left( \begin{matrix}
x_1 \\
x_2
\end{matrix}\right)
 \left( \begin{matrix}
x_1^\top & x_2^\top
\end{matrix}\right)
 q_u(x_1,x_2) d(P_0 \times P_0) (x_1,x_2) =: \Sigma \in \R^{2md \times 2md},
\end{equation}
where the former display can be seen to follow by the law of total expectation,  $\Sigma = \text{Diag}\left(\Sigma^1,\dots,\Sigma^m\right) \in \R^{2md}$ for
\begin{equation*}
\Sigma^j := \left(\begin{matrix}
\frac{m}{n} I_{d} & \Xi_u^j \\
\Xi_u^j &  \frac{m}{n}  I_{d},
\end{matrix}\right)
\end{equation*}
with
\begin{equation*}
\Xi_u^j := \E_0^{Y^{j}|U=u} \E_0\left[ {X}^{j} \bigg| Y, U=u \right] \E_0\left[ {X}^{j} \bigg| Y^{j}, U=u \right]^\top.
\end{equation*} 
Writing $\mathscr{L}^{j}_\pi := \int \frac{d\P^{X^{j}}_f}{d\P^{X^{j}}_0} d\pi(f)$,  \eqref{eq : quantity post factorization} can be seen to equal
\begin{equation}\label{eq : denominator key ratio q notation}
\underset{j=1}{\overset{m}{\Pi}} \int \mathscr{L}^{j}_\pi(x_1^{j}) \mathscr{L}^{j}_\pi(x_2^{j}) q_u^{j}(x_1^{j},x_2^{j}) d(P_0 \times P_0) (x_1^{j},x_2^{j}),
\end{equation}

Lemma \ref{lem : key gaussian kernel maximizer} below applies to the ratio between \eqref{eq : numerator key ratio q notation} and \eqref{eq : denominator key ratio q notation} whenever $\pi$ is chosen to be centered Gaussian. The lemma yields that the aforementioned ratio is maximized when $q_u(x_1,x_2) d(P_0 \times P_0) (x_1,x_2)$ is a Gaussian distribution on $\R^{2md}$ with covariance $\Sigma$, where the maximization is among all choices of $q_u$ such that $q_u$ is nonnegative, bounded and satisfying \eqref{eq : q mean condition}-\eqref{eq : q cov condition}. {Deliberation and proof of the lemma is deferred to Section \ref{sec : proof gaussian kernel lemma} and the Supplementary Material to the article.} For $\pi$ a centered Gaussian distribution on $\R^d$, the above lemma applies with $k = 2d$, $\sigma^2 =m/n$, we obtain that the ratio between \eqref{eq : numerator key ratio q notation} and \eqref{eq : denominator key ratio q notation} is bounded above by
\begin{equation}\label{eq : gaussian upper bound ratio lb prf}
\frac{\int \mathscr{L}_\pi(x_1) \mathscr{L}_\pi(x_2) dN(0,\Sigma)(x_1,x_2)}{\underset{j=1}{\overset{m}{\Pi}} \int \mathscr{L}^{j}_\pi(x_1^{j}) \mathscr{L}^{j}_\pi(x_2^{j})  dN(0,\Sigma^j) (x_1^{j},x_2^{j})}.
\end{equation}
Combining the result of the lemma with the bound
\begin{equation}\label{eq : factorized problem ip}
\underset{j=1}{\overset{m}{\Pi}} \E_0^{Y^{j}|U=u} \E_0 \left[ \mathscr{L}_\pi \left({X}^{j}\right) \bigg| Y^{j}, U=u \right]^2 \leq \underset{j=1}{\overset{m}{\Pi}} \E_0^{X^{j}|U=u} \left[ \mathscr{L}_\pi \left({X}^{j}\right)^2 \right]
\end{equation}
following from Jensen's inequality, we obtain that 
\begin{align}
\E_0^{Y|U=u} \left( \frac{\P_\pi^{Y|U=u}}{\P_0^{Y|U=u}}(Y) \right)^2 &\leq \frac{\int \mathscr{L}_\pi(x_1) \mathscr{L}_\pi(x_2) dN(0,\Sigma) (x_1,x_2) }{\underset{j=1}{\overset{m}{\Pi}} \int \mathscr{L}^{j}_\pi(x_1^{j}) \mathscr{L}^{j}_\pi(x_2^{j}) dN(0,\Sigma) (x_1,x_2)}\nonumber\\
&\qquad\qquad\times  \underset{j=1}{\overset{m}{\Pi}} \E_0^{X^{j}|U=u} \left[ \mathscr{L}_\pi \left({X}^{j}\right)^2 \right].\label{eq:help:lem4.1}
\end{align}
By the block diagonal matrix structure of $\Sigma$, the denominator in the first factor of the rhs  of  \eqref{eq:help:lem4.1} satisfies
\begin{align*}
\underset{j=1}{\overset{m}{\Pi}} \int \mathscr{L}^{j}_\pi(x_1^{j}) \mathscr{L}^{j}_\pi(x_2^{j}) dN(0,\Sigma) (x_1,x_2)
&= \underset{j=1}{\overset{m}{\Pi}}  \int e^{ \frac{n}{2m} \left(\frac{n}{m}\|\sqrt{\Sigma^j} \left({f},  {g}\right) \|_2^2-  \|\left({f},  {g}\right)\|_2^2 \right)} d (\pi \times\pi)\left({f}, {g}\right)\\
&= \underset{j=1}{\overset{m}{\Pi}} \int e^{\frac{n^2}{m^2} f^\top  \Xi^j_u g} d(\pi \times  \pi)(f,g)\\
&\geq \underset{j=1}{\overset{m}{\Pi}}e^{ \frac{n^2}{m^2}  \int f^\top  \Xi^j_u g  \,d(\pi \times  \pi)(f,g)}=1.
\end{align*}
Similarly, the numerator is equal to
\begin{equation*}%\label{eq : numerator after brascamp lieb}
\int \mathscr{L}_\pi(x_1) \mathscr{L}_\pi(x_2) dN(0,\Sigma) (x_1,x_2)=\int e^{\frac{n^2}{m^2} f^\top \sum_{j=1}^{m} \Xi^j_u g} d(\pi \times  \pi)(f,g).
\end{equation*}
Combining the above displays (i.e. \eqref{eq : protocol induced chi-sq div} and the last three displays), we obtain that
\begin{equation}\label{eq : continue to bound this display}
D_{\chi^2}(\P^{Y|U=u}_{0,K} ; \P^{Y|U=u}_{\pi,K}) \leq \underset{j=1}{\overset{m}{\Pi}} \E_0^{X^{j}|U=u} \left[ \mathscr{L}_\pi \left({X}^{j}\right)^2 \right] \cdot \int e^{\frac{n^2}{m^2} f^\top \sum_{j=1}^{m} \Xi^j_u g} d(\pi \times  \pi)(f,g)-1.
\end{equation}

\emph{Step 2.} What is left to show in this step, is that for $\pi = N(0,\Gamma)$, $\Gamma \in \R^{d \times d}$ can be chosen such that the rhs of the previous display is small enough whilst also ensuring that $\pi(H_\rho^c)$ is controlled whenever $\rho^2$ satisfies \eqref{eq : pub coin rho lb}-\eqref{eq : priv coin rho lb} for $c_\alpha$ depending only on $\alpha \in (0,1)$.

For a given $c_\alpha > 0$, set $\epsilon :=  \frac{\rho}{c_\alpha^{1/4} d^{1/2}}$ and $\Gamma := \epsilon^2 \bar{\Gamma}$ for some $ \bar{\Gamma}\in\mathbb{R}^{d\times d}$ to be specified later, separately for the private and public coin protocols. The remaining mass $\pi(H_\rho)$ can now be seen to equal
\begin{equation*}
\pi(f : \|f\|_2^2 \leq \rho^2 ) = \text{Pr}\left(  Z^\top \bar{\Gamma} Z  \leq \sqrt{c_\alpha}d \right),
\end{equation*}
where $Z$ is a $d$-dimensional standard normal vector. If $\bar{\Gamma}$ is symmetric, idempotent and has rank (proportional to) $d$, the concentration inequality in Lemma \ref{lem : Chernoff-Hoeffding bound chisq} yields that the probability on the rhs of the above display can be made arbitrarily small for small enough choice of $c_\alpha > 0$. 

We now proceed to bound the first factor {in the product} on the rhs of \eqref{eq : continue to bound this display}, which for a positive semi-definite choice of $\bar{\Gamma}$ equals
\begin{equation*}
\underset{j=1}{\overset{m}{\Pi}} \int \E_0^{X^{j}|U=u}  \exp \left( \frac{n}{m} (\sqrt{\bar{\Gamma}}(f+g))^\top {X}^{j} - \frac{n}{2m} \|\sqrt{\bar{\Gamma}}f\|_2^2 - \frac{n}{2m} \|\sqrt{\bar{\Gamma}}g\|_2^2 \right) dN(0, {\epsilon^2}  I_{2d}) (f,g).
\end{equation*}
By direct computation, the latter display equals
\begin{equation*}
\underset{j=1}{\overset{m}{\Pi}}  \int \exp \left( \frac{n \epsilon^2}{m} z^\top \bar{\Gamma} z'  \right) dN(0,I_{2d}) (z,z'). 
\end{equation*}
By applying the moment generating function of the Gaussian chaos,  e.g. Lemma 6.2.2 in \cite{vershynin_high-dimensional_2018} to the above display and using that $\rho^2$ satisfies \eqref{eq : pub coin rho lb} or \eqref{eq : priv coin rho lb}, we obtain that for $ \frac{n \epsilon^2}{m} \|\bar{\Gamma}\| \lesssim \frac{ n \rho^2}{c_\alpha^{1/2} m \sqrt{d}} \leq \sqrt{c_\alpha/m}\leq\sqrt{c_{\alpha}}$ small enough, where $\|\cdot\|$ denotes the spectral norm of a matrix, there exists a constant $C \geq \|\bar{\Gamma}\|^2/d$ such that 
\begin{equation}\label{eq : factorized final bound}
 \underset{j=1}{\overset{m}{\Pi}} \E_0^{X^{j}|U=u} \left[ \mathscr{L}_\pi \left({X}^{j}\right)^2 \right] \leq \exp \left(C c_\alpha^{-1} \frac{n^2 \rho^{4}}{m d}  \right)\leq \exp(C c_\alpha).
\end{equation}
The exponent can be made arbitrarily close to zero per choice of $c_\alpha>0$.

We switch our attention now to the second factor {in the product} on the rhs of \eqref{eq : continue to bound this display}, which we bound by applying Lemma 6.2.2 in \cite{vershynin_high-dimensional_2018} once more,
\begin{equation}\label{eq : Xi trace bound}
\int e^{\frac{n^2}{m^2} (\sqrt{\Gamma}f)^\top \sum_{j=1}^{m}  \Xi^j_u (\sqrt{\Gamma}g)} dN(0, \epsilon^2  I_{2d}) (f,g) \leq e^{ C  \frac{n^4 \epsilon^4}{m^4 }  \text{Tr} \left((\sqrt{\bar{\Gamma}}^\top \Xi_u \sqrt{\bar{\Gamma}})^2 \right) },
\end{equation}
whenever 
\begin{equation}\label{eq : to check before 6.2.2}
\frac{n^2 \epsilon^2}{ m^2} \|\sqrt{\bar{\Gamma}}^\top \Xi_u \sqrt{\bar{\Gamma}} \|
\end{equation}
 is small enough. 

It remains to choose a symmetric, idempotent positive semi-definite $\bar{\Gamma}$ that sufficiently bounds \eqref{eq : to check before 6.2.2} and to combine the above displays providing the stated lower bound for the testing risk. For the exact choice of $\bar{\Gamma}$, we distinguish between the public coin and private coin cases. In both cases, we employ the data processing inequalities of Lemma \ref{lem : trace of fisher info} and Lemma \ref{lem : strict DPI}, which yield that
\begin{equation}\label{eq : data processing consequence2}
\text{Tr}(\Xi_u) = \underset{j=1}{\overset{m}{\sum}} \text{Tr}(\Xi^j_u)  \leq  \min \{ 2 \log 2 \cdot \frac{b}{d}, 1 \} \frac{m^2d}{n}.
\end{equation} 

\emph{The public coin case:} In this case, it suffices to take $\bar{\Gamma} = I_d$, which is trivially symmetric, idempotent and positive semi-definite. By Lemma \ref{lem : strict DPI}, $ \frac{n}{m} \Xi_u^j  \leq I_d$, so \eqref{eq : to check before 6.2.2} holds as well for this choice of $\bar{\Gamma}$:
\begin{equation*}
\frac{n^2 \epsilon^2}{ m^2} \| \Xi_u  \|\leq n\epsilon^2  \leq \frac{n \rho^2}{\sqrt{c_\alpha} d}\leq \sqrt{c_{\alpha}},
\end{equation*}
where the second to last inequality holds for $\rho^2$ satisfying \eqref{eq : pub coin rho lb}.

It remains to combine our results and provide a lower bound for the testing risk. Note that
\begin{align*}
\text{Tr}( \Xi^2_u) = \|\Xi_u\| \, \text{Tr} (\Xi_u) \leq \frac{m^2}{n} \text{Tr} (\Xi_u) \lesssim \frac{(b \wedge d) m^4}{n^2},
\end{align*}
where the last inequality follows from \eqref{eq : data processing consequence2}. Combining the above bound with assertions \eqref{eq : Xi trace bound},  \eqref{eq : factorized final bound}, \eqref{eq : continue to bound this display}, \eqref{eq : protocol induced chi-sq div}, and \eqref{eq : sup inf public coin divergence lb}, $\epsilon^4 = c_{\alpha}^{-1}d^{-2}\rho^4$ and the fact that $\pi(H_\rho) \leq \alpha/2$, we obtain that 
\begin{align*}
\underset{T \in \cT_{pub}(b)}{\inf} \cR(H_\rho,T) &\geq 1 - \sqrt{2( e^{C (\frac{n^2 \rho^4}{c_\alpha md} + \frac{n^2 \rho^4 (b\wedge d)}{c_\alpha d^2})} -1)} - \pi(H_\rho^c)\\
&\geq 1-\sqrt{2(e^{2Cc_{\alpha}}-1)}-\alpha/2 > 1-\alpha,
\end{align*}
whenever $\rho^2$ satisfies \eqref{eq : pub coin rho lb} for $c_\alpha > 0$ small enough. This finishes the proof for the public coin case.

\emph{The private coin case:} Since without loss of generality we can assume that $U$ is degenerate in the private coin case, $\Xi_u = \Xi$ for $\P^U$-almost every $u$. The matrix $\Xi$ is positive definite and symmetric, therefore it possesses a spectral decomposition $V^\top \text{Diag}(\xi_1,\dots,\xi_d) V$. Without loss of generality, assume that $\xi_1 \geq \xi_2 \geq \ldots \geq \xi_d$ with corresponding eigenvectors $V = \left(\begin{matrix} v_1 & \dots & v_d \end{matrix}\right)$. Let $\check{V}$ denote the $d \times \lceil d/2 \rceil$ matrix $\left(\begin{matrix} v_{\lfloor d/2 \rfloor+1} & \dots & v_d \end{matrix}\right)$. The choice of prior may depend on $\Xi$, to see this, note the order of the supremum and infimum in \eqref{eq : sup inf public coin divergence lb} and the fact that $\Xi$ soley depends on the choice of kernel. To that extent, set $\bar{\Gamma} = \check{V} \check{V}^\top$. It holds that
\begin{align*}
    \text{Tr}(\check{V} \check{V}^\top) &= \underset{i=1}{\overset{d}{\sum}}  \underset{k=\lfloor d/2 \rfloor+1}{\overset{d}{\sum}} (v_k)_i^2 = \lceil d/2 \rceil.
\end{align*}
The choice $\Gamma = \epsilon^2 \bar{\Gamma}$ is thus seen to satisfy the conditions of symmetry and positive definiteness and is idempotent with rank $\lceil d/2 \rceil$. 

Since the eigenvalues are decreasingly ordered,
\begin{equation*}
\xi_{\lfloor d/2 \rfloor} \leq \frac{2}{d} \underset{i=1}{\overset{\lfloor d/2 \rfloor}{\sum}} \xi_i \leq \frac{2}{d} \text{Tr} ( \Xi ).
\end{equation*}
By orthogonality of the columns of $V$, $\check{V}^\top \Xi \check{V} = \text{Diag}(\xi_{\lfloor d/2 \rfloor+1},\dots,\xi_d)$. 
Combining this inequality with the last display and assertion  \eqref{eq : data processing consequence2}  we get that for $\rho^2$ satisfying \eqref{eq : priv coin rho lb} the term  \eqref{eq : to check before 6.2.2} can be made arbitrarily small  for small enough choice of $c_{\alpha}$, i.e.
\begin{align*}
\frac{n^2 \epsilon^2}{ m^2} \|\sqrt{\bar{\Gamma}}^\top \Xi_u \sqrt{\bar{\Gamma}} \| &\leq  \frac{n^2 \epsilon^2}{ m^2} \xi_{\lfloor d/2 \rfloor} \leq 2 \frac{n^2 \rho^2}{\sqrt{c_\alpha} d^2 m^2} \text{Tr} ( \Xi )\\
& \leq (4\log 2) \frac{n \rho^2 (b \wedge d) }{\sqrt{c_\alpha} d^2 }\leq (4\log 2)\sqrt{c_{\alpha}/d}.
\end{align*}

Finally, a similar argument will be used to bound the right hand side of \eqref{eq : Xi trace bound} and finally to provide a lower bound for the testing risk. Note that
\begin{align*}
\text{Tr} \big( (\sqrt{\bar{\Gamma}}^\top \Xi_u \sqrt{\bar{\Gamma}} )^2 \big) =\text{Tr} \left( (\check{V}^\top \Xi \check{V})^2 \right) = \underset{i=\lfloor d/2 \rfloor+1}{\overset{d}{\sum}} \xi_i^2 \leq d \xi_{\lfloor d/2 \rfloor}^2 \leq \frac{4}{d} \text{Tr} ( \Xi )^2,
\end{align*}
which implies in turn that
\begin{equation*}
\frac{n^4 \epsilon^4}{m^4 } \text{Tr} \left( (\check{V}^\top \Xi \check{V})^2 \right) \leq 4\frac{n^4 \rho^4}{c_\alpha m^4 d^3} \text{Tr} ( \Xi )^2 \leq 4\frac{n^2 \rho^4 (b \wedge d)^2}{c_\alpha d^3},
 \end{equation*}
where the last inequality follows from \eqref{eq : data processing consequence2}. Consequently, we have obtained that 
\begin{align*}
\underset{T \in \cT_{priv}(b)}{\inf} \cR(H_\rho,T) &\geq 1 - \sqrt{2(e^{C (\frac{n^2 \rho^4}{c_\alpha md} + \frac{n^2 \rho^4 (b\wedge d)^2}{c_\alpha d^3})} -1)}  - \pi(H_\rho^c) \\
&\geq 1-\sqrt{2(e^{2Cc_{\alpha}}-1)}-\alpha/2 > 1-\alpha,
\end{align*}
for $\rho^2$ satisfying \eqref{eq : priv coin rho lb} and $c_\alpha >0$ small enough.

\section{Lemma \ref{lem : key gaussian kernel maximizer}: Gaussian maximization}\label{sec : proof gaussian kernel lemma}

Before giving the detailed statement of the lemma below, we briefly contemplate on its aim and proof. The lemma bears a close connection to Brascamp-Lieb inequalities \cite{brascamp1976best, lieb_gaussian_1990, bennett2008brascamp}. Brascamp-Lieb type inequalities have appeared in context of information theory in the literature before, see e.g. \cite{carlen2009subadditivity, liu_brascamp-lieb_2016}, where Gaussian extremality is established for certain information theoretic optimization problems. Instead of the information theoretic entropy based route, we rely on the technique of \cite{lieb_gaussian_1990}. {The resulting lemma allows us to bound the ratio between \eqref{eq : numerator key ratio q notation} and \eqref{eq : denominator key ratio q notation}, {i.e.}
\begin{equation}\label{eq : global vs local posteriors}
\frac{\int \mathscr{L}_\pi(x_1) \mathscr{L}_\pi(x_2) q_u(x_1,x_2) d(P_0 \times P_0) (x_1,x_2)}{\underset{j=1}{\overset{m}{\Pi}} \int \mathscr{L}^{j}_\pi(x_1^{j}) \mathscr{L}^{j}_\pi(x_2^{j}) q_u^{j}(x_1^{j},x_2^{j}) d(P_0 \times P_0) (x_1^{j},x_2^{j})},
\end{equation}
by \eqref{eq : gaussian upper bound ratio lb prf}, i.e. a Gaussian distribution with matching mean and covariance. Consequently, we obtain a quadratic form in the covariance that we would otherwise obtain via a Taylor expansion.  That such a quadratic form does not follow through more standard means such as Taylor expansion is described in \cite{pmlr-v125-acharya20b}, Section 4. 

The proof of the lemma exploits the conjugacy between likelihood of the observation $X$ and the Gaussian prior on the parameter to obtain that a Gaussian distribution is in fact an extremal case.} {For reasons of space, we defer the proof to Section \ref{sec : supplement Brascamp-Lieb details} of the Supplementary Material.}

\begin{lemma}\label{lem : key gaussian kernel maximizer}
For $x \in \R^{mk}$, let $x^j \in \R^k$, $j=1,...,m$, denote the projection of $x$ on the coordinates $\{ (j-1)k + 1,\dots, jk \}$. Let $\Lambda\in\mathbb{R}^{k\times k}$ a positive definite symmetric matrix and $\Lambda^{\otimes m}=\text{Diag}(\Lambda,....,\Lambda)\in \mathbb{R}^{mk\times mk}$. For $h \in \R^k$, let $p_h$ denote the density of a $N( h,\Lambda)$ distribution with respect to the Lebesgue measure on $\R^{k}$ and let $p^m_h(x) := \Pi_{j=1}^m p_h(x^j)$. 
 Consider for some $M > 0$, $\cQ \equiv \cQ(M,\Sigma)$ the class of all nonnegative functions $ q \in L_\infty(\R^{mk})$ satisfying $\frac{q(x)}{\int q(x) p_0^m(x) dx} \leq M$ $P_0^m$-a.e., $\int x \, q(x) p_0^m(x) dx = 0$ and $\int xx^\top \, q(x) p^m_0(x)dx = \Sigma$. Furthermore, let $H$ a $N(0, \Upsilon)$-distributed random vector in $\R^k$. Then
\begin{equation*}
\underset{q \in \cQ}{\sup} \; \frac{\int \E^H \Pi_{j=1}^{m} \frac{ p_H}{p_0} \left(  x^j \right) \, q(x)  p_0^m(x)dx }{\int \Pi_{j=1}^{m} \E^H \frac{ p_H}{p_0} \left(  x^j \right) \, q(x) p_0^m(x)dx  } \leq \frac{\int \E^H \Pi_{j=1}^{m} \frac{ p_H}{p_0} \left(  x^j \right) dN(0,\Sigma)(x) }{\int \Pi_{j=1}^{m} \E^H \frac{ p_H}{p_0} \left(  x^j \right) dN(0,\Sigma)(x) }.
\end{equation*} 
\end{lemma}

%%%%%%%%%%%%%%%%%%%%%%%%%%%%%%%%%%%%%%%%%%%%%%
%% Support information (funding), if any,   %%
%% should be provided in the                %%
%% Acknowledgements section.                %%
%%%%%%%%%%%%%%%%%%%%%%%%%%%%%%%%%%%%%%%%%%%%%%
% \section*{Acknowledgements}
% 
% The first author was supported by ...
% 
% The second author was supported in part by ...

\textbf{Acknowledgements: } We would like to thank Elliot H. Lieb for a helpful comment regarding the proof of Lemma \ref{lem : key gaussian kernel maximizer}. This project has received funding from the European Research Council (ERC) under the European Union’s Horizon 2020 research and innovation programme (grant agreement No. 101041064).

%%%%%%%%%%%%%%%%%%%%%%%%%%%%%%%%%%%%%%%%%%%%%%
%% Supplementary Material, if any, should   %%
%% be provided in {supplement} environment  %%
%% with title and short description.        %%
%%%%%%%%%%%%%%%%%%%%%%%%%%%%%%%%%%%%%%%%%%%%%%
\begin{supplement}
\stitle{Supplementary Material to Optimal high-dimensional and nonparametric distributed testing under communication constraints}
\sdescription{In the supplement to this paper \cite{szabo2022nonparametric_supplement}, we present the detailed proofs for the main theorems in the paper ``Optimal high-dimensional and nonparametric distributed testing under communication constraints''.}
\end{supplement}

%%  MR numbers will be added by VTeX.                      %%
%%                                                         %%
%%  Use \cite{...} to cite references in text.             %%
%%                                                         %%
%%%%%%%%%%%%%%%%%%%%%%%%%%%%%%%%%%%%%%%%%%%%%%%%%%%%%%%%%%%%%

%% if your bibliography is in bibtex format, uncomment commands:
\bibliographystyle{imsart-number} % Style BST file (imsart-number.bst or imsart-nameyear.bst)
\bibliography{references.bib}       % Bibliography file (usually '*.bib')

%% or include bibliography directly:
% \begin{thebibliography}{}
% \bibitem{b1}
% \end{thebibliography}

%%%%%%%%%%%%%%%%%%%%%%%%%%%%%%%%%%%%%%%%%%%%%%
%% Multiple Supplementes:                     %%
%%%%%%%%%%%%%%%%%%%%%%%%%%%%%%%%%%%%%%%%%%%%%%
%\begin{supplement}
%\section{???}
%
%\section{???}
%
%\end{supplement}

\pagebreak

\begin{frontmatter}
\title{Supplementary Material to ``Optimal high-dimensional and nonparametric distributed testing under communication constraints''}
%\title{A sample article title with some additional note\thanksref{T1}}
\runtitle{Optimal distributed testing}
%\thankstext{T1}{A sample of additional note to the title.}

\begin{aug}
%%%%%%%%%%%%%%%%%%%%%%%%%%%%%%%%%%%%%%%%%%%%%%
%%Only one address is permitted per author. %%
%%Only division, organization and e-mail is %%
%%included in the address.                  %%
%%Additional information can be included in %%
%%the Acknowledgments section if necessary. %%
%%%%%%%%%%%%%%%%%%%%%%%%%%%%%%%%%%%%%%%%%%%%%%
\author[A]{\fnms{Botond} \snm{Szab\'{o}}\ead[label=e1]{botond.szabo@unibocconi.it}},
\author[B]{\fnms{Lasse} \snm{Vuursteen}\ead[label=e2]{
l.vuursteen@tudelft.nl}}
\and
\author[C]{\fnms{Harry} \snm{van Zanten}\ead[label=e3]{j.h.van.zanten@vu.nl
}}
%%%%%%%%%%%%%%%%%%%%%%%%%%%%%%%%%%%%%%%%%%%%%%
%% Addresses                                %%
%%%%%%%%%%%%%%%%%%%%%%%%%%%%%%%%%%%%%%%%%%%%%%
\address[A]{Department of Decision Sciences, Bocconi University,\\
Bocconi Institute for Data Science and Analytics (BIDSA), \printead{e1}}

\address[B]{Delft Institute of Applied Mathematics (DIAM), Delft University of Technology \printead{e2}}
\address[C]{Department of Mathematics, Vrije Universiteit Amsterdam \printead{e3}}
\end{aug}

\begin{abstract}

In this supplement, we present the detailed proofs for the main theorems in the paper ``Optimal high-dimensional and nonparametric distributed testing under communication constraints''.

\end{abstract}

\begin{keyword}[class=MSC2020]
\kwd[Primary ]{62G10}
\kwd{62F30}
\kwd[; secondary ]{62F03}
\end{keyword}

\begin{keyword}
\kwd{Distributed methods}
\kwd{Nonparametric}
\kwd{Hypothesis testing}
\kwd{Minimax optimal}
\end{keyword}

\end{frontmatter}
%%%%%%%%%%%%%%%%%%%%%%%%%%%%%%%%%%%%%%%%%%%%%%
%% Please use \tableofcontents for articles %%
%% with 50 pages and more                   %%
%%%%%%%%%%%%%%%%%%%%%%%%%%%%%%%%%%%%%%%%%%%%%%
%\tableofcontents

%%%%%%%%%%%%%%%%%%%%%%%%%%%%%%%%%%%%%%%%%%%%%%
%%%% Main text entry area:
\setcounter{section}{0}
\setcounter{equation}{0}
\renewcommand{\theequation}{S.\arabic{equation}}
\renewcommand{\thesection}{\Alph{section}}
\section{Auxilliary lemmas for finite dimensional Gaussian mean testing}

\subsection{Lemmas related to the lower bound (Theorem \ref{thm : detection lb})}

% Le Cam lemma
% Proof of Lemma \ref{}
Following the notation of Section \ref{sec:model} in the article, let $\P_f\equiv\P_{f,K}$ denote the joint distribution of $Y$, $U$ and $X$ where $X^{j}$ follows $N(f,\frac{{m}}{{n}}I_d)$ and $Y \sim \E_f^{X,U} K(\cdot | X,U) =: \P_{f,K}^Y$ for $f\in \R^d$. Let $\pi$ be a probability distribution on $\mathbb{R}^d$ and define the mixture distribution $P_\pi$ by
\begin{equation*}
P_\pi(A) = \int P_f(A) d\pi(f),
\end{equation*} 
where $P_f=\P_{f}^X$.
\begin{lemma}\label{lem : dist le cam bound}[Le Cam bound]
For any distribution on $U$, it holds that
\begin{equation*}
\underset{\varphi,K}{\inf}  \left( \E_{0,K}^Y \varphi  + \underset{f \in H_\rho}{\sup} \, \E_{f,K}^Y ( 1 - \varphi) \right) \geq \underset{K}{\inf} \left( \underset{\pi}{\sup} ( 1 - \|\P_{0,K}^Y - \P_{\pi,K}^Y \|_{TV}) - \pi(H_\rho^c)\right), 
\end{equation*}
where 
\begin{itemize}
    \item the infimum on the lhs is taken over all Markov kernels $K:\, 2^{\cY} \times \cX \times \cU \to [0,1]$ in a suitable way and maps $\varphi : \cY \to \{0,1\}$, 
    \item the infimum on the rhs is over the same class od Markov kernels , 
    \item the supremum on the rhs is over all prior distributions $\pi$ on $\mathbb{R}^d$.
\end{itemize}
\end{lemma}
\begin{proof} 
  It trivially holds that for any $\varphi':\,\cY\mapsto \{0,1\}$,
\begin{equation*}
\left( \E_{0,K}^Y \varphi'(Y)  + \underset{f \in H_\rho}{\sup} \E_{f,K}^Y ( 1 - \varphi'(Y)) \right)  \geq \underset{\varphi}{\inf}  \left( \E_{0,K}^Y \varphi(Y)  + \underset{f \in H_\rho}{\sup} \E_{f,K}^Y ( 1 - \varphi(Y)) \right), 
\end{equation*}
where the infimum is over all $\varphi : \cY \mapsto \{0,1\}$. Furthermore, for any prior distribution $\pi$ on $\mathbb{R}^d$ it holds that
\begin{align}
\underset{f \in H_\rho}{\sup} \, \E_{f,K}^Y ( 1 - \varphi(Y))  &\geq \int_{\{f \in H_\rho \}} \E_{f,K}^Y( 1 - \varphi(Y)) d\pi(f) \nonumber \\
&\geq \int  \E_{f,K}^Y (1 - \varphi(Y)) d\pi(f) - \pi(H_\rho^c).\label{eq : type II error minus mass lb}
\end{align}
Hence the rhs of the second last display is further bounded from below by
\begin{equation*}
\inf_{\varphi} \big(\E_{0,K}^Y \varphi(Y)  + \E_{\pi,K}^Y( 1 - \varphi(Y)) - \pi(H_\rho^c)\big)
\end{equation*}
for all prior distributions $\pi$ on $\mathbb{R}^d$. For any $\varphi$, write $A_\varphi = \varphi^{-1}( \{ 0 \} )$ and note that
\begin{align*}
\P_{0,K}^Y \varphi(Y)  + \P_{\pi,K}^Y ( 1 - \varphi(Y) ) &= 1 - \big(\P_{0,K}^Y(Y \in A_\varphi) - \P_{\pi,K}^Y(Y \in A_\varphi)\big).
\end{align*}
By combining the above two displays we get that
\begin{align*}\label{bound_lecam}
\underset{\varphi}{\inf}  \left( \E_{0,K}^Y \varphi(Y)  + \underset{f \in H_\rho}{\sup} \, \E_{f,K}^Y ( 1 - \varphi(Y)) \right)  &
\geq 1 - \underset{A }{\sup} | \P_{0,K}^Y (A) - \P_{\pi,K}^Y (A) | - \pi(H_\rho^c).
\end{align*}
Since the above is true for any distribution $\pi$ on $\mathbb{R}^d$, the statement is true after taking the supremum over $\pi$ also. Since the above holds for an arbitrary Markov kernel $K:\, 2^{\cY} \times \cX \times \cU \to [0,1]$, the proof is concluded.
\end{proof}

\begin{lemma}\label{lem : strict DPI}
Let $\Xi_u^j$ denote the matrix 
\begin{align*}
  \Xi_u^j=\E_0^{Y^{j}}\E_0^{Y^{j}|U=u}\left[  {X}^{j} \bigg| Y^{j}, U=u \right] \E_0^{Y^{j}|U=u}\left[  {X}^{j} \bigg| Y^{j}, U=u \right]^\top.
\end{align*}
It holds that $\Xi_u^j \leq \frac{m}{n} I_{d}$. 
\end{lemma}
\begin{proof}
Let $v \in \R^d$, then
\begin{align*}
v^\top \Xi_u^j v &= \E_0^{Y^{j}}\E_0^{Y|U=u} \left[ v^\top {X}^{j} \bigg| Y^{j}, U=u \right] \E_0^{Y|U=u}\left[  ({X}^{j})^\top v \bigg| Y^{j}, U=u \right] \\
&=  \E_0^{Y^{j}}\E_0^{Y|U=u}\left[ v^\top {X}^{j} \bigg| Y^{j}, U=u \right]^2.
\end{align*}
Since the conditional expectation contracts the $L_2$-norm, we obtain that the latter is bounded by
\begin{equation*}
\E_0 v^\top {X}^{j} ({X}^{j})^\top v= \frac{m}{n} \|v\|_2^2,
\end{equation*}
which completes the proof.
\end{proof}

The previous lemma is in some sense a data processing inequality: the covariance matrix of $X|Y$ is strictly dominated by the covariance of the original process $X$. The following lemma extends this and shows that the trace of the covariance satisfies a different data processing inequality, where the loss of information due to $Y^{j}$ having only $b^{j}$ bits available is captured. When $b^{j} \ll d$, the latter data processing inequality is stronger than the one implied by Lemma \ref{lem : strict DPI}. The lemma below is essentially Theorem 2 of \cite{barnes2020lower} adapted to our setting, for which we provide a different proof that results in a smaller constant.

\begin{lemma}\label{lem : trace of fisher info}
Consider the matrix $\Xi_u^j$ given in Lemma \ref{lem : strict DPI}, then
\begin{equation}\label{eq : spdi trace fisher info}
\text{Tr}(\Xi_u^j) \leq 2 \log(2) \frac{m}{n} (\log_2 |\cY^j|).
\end{equation}
In particular, for $\log_2 |\cY^j| = b^{j}$,  
\begin{equation}\label{eq : tace spdi}
\text{Tr}(\Xi_u^j) \leq \left(2 \log(2) \frac{b^{j} }{d} \bigwedge 1 \right) \frac{md}{n}.
\end{equation}
\end{lemma}
\begin{proof}
We start by noting that under $\P_0$, ${X}^{j}$ follows a $N(0, \frac{m}{n} I_d)$ distribution. For any unit vector $v \in \R^d$ and $s \in \R$ this means that
\begin{equation*}
\E_0 e^{ s \langle {X}^{j}, v \rangle } \leq e^{\frac{s^2 m}{2n}}.
\end{equation*}
Furthermore, for arbitary $y\in\cY$,
\begin{align*}
\underset{y}{\overset{}{\sum}} \P^{Y^{j}|U=u} (y) \E_0 \left[ e^{ s \langle {X}^{j}, v \rangle } \big| Y^{j} = y, U=u \right] &\geq \P^{Y^{j}|U=u} (y) \E_0 \left[ e^{ s \langle {X}^{j}, v \rangle } \big| Y^{j} = y, U=u \right] \\
&\geq \P^{Y^{j}|U=u} (y)  e^{ s \E_0 \left[ \langle {X}^{j}, v \rangle \big| Y^{j} = y, U=u \right] }, 
\end{align*}
where the last line follows by Jensen's inequality. By combining the above displays we obtain that
\begin{equation*}
s \E_0 \left[ \langle {X}^{j}, v \rangle \big| Y^{j} = y, U=u \right] \leq \frac{s^2 m}{2n} - \log \P^{Y^{j}|U=u}(y)
\end{equation*}
for all $s \in \R$. Choosing $s = \frac{n}{m} \E_0 \left[ \langle{X}^{j}, v \rangle \big| Y^{j} = y, U=u \right]$, we have for any unit vector $v \in \R^d$,
\begin{equation*}
\E_0 \left[ \langle{X}^{j}, v \rangle \big| Y^{j} = y, U=u \right]^2 \leq - 2 \frac{m}{n} \log \P^{Y^{j}|U=u}(y).
\end{equation*}

Next define for $y \in \mathcal{Y}^{j}$ 
\begin{equation}\label{eq : covariance induced basis}
w_{1,y} = \frac{1}{\| \E_0 ( {X}^{j} | Y^{j} =y , U=u)\|_2 } \E_0 \left[ {X}^{j}| Y^{j} =y, U=u \right].
\end{equation}
Choose now $w_{2,y},\dots,w_{d,y}$ such that together with $w_{1,y}$ the vectors form an orthonormal basis for $\R^d$. We then have
\begin{align*}
\text{Tr}(\Xi^j_u) &= \underset{y\in \cY^{j}}{\overset{}{\sum}} \P^{Y^{j}|U=u} (y) \underset{i=1}{\overset{d}{\sum}} \E_0 \left[ \langle w_{i,y},X^{j}\rangle | Y^{j} =y, U=u \right]^2 \\
&= \underset{y\in \cY^{j}}{\overset{}{\sum}} \P^{Y^{j}|U=u} (y) \E_0 \left[ \langle w_{1,y},X^{j}\rangle | Y^{j} =y, U=u \right]^2\\
&\leq - 2 \frac{m}{n} \underset{y \in \cY^{j}}{\overset{}{\sum}}   \P^{Y^{j}|U=u} (y) \log \P^{Y^{j}|U=u}(y) \leq 2 \frac{m}{n} \log |\cY^{j}| ,
\end{align*}
where the last inequality follows from the fact that uniform distribution on $\cY^{j}$ maximizes the entropy on the lhs.  For the second statement note that by construction $ \log |\cY^{j}|\leq b^{j}\log 2$. Furthermore in view of of Lemma \ref{lem : strict DPI}, $ \log |\cY^{j}|\leq dm/n$. Then the statement follows by combining the above upper bounds for $ \log |\cY^{j}|$ with the preceding display.
\end{proof}

\subsection{Lemmas {for the upper bound theorems in the} finite dimensional Gaussian mean {model}}

We state a slightly extended version of Lemma \ref{lem : binomial testing}.
\begin{lemma}\label{lem : binomial testing supplement}
Consider for $k,l \in \N$, $l \geq 2$, independent random variables $\{B^j_i : i=1,\dots,k,\;j=1,\dots,l\}$ with $B^j_i \sim \text{Ber}(p_i)$. If $p_i = 1/2$ for $i=1,\dots,k$, for each $\alpha \in (0,1)$ there exists $\kappa_\alpha > 0$ such that
\begin{equation*}
\text{Pr} \left( \bigg| \frac{1}{\sqrt{k} l} \underset{i=1}{\overset{k}{\sum}}  \left( \underset{j=1}{\overset{l}{\sum}} (B_i^j - \frac{1}{2} ) \right)^2 - \sqrt{k}/4 \bigg| \geq \kappa_\alpha \right) \leq \alpha.
\end{equation*}
On the other hand, for arbitrary $c_{\alpha,n}>0$,
\begin{equation}\label{eq : binomial l2 divergence lb supplement}
\eta_{p,l,k} := \frac{l-1}{2\sqrt{k}} \underset{i=1}{\overset{k}{\sum}} \left( p_i - \frac{1}{2} \right)^2 \geq c_{\alpha,n},
\end{equation}
 it holds that
\begin{equation}\label{eq : binomial lem type II bound supplement}
\text{Pr} \left( \bigg| \frac{1}{\sqrt{k} l} \underset{i=1}{\overset{k}{\sum}}  \left( \underset{j=1}{\overset{l}{\sum}} (B_i^j - \frac{1}{2} ) \right)^2 - \sqrt{k}/4 \bigg| \leq c_{\alpha,n} \right) \leq  \frac{1/2+16\eta_{p,l,k}/\sqrt{k}}{\eta_{p,l,k}^2} .
\end{equation}
\end{lemma}

\begin{proof}
The LHS in the event having bounded variance: a straightforward computation (using that for $B_i^j\sim \text{Bern}(p_i)$, the central fourth moment is $E(B_i^j-p_i)^4=p_i(1-p_i)(1-3p_i(1-p_i))\leq 1/16$ and $\text{Var}(X)\leq EX^2$) yields
\begin{align}
\E \Big[ \frac{1}{\sqrt{k} l} \underset{i=1}{\overset{k}{\sum}}& \Big( \underset{j=1}{\overset{l}{\sum}} (B_i^j - \frac{1}{2} ) \Big)^2  - \frac{\sqrt{k}}{4} \Big]^2
=\frac{1}{kl^2}\underset{i=1}{\overset{k}{\sum}}\text{Var}\Big[  \Big( \underset{j=1}{\overset{l}{\sum}} (B_i^j - 1/2 ) \Big)^2 \Big]\nonumber\\
 &\leq \frac{1}{l^2} \sum_{j=1}^l\E (B_i^j - 1/2)^4+ \frac{1}{l^2} \sum_{j=1}^l\big(\E (B_i^j - 1/2 )^2\big)^2 \leq  1/8,\label{eq:help:lemA4}
\end{align}
after which Chebyshev's inequality yields the first statement. 

We turn to the second statement. Adding and subtracting $p_i$ and expanding the square, the lhs of the display in the lemma can be written as
\begin{align}\label{eq : binomial_test type 2 square exp}
\text{Pr} \left( \bigg| \frac{1}{\sqrt{k}l} \underset{i=1}{\overset{k}{\sum}} \left( \underset{j=1}{\overset{l}{\sum}}  B^j_i - lp_i \right)^2 - \mu_p + \frac{l-1}{\sqrt{k}} \underset{i=1}{\overset{k}{\sum}} \left( p_i - \frac{1}{2} \right)^2 + \zeta \bigg| \leq c_{\alpha,n}\right) 
\end{align}
where 
\begin{equation*}
 \mu_p := \frac{1}{\sqrt{k}} \underset{i=1}{\overset{k}{\sum}} p_i(1-p_i)  \;\; \text{ and } \;\; \zeta := \frac{2}{\sqrt{k}} \underset{i=1}{\overset{k}{\sum}} \left( p_i - \frac{1}{2} \right) \left( \underset{j=1}{\overset{l}{\sum}}  B^j_i - lp_i\right).
\end{equation*}
The first term in the event of \eqref{eq : binomial_test type 2 square exp} has mean $ \mu_p $ and variance (by the same computations as in \eqref{eq:help:lemA4}) 
\begin{align*}
\text{Var} \Big[ \frac{1}{\sqrt{k}l} \underset{i=1}{\overset{k}{\sum}} \Big( \underset{j=1}{\overset{l}{\sum}}  B^j_i - lp_i \Big)^2 \Big]
 &=  \frac{1}{kl^2}\sum_{i=1}^k \text{Var}\Big[ \Big(\sum_{j=1}^l B_i^j - lp_i  \Big)^2\Big] \leq 1/8.
\end{align*}
The term $\zeta$ has mean $0$ and
\begin{align*}
\text{Var}(\zeta) = \frac{4l}{k} \underset{i=1}{\overset{k}{\sum}} (p_i - \frac{1}{2})^2 p_i(1-p_i) \leq \frac{l}{k} \underset{i=1}{\overset{k}{\sum}} \left(p_i - \frac{1}{2}\right)^2.
\end{align*}
Applying the reverse triangle inequality and condition \eqref{eq : binomial l2 divergence lb supplement}, the probability in \eqref{eq : binomial_test type 2 square exp} is bounded from above by
\begin{align*}
&\text{Pr} \Big[ \Big| \frac{1}{\sqrt{k}l} \underset{i=1}{\overset{k}{\sum}} \Big( \underset{j=1}{\overset{l}{\sum}}  B^j_i - lp_i \Big)^2 - \mu_p \Big| + | \zeta |  \geq  \frac{l-1}{2\sqrt{k}} \underset{i=1}{\overset{k}{\sum}} \Big( p_i - \frac{1}{2} \Big)^2 \Big]\\
&\qquad\qquad \leq  \text{Pr} \Big[\Big| \frac{1}{\sqrt{k}l} \underset{i=1}{\overset{k}{\sum}} \Big( \underset{j=1}{\overset{l}{\sum}}  B^j_i - lp_i \Big)^2 - \mu_p \Big|  \geq \eta_{p,l,k}/2\Big]
+\text{Pr} \Big[ | \zeta |\geq  \eta_{p,l,k}/2\Big]\\
&\qquad\qquad\leq \frac{1/8}{(\eta_{p,l,k}/2)^2}+\frac{2lk^{-1/2}\eta_{p,l,k}/(l-1)}{(\eta_{p,l,k}/2)^2}\leq \frac{1/2+16\eta_{p,l,k}/\sqrt{k} }{\eta_{p,l,k}^2},
\end{align*}
where the last line follows by Chebyshev's inequality. 

\end{proof}

Next we provide another version of the above lemma, with the sum over the index $i$ moved inside of the square.

\begin{lemma}\label{lem : binomial testing supplement_version2}
Consider for $k,l \in \N$, $l \geq 2$, independent random variables $\{B^j_i : i=1,\dots,k; \, j=1,\dots,l\}$ with $B^j_i \sim \text{Ber}(p_i)$. If $p_i = 1/2$ for $i=1,\dots,k$, for each $\alpha \in (0,1)$ there exists $\kappa_\alpha > 0$ such that
\begin{equation*}
\text{Pr} \left( \bigg| \frac{1}{lk} \left( \underset{i=1}{\overset{k}{\sum}} \underset{j=1}{\overset{l}{\sum}}  (B_i^j - \frac{1}{2} ) \right)^2 - 1/4 \bigg| \geq \kappa_\alpha \right) \leq \alpha.
\end{equation*}
On the other hand, if $p_i\geq 1/2$ for all $i=1,...,k$ and for arbitrary $c_{\alpha,n}>0$
\begin{equation}\label{eq : binomial l2 divergence lb supplement_alternative}
\eta_{p,l,k}' := \frac{l-1}{2k} \left(  \underset{i=1}{\overset{k}{\sum}} (p_i - \frac{1}{2}) \right)^2 \geq c_{\alpha,n}
\end{equation}
 it holds that
\begin{equation*}
\text{Pr} \left( \bigg| \frac{1}{k l}  \left( \underset{i=1}{\overset{k}{\sum}} \underset{j=1}{\overset{l}{\sum}}  (B_i^j - \frac{1}{2} ) \right)^2 - 1/4 \bigg| \leq c_{\alpha,n} \right) \leq  \frac{1/2+16\eta_{p,l,k}'/k}{(\eta_{p,l,k}')^2} .
\end{equation*}
\end{lemma}

\begin{proof}
The lhs in the event having bounded variance: by the same arguments as in \eqref{eq:help:lemA4} we have
\begin{align*}
&\E \Big[ \frac{1}{lk} \Big( \underset{i=1}{\overset{k}{\sum}} \underset{j=1}{\overset{l}{\sum}} (B_i^j - \frac{1}{2} ) \Big)^2  - \frac{1}{4} \Big]^2 \leq  1/8,
\end{align*}
after which Chebyshev's inequality yields the first statement. 

We turn to the second statement. Adding and subtracting $p_i$ and expanding the square, the lhs of the display in the lemma can be written as
\begin{align}\label{eq : binomial_test type 2 square exp2}
\text{Pr} \left( \bigg| \frac{1}{lk}  \left( \underset{i=1}{\overset{k}{\sum}}\underset{j=1}{\overset{l}{\sum}}   (B^j_i - p_i) \right)^2 - \mu_p' + \frac{l-1}{k}\left(  \underset{i=1}{\overset{k}{\sum}} (p_i - \frac{1}{2}) \right)^2 + \zeta \bigg| \leq c_{\alpha,n} \right), 
\end{align}
where 
\begin{equation*}
 \mu_p' := 1/4-\frac{1}{k} \Big(\underset{i=1}{\overset{k}{\sum}} (p_i-\frac{1}{2})\Big)^2   \;\; \text{ and } \;\; \zeta := \frac{2}{k}  \left(  \underset{i=1}{\overset{k}{\sum}} (p_i - \frac{1}{2}) \right) \left( \underset{i=1}{\overset{k}{\sum}}\underset{j=1}{\overset{l}{\sum}}   (B^j_i - p_i)\right).
\end{equation*}
Next we note that in view of the assumption $p_i\geq 1/2$ we have that 
\begin{align*}
  \mu_p'\leq 1/4- \frac{1}{k} \underset{i=1}{\overset{k}{\sum}} (p_i-\frac{1}{2})^2=\frac{1}{k} \underset{i=1}{\overset{k}{\sum}} p_i(1-p_i)=:\mu_p. 
\end{align*}
The first term in the event of \eqref{eq : binomial_test type 2 square exp2} has mean $ \mu_p $ and variance (by the same computations as in \eqref{eq:help:lemA4}) 
\begin{align*}
\text{Var} \Big[ \frac{1}{lk} \Big( \underset{i=1}{\overset{k}{\sum}} \underset{j=1}{\overset{l}{\sum}}   (B^j_i - p_i) \Big)^2 \Big]
 &=  \frac{1}{l^2k^2}\text{Var}\Big[ \Big(\sum_{j=1}^l \sum_{i=1}^k  (B_i^j - p_i)  \Big)^2\Big] \leq 1/8.
\end{align*}
The term $\zeta$ has mean $0$ and
\begin{align*}
\text{Var}(\zeta) = \frac{4l}{k^2}\left(\underset{i=1}{\overset{k}{\sum}}  (p_i - \frac{1}{2}) \right)^2   \underset{i=1}{\overset{k}{\sum}} p_i(1-p_i) \leq  \frac{l}{k^2}\left(\underset{i=1}{\overset{k}{\sum}}  (p_i - \frac{1}{2}) \right)^2  .
\end{align*}
Applying the reverse triangle inequality, condition \eqref{eq : binomial l2 divergence lb supplement_alternative} and the inequality $\mu_p\geq\mu_p'$, the probability in \eqref{eq : binomial_test type 2 square exp2} is bounded from above by
\begin{align*}
&\text{Pr} \Big[ \Big| \frac{1}{kl} \Big( \underset{i=1}{\overset{k}{\sum}} \underset{j=1}{\overset{l}{\sum}}   (B^j_i - p_i) \Big)^2 - \mu_p \Big| + | \zeta |  \geq  \frac{l-1}{2}\Big(  \underset{i=1}{\overset{k}{\sum}} p_i - \frac{1}{2} \Big)^2 \Big]\\
&\qquad\qquad \leq  \text{Pr} \Big[\Big| \frac{1}{kl} \Big( \underset{i=1}{\overset{k}{\sum}} \underset{j=1}{\overset{l}{\sum}}   (B^j_i - p_i) \Big)^2 - \mu_p \Big|  \geq \eta_{p,l,k}'/2\Big]
+\text{Pr} \Big[ | \zeta |\geq  \eta_{p,l,k}'/2\Big]\\
&\qquad\qquad\leq \frac{1/8}{(\eta_{p,l,k}'/2)^2}+\frac{2lk^{-1}\eta_{p,l,k}'/(l-1)}{(\eta_{p,l,k}'/2)^2}\leq \frac{1/2+16\eta_{p,l,k}'/k }{\eta_{p,l,k}'^2},
\end{align*}
where the last line follows by Chebyshev's inequality. 

\end{proof}

{Next we provide the lemmas used in Section  \ref{sec : upper bounds}, providing guarantees for the testing procedures $T_{\text{I}}$, $T_{\text{II}}$ and $T_{\text{III}}$, proposed in subsections \ref{ssec : private coin low-budget}, \ref{ssec : public coin test} and \ref{ssec : private coin high-budget}, respectively.}

\begin{lemma}\label{lem : T_I proof}
For each $\alpha \in (0,1)$, there exist constants $\kappa_\alpha, C_\alpha, M_\alpha, D_0 >0$ such that for $m \geq M_\alpha$ and $d \geq D_0$ it holds that
\begin{equation*}
\cR(H_\rho,T_{\text{I}}) \leq \alpha,
\end{equation*}
whenever $\rho^2 \geq C_\alpha \frac{\sqrt{md}}{n}$.
\end{lemma}
\begin{proof}[Proof of Lemma \ref{lem : T_I proof}]
Under the null hypothesis the random variables $Y_{\text{I}}^{j}\sim^{iid}\text{Bern}(1/2)$. Next we shall apply Lemma \ref{lem : binomial testing} with $k = 1$, and $l = m$. By the first statement of the lemma, we obtain that there exists ${\kappa}_\alpha > 0$ such that $\P_0 T_{\text{I}} \leq \alpha/2$.

We give an upper bound for the Type II error by using the second statement of the lemma, but before that we show that condition \eqref{eq : binomial l2 divergence lb} holds. Note that the law of total expectation yields
\begin{equation*}
\E_f Y_I^{j} = \E_f \E_f\left[ Y_I^{j} \big| S_{\text{I}}^{j} \right] = \E_f F_{\chi^2_d} \left(S_{\text{I}}^{j} \right)=\text{Pr}(S_{\text{I}}^{j}\geq W_d),
\end{equation*}
where $S_{\text{I}}^{j}$ is noncentral Chi-square distributed under $\P_f$ with $d$-degrees of freedom and noncentrality parameter $ \frac{n}{m}\|f\|_2^2$ and $W_d$ is an independent chi-square distributed random variable with $d$-degrees of freedom. Then Lemma 4 in \cite{szabo2022optimal} yields that 
\begin{equation}
\eta_{p,m,1}=\frac{m-1}{2}\left( \E_f Y_I^{j} - \frac{1}{2} \right)^2 \geq  \frac{m-1}{3200} \left(\frac{n\|f\|_2^2}{m\sqrt{d}} \bigwedge \frac{1}{2} \right)^2.\label{help:T_I} 
\end{equation}
whenever $d \geq D_0$ for some universal constant $D_0 > 0$. Consequently, as $\|f\|_2^2 \geq \rho^2 \geq C_\alpha \frac{\sqrt{m d}}{n}$, we obtain that condition \eqref{eq : binomial l2 divergence lb} is satisfied whenever $m \geq M_\alpha$ for some large enough $C_\alpha > 0$ and $M_\alpha > 0$. Therefore the Type II error is bounded by the rhs of \eqref{eq : binomial lem type II bound}, which is monotone decreasing in $\eta_{p,m,1}$ hence also in $C_{\alpha}$. Therefore by large enough choice of $C_{\alpha}$ the Type II error is bounded from above by $\alpha/2$.
\end{proof}

\begin{lemma}\label{lem : public coin test risk ub}
For each $\alpha \in (0,1)$, there exist constants $\kappa_\alpha, C_\alpha, M_\alpha>0$ such that for $m \geq M_\alpha$
\begin{equation*}
\cR(H_\rho,T_{\text{II}}) \leq \alpha,
\end{equation*}
whenever $\rho^2 \geq C_\alpha \frac{d}{n \sqrt{d \wedge b}}$.
\end{lemma}

\begin{proof}[Proof of Lemma \ref{lem : public coin test risk ub}]
{First note that it is sufficient to consider the case $b\leq d$ as one can simply take $b=b\wedge d$.} Then note that under $\P_f$, $ \sqrt{n/m} U X^{j}|U  \sim N_d( \sqrt{n/m} Uf,  I_d)$ by the rotational invariance of the Gaussian distribution. By linearity of the coordinate projection, conditionally on $U$,
\begin{equation*}
\mathbbm{1} \left \{ \left( \sqrt{n/m} U X^{j} \right)_i > 0 \right\} \overset{d}{=} \mathbbm{1} \left \{ \sqrt{n/m} (Uf)_i + Z > 0 \right\},
\end{equation*}
where $Z \sim N(0,1)$. As a consequence, the vector $S_{\text{II}}$ is conditionally on $U$ coordinate wise independent binomially distributed with parameters $m$ and $p_{f,U} \in [0,1]^b$ under $\P_f^{Y|U}$, where 
\begin{equation*}
(p_{f,U})_i = \Phi( \sqrt{n/m} (Uf)_i),
\end{equation*}
with $\Phi$ the standard normal cdf. Under the null hypothesis, $(S_{\text{II}})_i$ is $\text{Bin}(m,1/2)$ distributed since $p_{0,U} = (1/2,\dots,1/2)\in [0,1]^b$. Next we apply Lemma \ref{lem : binomial testing} with $k =b$ and $l = m$. By the first statement of the lemma, it follows that for $\kappa_\alpha$ large enough, $\P_0 T_{\text{II}} \leq \alpha / 2$. 

In order apply the second statement of the lemma, which yields that the Type II error is bounded by $\alpha/2$, it suffices to show that the event
\begin{equation*}
A = \Big\{ \frac{m-1}{2\sqrt{b}} \underset{i=1}{\overset{b}{\sum}} \Big( (p_{f,U})_i - \frac{1}{2}  \Big)^2 \geq N_\alpha  \Big\},
\end{equation*}
where $N_\alpha := \kappa_\alpha \vee \frac{16}{\alpha}$, occurs with $\P^U$-probability greater than $1 - \alpha/4$. Note that for this choice of $N_\alpha$, \eqref{eq : binomial l2 divergence lb} is satisfied on the event $A$ and the rhs of \eqref{eq : binomial lem type II bound} is smaller than $\alpha/4$. The Type II error is then bound by $\P_f T_{\text{II}} \leq \P_f T_{\text{II}} \mathbbm{1}_A + \P_f\mathbbm{1}_{A^c} \leq \alpha/2$.

We proceed to show that $\P_f\mathbbm{1}_{A^c} \leq \alpha/4$. By a standard bound on the Gaussian error function $x \mapsto 2\Phi(x) - 1$ (see Lemma \ref{lem : normal cdf divergence lb}),
\begin{equation*}
\left( \Phi(\sqrt{n/m} (Uf)_i ) - \frac{1}{2}  \right)^2 \geq \frac{1}{12} \min \left\{ \frac{n}{m} (Uf)_i^2,1 \right\},
\end{equation*}
which in turn implies that
\begin{align*}
\P^U \left( \frac{m-1}{2\sqrt{b}} \underset{i=1}{\overset{b}{\sum}} \left( (p_{f,U})_i - \frac{1}{2}  \right)^2 \leq N_\alpha \right) &\leq \P^U \left( \frac{m-1}{24\sqrt{b}} \underset{i=1}{\overset{b}{\sum}}  \min \left\{ \frac{n}{m} (Uf)_i^2, 1 \right\}  \leq N_\alpha \right). 
\end{align*}
Note that $Uf \overset{d}{=} \|f\|_2 (Z_1,\dots,Z_d)/\|Z\|_2$, where $Z=(Z_1,\dots,Z_d) \sim N(0,I_d)$ (see e.g.  Section 3.4 of \cite{vershynin_high-dimensional_2018}). Using that $\| f \|_2 \geq \rho$ and $\rho^2 \geq C_\alpha \frac{d}{n \sqrt{b}}$, the previous display is further bounded by
\begin{equation*}
\text{Pr} \left( \frac{m-1}{24\sqrt{b}} \underset{i=1}{\overset{b}{\sum}}  \min \left\{ C_\alpha \frac{d Z_i^2}{m 
 \sqrt{b}\|Z\|_2^2} , 1 \right\}  \leq N_\alpha \right).
\end{equation*}
Considering the intersection with the event $\{\|Z\|_2^2 \leq kd \}$ for some $k > 0$, the above display can be bounded by
\begin{equation*}
\text{Pr} \left( \underset{i=1}{\overset{b}{\sum}}  \min \{ Z_i^2 , C_\alpha^{-1} {m 
 \sqrt{b}k} \}  \leq  \frac{24{b} m k}{C_\alpha(m-1)}  N_\alpha \right)
+ \text{Pr} \left( \|Z\|_2^2 \geq k d \right).
\end{equation*}
For $k$ large enough (independent of $d$), the second term is less than $\alpha/8$. By Lemma \ref{lem : Gaussian maximum}, 
\begin{equation*}
\text{Pr} \left( \underset{1 \leq i \leq b}{\max} Z_i^2 \geq C_\alpha^{-1} {m 
 \sqrt{b}k} \right) \leq \frac{2b}{e^{C_\alpha^{-1} m \sqrt{b} k/4}}.
\end{equation*}
For large enough $M_\alpha \geq C_\alpha$, the condition $m \geq M_\alpha$ implies that the right hand side is less than $\alpha/8$. The first term in the second to last display is consequently bounded by
\begin{align*}
\text{Pr} \Big( \underset{i=1}{\overset{b}{\sum}}  Z_i^2 & \leq  \frac{24{b} m k}{C_\alpha(m-1)}  N_\alpha  \Big) + \text{Pr} \left( \underset{1 \leq i \leq b}{\max} Z_i^2 \geq C_\alpha^{-1} {m 
 \sqrt{b}k} \right) \\
 &\qquad\leq \text{Pr} \left( \underset{i=1}{\overset{b}{\sum}}  Z_i^2  \leq  \frac{24{b} m k}{C_\alpha(m-1)}  N_\alpha \right) + \alpha/8.
\end{align*}
For $m\geq M_{\alpha}\geq 25$ and by choosing $C_\alpha$ large enough such that the Chernoff-Hoeffding bound on the left tail of the chi-square distribution  (see Lemma \ref{lem : Chernoff-Hoeffding bound chisq}) can be applied to the first term of the preceding display we get that
\begin{equation}
\text{Pr} \left(  \underset{i=1}{\overset{b}{\sum}}  Z_i^2 \leq  \frac{25 kN_\alpha}{C_\alpha}b \right) \leq \exp \left( - b \frac{\frac{25k N_\alpha}{C_\alpha} - 1 - \log \left( \frac{25k N_\alpha}{C_\alpha} \right)}{2} \right) \leq \alpha/8,\label{eq:help:testII}
\end{equation}
finishing the proof of the lemma.
\end{proof}

\begin{lemma}\label{lem: private:large:budget}
For $\alpha \in (0,1)$, there exist constants $M_\alpha, C_\alpha > 0$ such that when $m \geq M_\alpha d^2/b^2$, the $b$-bit distributed private testing protocol $T_{\text{III}}$ given in \eqref{def:test3} satisfies
\begin{equation*}
\cR( H_\rho,T_{\text{III}}) \leq \alpha,
\end{equation*}
whenever $\rho^2 \geq C_\alpha \frac{d \sqrt{d}}{n b}$.
\end{lemma}
\begin{proof}
Fix an arbitrary $f \in H_\rho$ and define
\begin{equation}
\cJ = \{ i : \; 1 \leq i \leq d, \; \frac{n}{m} f_i^2 \geq 1 \}. \label{eq: def:J}
\end{equation}
By Lemma \ref{lem : private coin test risk ub I} below, the test $T_{\text{III}}^1$ given in \eqref{eq : large budget priv coin test I def} with $\kappa_\alpha, C_\alpha, M_\alpha > 0$ large enough satisfies
\begin{equation*}
\E_0 T_{\text{III}}^1\leq \alpha /6,\quad\text{and}\quad   \E_f (1 - T_{\text{III}}^1) \leq \alpha /6,
\end{equation*}
whenever 
\begin{equation}\label{eq : large f_i's dominate L2 norm}
\underset{i \notin \cJ}{\overset{}{\sum}} f_i^2 \geq \rho^2/2 \; \text{ or } \; \frac{mb}{d \sqrt{d}} > M_\alpha.
\end{equation}

Next we consider the case where \eqref{eq : large f_i's dominate L2 norm} does not hold. Then $M_\alpha\geq \frac{mb}{d \sqrt{d}}\geq  M_\alpha \frac{\sqrt{d}}{b}$, where the second inequality follows from the assumption of the lemma. This implies that $b \geq  \sqrt{d}$. Since $\frac{mb}{d \sqrt{d}} \leq M_\alpha$ and $m$ can be taken to be larger than arbitrary constant (otherwise we are in the non-distributed regime in which the minimax rate can be achieved locally), we can without loss of generality assume $d$ is larger than an arbitrary constant (depending only on $\alpha$), hence $b\geq\sqrt{d}\geq 2\log (d+1)$ and the  test $T_{\text{III}}^2$ and the corresponding transcripts can be constructed. Furthermore, $\sum_{i \in \cJ^c} f_i^2 < \rho^2/2$ implies $\cJ\neq \emptyset$  in view of $\sum_i  f_i^2\geq \rho^2$. Consequently, the conditions of Lemma \ref{lem : priv coin high-budget additional test II} are satisfied, yielding that there exists a test $T_{\text{III}}^2$ such that $\E_0 T_{\text{III}}^2\leq \alpha/6$ and $\E_f (1 - T_{\text{III}}^2) \leq \alpha /6$. We note that in case $\frac{mb}{d \sqrt{d}} > M_\alpha$, the test $ T_{\text{III}}^2$ cannot necessarily be computed (not enough communication budget), but this is not required as this case is covered by $T_{\text{III}}^1$.

We now have that for any $f \in H_\rho$, whenever $\frac{mb}{d \sqrt{d}} \leq M_\alpha$, the test $T_{\text{III}}$ can be computed and using that for nonnegative $x,y \geq 0$, $x \vee y \leq x + y$ and $x \vee y \geq x$, we obtain that
\begin{align*}
\cR(H_\rho, T_{\text{III}})&\leq \E_0 T_{\text{III}}^1+\E_0 T_{\text{III}}^2 \mathbbm{1}_{\{b \geq 2\log(d+1)\}} \\
&\quad+\sup_{f\in H_\rho}\min \big\{\E_{f}(1-T_{\text{III}}^1), \E_{f}(1-T_{\text{III}}^2\mathbbm{1}_{\{b \geq2\log(d+1)\}}) \big\} \\
& \leq  2\alpha/6+\alpha/6=\alpha/2.
\end{align*}

\end{proof}

Next we provide the risk bounds for the partial tests $T_{\text{III}}^1$ and $T_{\text{III}}^2$, {used in the previous lemma.}

\begin{lemma}\label{lem : private coin test risk ub I}
For any $\alpha \in (0,1)$ there exist constants $\kappa_\alpha,M_\alpha, C_\alpha>0$ such that 
$\E_0 T_{\text{III}}^1\leq \alpha/2.$
 Furthermore, for $f \in H_\rho$ if $\rho^2\geq C_\alpha \frac{d\sqrt{d}}{n (d \wedge b)}$ and either $\frac{mb}{d \sqrt{d}} \geq M_\alpha$ or 
\begin{align}
{\sum}_{i\in \cJ^c} f_i^2 \geq \rho^2/2\label{eq : private coin test condition majority of coordinates small},
\end{align}
 holds, where $\cJ$ was defined in \eqref{eq: def:J}, then  
$$\E_f (1 -  T_{\text{III}}^1)\leq \alpha/2.$$
\end{lemma}

\begin{proof}
Under the null hypothesis, $Y_i^{j}\sim^{iid}\text{Bern}(1/2)$. For each $\alpha \in (0,1)$ by applying Lemma \ref{lem : binomial testing} (with $k=d$ and $l = | \cI_1 |$) we get that $\E_0 T_{\text{III}}^1 \leq \alpha/2$ for large enough constant $\kappa_\alpha$. For $f \in H_\rho$, we have
\begin{equation*}
\E_f Y_i^{j} = \E_f \E_f \left[  Y_i^{j} | X_i^{j} \right] = \Phi \left( \sqrt{\frac{n}{m}}f_i \right).
\end{equation*}

To bound the Type II error, we use the second statement of Lemma \ref{lem : binomial testing} (with $k=d$ and $l = | \cI_1 | $), but before that we show that condition \eqref{eq : binomial l2 divergence lb} holds. Note that by Lemma \ref{lem : normal cdf divergence lb}, 
\begin{align}\label{eq : private coin key Lemma condition}
\frac{| \cI_1 | - 1}{2 \sqrt{d}} \underset{i=1}{\overset{d}{\sum}} \left(  \E_f Y_i^{j} - \frac{1}{2} \right)^2 &\geq  
\frac{| \cI_1 | - 1}{24\sqrt{d}} \underset{i=1}{\overset{d}{\sum}} \left( {\frac{n}{m}}f_i^2 \bigwedge 1 \right). 
\end{align}
In case \eqref{eq : private coin test condition majority of coordinates small} holds, the preceding display is bounded from below by
\begin{equation*}
    \frac{| \cI_1 |- 1}{24\sqrt{d}} \sum_{i\in \cJ^c } {\frac{n}{m}}f_i^2 \geq \frac{n(| \cI_1 | - 1) \rho^2}{48 m \sqrt{d}}.
\end{equation*}
Note, that for large enough $C_\alpha >0$, $\frac{n(| \cI_1 | - 1) \rho^2}{48 m \sqrt{d}}\geq n(\frac{mb}{d})C_{\alpha}\frac{d\sqrt{d}}{nb}/(96m\sqrt{d}) \geq \kappa_\alpha \vee \frac{16}{\alpha}$. If \eqref{eq : private coin test condition majority of coordinates small} does not hold, then there exists $ i^* \in \{ 1,\dots,d\}$ such that $f_{i^*} \geq \sqrt{m/n}$, so \eqref{eq : private coin key Lemma condition} is lower bounded by
\begin{equation*}
\frac{| \cI_1 | - 1}{24\sqrt{d}} \geq \frac{mb}{24d\sqrt{d}}  - \frac{1}{12\sqrt{d}} \geq \frac{M_\alpha}{24} - \frac{1}{12}.
\end{equation*}
Then for large enough $M_\alpha > 0$, the condition \eqref{eq : binomial l2 divergence lb} is satisfied. Consequently, the statement of the proof follows by the second statement of Lemma \ref{lem : binomial testing}. 
\end{proof}

\begin{lemma}\label{lem : priv coin high-budget additional test II}
For any $\alpha \in (0,1)$ there exists a $\kappa_\alpha > 0$ large enough such that $\E_0 T_{\text{III}}^2 \leq \alpha/2$. Furthermore, if $\rho^2 \geq C_\alpha \frac{d \sqrt{d}}{n (d \wedge b)}$, $m\geq M_{\alpha}$, for some large enough $C_\alpha,M_\alpha > 0$,  the set $\cJ$ defined in \eqref{eq: def:J} is non-empty  and $b\geq 2\log (d+1)$, then $\E_f T_{\text{III}}^2\leq \alpha/2$.
\end{lemma}

\begin{proof}[Proof of Lemma \ref{lem : priv coin high-budget additional test II}]
We apply Lemma \ref{lem : binomial testing supplement_version2} (with $k=d$ and $l= C_{b,d}m$), which is a version of  Lemma \ref{lem : binomial testing}, given in the Supplement. Under the null hypothesis, $\left(\sqrt{n/m}X_i^{j} \right)^2$ follows a chi-square distribution with one degree of freedom. Consequently, 
\begin{equation*}
\E_0 B^{j}_{li} = \E_0 F_{\chi^2_1}\Big( \Big(\sqrt{\frac{n}{m}}X_i^{j} \Big)^2 \Big) = 1/2
\end{equation*}
and 
\begin{equation}
\underset{j=1}{\overset{m}{\sum}}  N^{j} \sim \text{Bin}\left(1/2, mdC_{b,d}\right).
\end{equation}
Then Lemma \ref{lem : binomial testing supplement_version2} yields that $\E_0 T_{\text{III}}^2 \leq \alpha/2$.

Next we deal with the upper bound for the Type II error. Let $p_i:= \E_f F_{\chi^2_1}\big( (\sqrt{n/m}X_{i}^{j} )^2 \big)$ and note that $p_i \geq 1/2$. We apply again Lemma \ref{lem : binomial testing supplement_version2} (with $k=d$, $l= mC_{b,d}$). Hence it is sufficient to show that the condition \eqref{eq : binomial l2 divergence lb supplement_alternative} of the lemma holds. For this first note  that $\big(\sqrt{n/m}X_{i}^{j} \big)^2$ is a non-central chi-square distributed random variable with non-centrality parameter $\frac{n}{m} f_{i}^2$ and one degree of freedom. Consequently, for all $i \in \cJ\neq \emptyset$ we have
\begin{equation}\label{eq : high-budget priv coin p-val lb under alt}
{p}_{i} = \E_f F_{\chi^2_1}\Big( \Big(\sqrt{\frac{n}{m}}X_{i}^{j} \Big)^2 \Big) = \text{Pr} \left( V \geq 1 \right) > 3/5,
\end{equation}
where it is used that $V$ is noncentral F-distributed with noncentrality parameter $\frac{n}{m} f_{i}^2 \geq 1$ and $(1,1)$-degrees of freedom. Then by recalling that $\tilde{p}_i\geq 1/2$ we get that
\begin{align*}
\frac{mC_{b,d}-1}{2d} \Big(\sum_{i=1}^d (p_i - \frac{1}{2})\Big)^2 &\geq \frac{mC_{b,d}-1}{2d} \Big(\sum_{i\in \cJ}(\tilde{p}_i - \frac{1}{2})\Big)^2\\
&\geq \frac{mC_{b,d}-1}{2d} (|\cJ|/10)^2\geq \frac{m2^b}{400d^2 }\geq M_{\alpha}/400,
\end{align*}
yielding \eqref{eq : binomial l2 divergence lb supplement_alternative} for large enough choice of $M_{\alpha}$ and hence concluding the proof of our statement.
\end{proof}

The following three lemmas are standard, technical results, nevertheless we provided them for completeness.

\begin{lemma}\label{lem : normal cdf divergence lb}
Let $\Phi$ denote the cdf of a standard normal random variable. It holds that
\begin{equation*}
\left( \Phi(x) - \frac{1}{2} \right)^2 \geq \frac{1}{12} \min \left\{ x^2, 1 \right\}.
\end{equation*}

\end{lemma}
\begin{proof}
Since $\Phi(x) = 1 - \Phi(-x)$, it holds that $\left( \Phi(x) - \frac{1}{2} \right)^2 = \left( \Phi(-x) - \frac{1}{2} \right)^2$ hence one can consider $x \geq 0$ without loss of generality. We first show that $\Phi(x) - \frac{1}{2} \geq \frac{x^2}{8}$ for $0 \leq x \leq 1/\sqrt{2}$. We have
\begin{align}\label{eq : gaussian error function as series}
\Phi(x) - \frac{1}{2} = \frac{1}{\sqrt{2 \pi}} \int_0^x e^{-\frac{1}{2} z^2} dz  = \frac{1}{\sqrt{2 \pi}} \int_0^x \underset{i=0}{\overset{\infty}{\sum}} \frac{(-1)^i z^{2i} }{2^i i!}  dz = \frac{x}{\sqrt{2 \pi}} \Big(\underset{i=0}{\overset{\infty}{\sum}} \frac{(-1)^i ( x/\sqrt{2})^{2i} }{(2i+1) i!} \Big),
\end{align}
where the last equation follows by Fubini's theorem.  The series in the rhs is decreasing in $x\in[0,\sqrt{2}]$, as for each odd $i$ it holds that
\begin{equation*}
\frac{d}{d \epsilon} \left[\frac{(-1)^i \epsilon^{2i}}{(2i+1)i!} + \frac{(-1)^{i+1} \epsilon^{2i+2}}{(2i+3)(i+1)!} \right] = \frac{\epsilon^{2i-1}2i}{ i!(2i+1)} \left( \frac{2\epsilon^2 (2i + 1)(2i+2)}{(i+1)2i (2i+3)} - 1 \right) < 0
\end{equation*}
for $0 \leq \epsilon \leq 1$. Hence, for $0 \leq x/\sqrt{2} \leq c \leq 1$,
\begin{align*}
\frac{x}{\sqrt{2 \pi}} \Big(\underset{i=0}{\overset{\infty}{\sum}} \frac{(-1)^i ( x/\sqrt{2})^{2i} }{(2i+1) i!} \Big)\geq\frac{x}{\sqrt{2 \pi}} \left(\underset{i=0}{\overset{\infty}{\sum}} \frac{(-1)^i  c^{2i} }{(2i+1) i!} \right) &= \frac{ x}{\sqrt{2}c} \left( \Phi(\sqrt{2}c) - \frac{1}{2}\right),  
\end{align*}
where the last equality follows by \eqref{eq : gaussian error function as series}. For $ x > \sqrt{2}c$, it holds that
\begin{equation*}
\Phi(x) - 1/2 \geq \Phi(\sqrt{2}c) -1/2
\end{equation*}
as $x \mapsto \Phi(x) - 1/2$ is increasing. Taking $c=1$ we obtain 
\begin{equation*}
\Phi(x) - 1/2\geq   \min \Big\{x \big( \Phi(\sqrt{2}) - 1/2 \big)/\sqrt{2} , \Phi(\sqrt{2}) - 1\Big\}>  \min \{x,1\}/\sqrt{12},
\end{equation*}
which finishes the proof.
\end{proof}

\begin{lemma}\label{lem : Gaussian maximum}
Let $Z=(Z_1,\dots,Z_d) \sim N(0,I_d)$. It holds that $\E \underset{1 \leq i \leq d}{\max} |Z_i| \leq 3 \sqrt{\log(d) \vee \log(2)}$ and
\begin{equation*}
\text{Pr} \left(\underset{1 \leq i \leq d}{\max} Z_i^2 \geq x \right) \leq \frac{2d}{e^{x/4}},
\end{equation*}
for all $ x > 0$.
\end{lemma}
\begin{proof}
The case where $d = 1$ follows by standard Gaussian concentration properties. Assume $d\geq 2$. For $ 0 \leq t \leq 1/4$,
\begin{align*}
 \E e^{t \max_i (Z_i)^2} = e^{t}\E \max_i e^{t (Z_i^2 - 1)} \leq d  e^{{2t^2} + t},
 \end{align*}
see e.g. Lemma 12 in \cite{szabo2022optimal}. Taking $t=1/4$ and applying Markov's inequality yields the second statement of the lemma. Furthermore, in view of Jensen's inequality
\begin{equation*}
\E \max_i (Z_i)^2 \leq \frac{\log(d)}{t} + 2t + 1,
\end{equation*}
which in turn yields $\max_i  |Z_i| \leq 3 \sqrt{\log(d)}$.
\end{proof}

\begin{lemma}\label{lem : Chernoff-Hoeffding bound chisq}
Let $X_d$ be Chi-square random variable with $d$-degrees of freedom. For $0<c<1$ it holds that
\begin{equation*}
\text{Pr}\left( X_d \leq cd \right) \leq e^{-d \frac{c-1 - \log (c)}{2}}.
\end{equation*}
Similarly, for $c > 1$ it holds that
\begin{equation*}
\text{Pr}\left( X_d \geq cd \right) \leq e^{-d \frac{c-1 - \log (c)}{2}}.
\end{equation*}
\end{lemma}
\begin{proof}
Let $t < 0$. We have
\begin{align*}
\text{Pr}\left( X_d \leq cd \right) &= \text{Pr}\left( e^{tX_d} \geq e^{tcd} \right) \leq \frac{\E e^{tX_d} }{e^{tcd}}.
\end{align*}
Using that $\E e^{tX_d} = (1-2t)^{-d/2}$, the latter display equals
\begin{equation*}
\exp \Big( - d \big( tc + \frac{1}{2} \log(1-2t) \big) \Big).
\end{equation*}
The expression $tc + \frac{1}{2} \log(1-2t)$ is maximized when $t = \frac{1}{2}(1 - \frac{1}{c}) < 0$ which leads to the result. The second statement follows by similar steps.
\end{proof}

\section{Proof of Theorem \ref{thm : nonparametric SNWN minimax rate}}\label{sec: proof:nonparam}
%We now construct a prior $\pi$ of a specific form; taking values only in a finite dimensional subspace of $L_2[0,1]$. 

For convenience, we consider a sufficiently smooth orthonormal wavelet basis $\{ \psi_{li} : l \in \N_0, \; i = 0, 1, \dots, 2^l - 1\}$ for $L_2[0,1]$, see Section \ref{sec: wavelets} for a brief introduction of wavelets and collection of properties used during the proof. Nevertheless we note, that other basis (e.g. Fourier) could be used equivalently. Let $f^L$, $\tilde{X}^{j}_{L':L}$ and{ $\tilde{f}^L$} as defined in \eqref{eq : wavelet projection f}, \eqref{eq : wavelet coefficients X} and {below \eqref{eq : wavelet coefficients X}, respectively. Furthermore,  let $\Psi_L : \R^{2^L} \to L_2[0,1]$ be the measurable map defined by
\begin{equation}\label{eq : inverse wavelet transform}
\Psi_L \tilde{f}^L = \underset{i=0}{\overset{2^L-1}{\sum}} \tilde{f}_{i} \psi_{Li},
\end{equation}
for $\tilde{f}^L=(\tilde{f}_0,\dots,\tilde{f}_{2^L-1})$.}\\

% Xi MATRICES DEFINED NONPARAMETRIC SETTING
% Define the matrices
% \begin{align*}
% \Xi_{L,u}^j &= \E_0 \E_0 \left[ X^{j}_{L} \big| Y^{j} , U=u \right] \E_0 \left[ X^{j}_{L} \big| Y^{j}, U=u \right]^\top, \\  \Xi_{L':L, u}^j &= \E_0 \E_0 \left[ X^{j}_{L':L} \big| Y^{j}, U=u  \right] \E_0 \left[ X^{j}_{L':L} \big| Y^{j}, U=u \right]^\top,
% \end{align*}
% $\Xi_{L,u} := \sum_{j=1}^m  \Xi_{L,u}^j$ and $\Xi_{u} = \sum_{j=1}^m \Xi_{L_{\min}:L_{\max}, u}^j$.}

\emph{The existence of $C_\alpha > 0$ such that $f \in H^{s,R}_{C_\alpha \rho}$ can be detected.}  

In view of Theorem \ref{thm : detection ub}, there exists a constant $C_\alpha' > 0$ and a $b$-bit public coin distributed testing protocol $T$ with transcripts generated according to $Y^{j}|(\tilde{X}^{j}_{0:L},U) \sim K^j(\cdot | \tilde{X}^{j}_{0:L},U)$ such that if $\| \tilde{f^L} \|_2^2 \geq (C_\alpha')^2 \frac{\sqrt{2^L}}{n} \left( \sqrt{\frac{2^L }{b \wedge 2^L}} \bigwedge \sqrt{m}  \right)$, we have
\begin{equation*}
\E_0 T + \E_{\tilde{f}^L} (1-T) \leq \alpha.
\end{equation*}
Similarly, there exists a constant $C_\alpha' > 0$ and a $b$-bit private coin distributed testing protocol $T$ such that the above display holds if
$\| \tilde{f}^L \|_2^2 \geq (C_\alpha')^2 \frac{\sqrt{2^L}}{n} \left( \frac{2^L }{b \wedge 2^L} \bigwedge \sqrt{m}  \right)$. See Section \ref{sec : upper bounds} for the construction of such testing protocols.

Consequently, it suffices to show that for $f\in H_{C_\alpha \rho}^{s,R}$, $\| \tilde{f^L} \|_2^2$ satisfies the above lower bounds for some $L \in \N$ and $c > 0$. In view of $(a+b)^2/2-b^2\leq a^2$,
\begin{equation*}
\| f^L \|_{L_2}^2 \geq \frac{\| f\|_{L_2}^2}{2} - \| f - f^L \|_{L_2}^2.
\end{equation*}
Furthermore, $f \in H_{C_\alpha \rho}^{s,R}$ implies that   
\begin{equation*}
\|f - f^L \|_{L_2}^2=\sum_{l>L} \underset{i=0}{\overset{2^l-1}{\sum}}  \tilde{f}_{li}^2\leq 2^{-2Ls} \sum_{l>L} \underset{i=0}{\overset{2^l-1}{\sum}}  \tilde{f}_{li}^2 2^{2ls}  \leq  \frac{\|f\|_{\cH^s}^2}{2^{2Ls}} \leq \frac{R^2}{2^{2Ls}}\quad\text{and}\quad \|f\|_{L_2}^2\geq C_\alpha^2\rho^2.
\end{equation*}
Consequently, in view of Plancharel's theorem and taking $L = 1 \vee \lceil - \frac{1}{s} \log  \rho^{} \rceil$,
\begin{equation}\label{eq: help:triangle}
 \| \tilde{f^L} \|_2^2=\| f^L \|_{L_2}^2 \geq   \rho^2  C_\alpha^2/2- R^2 2^{-2Ls} \geq\rho^2 ( C_\alpha^2/2 - {R^2}).
\end{equation}
Consequently, there exists a $b$-bit public coin distributed testing protocol such that
 \begin{equation*}
\E_0 T + \E_{f} (1-T) \leq \alpha
\end{equation*}
whenever 
\begin{equation}\label{eq : pub coin nonparametric testability condition}
\rho^2 \gtrsim \frac{\sqrt{ 2^L}}{n} \left( \sqrt{\frac{ 2^L }{b \wedge 2^L}} \bigwedge \sqrt{m}  \right) \asymp \frac{\sqrt{1 \vee \rho^{-1/s}}}{n} \left( \sqrt{\frac{1 \vee \rho^{-1/s} }{b \wedge (1\vee \rho^{-1/s})}} \bigwedge \sqrt{m}  \right),
\end{equation}
since the constant $( \frac{C_\alpha^2}{2} - {R^2})$ can be made arbitrary large by large enough choice of $C_\alpha > 0$. In the case that $b \geq (1\vee \rho^{-1/s})$, the above display is satisfied whenever $\rho^{2+\frac{1}{2s}} \gtrsim n^{-1}$, which provides the first case in \eqref{eq : rho condition public coin nonparametric}. Similarly, if $b \leq \rho^{-1/s}$, the above display boils down to $\rho^{2+\frac{1}{s}} \gtrsim (\sqrt{b}n)^{-1}$ whenever $bm \geq \rho^{-1/s}$, which leads to the second case in \eqref{eq : rho condition public coin nonparametric}. If $bm \leq \rho^{-1/s}$, the inequality \eqref{eq : pub coin nonparametric testability condition} reduces to $\rho^{2+\frac{1}{2s}} \gtrsim\sqrt{m}/{n}$ and consequently provides the third case in \eqref{eq : rho condition public coin nonparametric}.

By similar argument as for the public coin protocol above, there exists a $b$-bit private coin distributed testing protocol with testing risk less than $\alpha$ whenever
\begin{equation*}
\rho^2 \gtrsim \frac{\sqrt{1 \vee \rho^{-1/s}}}{n} \Big( \frac{1 \vee \rho^{-1/s} }{b \wedge (1\vee \rho^{-1/s})} \bigwedge \sqrt{m}  \Big)
\end{equation*}
and $C_\alpha > 0$ large enough. Then a similar computation as in the public coin case above leads to the three cases in \eqref{eq : rho condition private coin nonparametric}.\\

\emph{The existence of $c_\alpha$ for which the risk is lower bounded.} 

For any distribution $\pi_L$ on $\R^{L}$, $\pi_L \circ \Psi^{-1}_L$ defines a probability measure on the Borel sigma algebra of $L_2[0,1]$. For $\tilde{f}^L \in \R^{2^L}$, the likelihood ratio $\frac{dP_{f}}{dP_0}(X^{j})$ with $f=\Psi_L \tilde{f}^L$ equals
\begin{equation}\label{eq : likelihood under wavelet alternative}
\exp \left( \frac{n}{m} \int_0^1 f d X_t^{j} - \frac{n}{2m} \|f\|_2^2\right) = \exp \left({ \frac{n}{m} (\tilde{f}^L)^\top \tilde{X}^{j}_L - \frac{n}{2m} \|\tilde{f}^L\|_2^2} \right) =: \mathscr{L}_{\tilde{f}^L}(\tilde{X}^{j}_L),
\end{equation} %through $\pi \circ \Psi^{-1}(B) = \pi( \tilde{f} : \Psi f \in B)$. 
where $\tilde{X}^{j}_L = (\int_0^1 \psi_{L0}(t) d{X}^{j}_t,\dots,\int_0^1 \psi_{L(2^L-1)}(t) d{X}^{j}_t)\in\mathbb{R}^{2^L}$.  %\overset{d}{=} \exp \left({ \frac{n}{m} \tilde{f}^\top Z - \frac{n}{2m} \|\tilde{f}\|_2^2} \right)
For an arbitrary $b$-bit distributed testing protocol $T = (T,K,\P^U)$, following the proof of Theorem \ref{thm : detection lb} up until equation \eqref{eq : sup inf public coin divergence lb} we obtain that
\begin{equation}\label{eq : nonparametric risk lower bound}
\cR(H_\rho,T) \geq 1 - \left( \int \sqrt{2D_{\chi^2}(\P^{Y|U=u}_{0,K} ; \P^{Y|U=u}_{\pi,K}) }d\P^U(u) + \pi\left(\tilde{f}^L \in \R^{L} :\Psi \tilde{f}^L \notin H_{c_\alpha \rho}^{s,R} \right) \right).
\end{equation}
By \eqref{eq : likelihood under wavelet alternative},
\begin{equation}\label{eq : nonparametric Chi-square div}
D_{\chi^2}(\P^{Y|U=u}_{0,K} ; \P^{Y|U=u}_{\pi,K}) = \E_0^{Y|U=u} \left( \E_0 \left[ \int \underset{j=1}{\overset{m}{\Pi}} \mathscr{L}_{\tilde{f}^L}(\tilde{X}_L^{j}) d\pi(\tilde{f}^L) \bigg| Y, U=u \right]^2 \right) - 1.
\end{equation}
Under $\P_0$, $\mathscr{L}_{\tilde{f}^L}(\tilde{X}^{j}_L)$ is equal in distribution to the likelihood ratio
\begin{equation*}
 \frac{dN(\tilde{f}^L, \frac{m}{n} I_{2^L})}{dN(0, \frac{m}{n} I_{2^L})}.
\end{equation*}
That means that the argument of the proof of Theorem \ref{thm : detection lb} for bounding the Chi-square divergence applies to the first term in \eqref{eq : nonparametric risk lower bound}. Choosing $\pi_L = N(0,\Gamma)$ with $\Gamma = \frac{\sqrt{c_\alpha} \rho^2}{2^{L}} \bar{\Gamma} \in \R^{2^L \times 2^L}$ and $\bar{\Gamma}$ as in the proof of Theorem \ref{thm : detection lb}. In particular,  we obtain that for some constant $C > 0$ not depending on $\rho,n,m,b,c_\alpha$ or $L$,
\begin{equation}\label{eq : nonpar. chi-sq div final bound}
\int \sqrt{2D_{\chi^2}(\P^{Y|U=u}_{0,K} ; \P^{Y|U=u}_{\pi,K}) }d\P^U(u) \leq \begin{cases}
\sqrt{2}(e^{C c_\alpha^{-1} (\frac{n^2 \rho^4}{2^L m} + \frac{n^2 \rho^4 (b\wedge 2^L)^2}{2^{3L}})} -1), &\mbox{ if } U \text{ is degenerate,}\\
\sqrt{2}(e^{C c_\alpha^{-1} (\frac{n^2 \rho^4}{2^{L}m} + \frac{n^2 \rho^4 (b\wedge 2^{L})}{2^{2L}})} -1), &\mbox{ otherwise.}
\end{cases}
\end{equation}
Note that for $\rho^2 \leq c \frac{\sqrt{ 2^L}}{n} \Big( \sqrt{\frac{ 2^L }{b \wedge 2^L}} \bigwedge \sqrt{m}  \Big)$ in the degenerate $U$ and for $\rho^2 \leq c \frac{\sqrt{ 2^L}}{n} \Big( \frac{ 2^L }{b \wedge 2^L} \bigwedge \sqrt{m}  \Big)$ in the not degenerate $U$ case, both terms on the rhs of the preceding display are bounded by $\sqrt{2}(e^{2c_{\alpha}C}-1)$, which is further bounded by $2^{5/2}Cc_{\alpha}\leq \alpha$ for small enough choice of $c_{\alpha}$. Taking again $L = 2 \vee \lceil \log \rho^{-1/s} \rceil$, by similar argument as given below display \eqref{eq : pub coin nonparametric testability condition} the above upper bounds for $\rho^2$ result in \eqref{eq : rho condition public coin nonparametric} and \eqref{eq : rho condition private coin nonparametric}. 

It remained to bound the prior mass term in \eqref{eq : nonparametric risk lower bound} for $L = 2 \vee \lceil \log \rho^{-1/s} \rceil$. That is, we will show that
\begin{equation}\label{eq : nonparametric remaining mass}
\pi_L\left(\tilde{f}^L \in \R^{2^L} : \|\Psi_L \tilde{f}^L\|_{L_2}^2 \geq c_\alpha \rho^2, \; \| \Psi_L \tilde{f}^L \|_{\cH^s}^2 \leq R^2 \right) \geq 1-\alpha/2,
\end{equation}
 for all $n$ large enough. Note that for all $L \in \N$, $\| \Psi_L \tilde{f}^L \|_{\cH^s}^2 \leq 2^{2Ls} \| \Psi_L \tilde{f}^L \|_{L_2}$. Consequently using Plancharel's theorem, we obtain that the lhs of \eqref{eq : nonparametric remaining mass} is bounded from below by
\begin{align}
\pi_L\left(\tilde{f}^L \in \R^{2^L} : c_\alpha \rho^2 \leq \| \tilde{f}^L \|_2^2 \leq  2^{-2Ls} R^2  \right)
&\geq\text{Pr}\left( c_\alpha \rho^2 \leq  Z^\top \Gamma Z  \leq R^2 \rho^2 \right)\nonumber\\
& = \text{Pr}\left( \sqrt{c_\alpha} 2^L \leq  Z^\top \bar{\Gamma} Z  \leq \frac{R^2}{\sqrt{c_\alpha}} 2^L \right),\label{eq:help:thm6.1}
\end{align}
where $Z$ is a $2^L$-dimensional standard normal vector. For both the public and private coin choices of $\bar{\Gamma}$ in the proof of Theorem \ref{thm : detection lb}, $\bar{\Gamma}$ is symmetric, idempotent and has rank $2^L$ and $\lceil 2^L / 2 \rceil$ respectively. In the public coin case $Z^\top \Gamma Z \sim \chi^2_{2^L}$, hence Lemma \ref{lem : Chernoff-Hoeffding bound chisq} yields that the rhs of the above display is bounded from below by
\begin{equation*}
1 - \exp \Big( - 2^L \frac{\sqrt{c_\alpha} - 1 - 0.5 \log c_\alpha}{4}\Big) - \exp \Big( - 2^L \frac{R^2/\sqrt{c_\alpha} - 1 - 0.5\log \left(R^4/c_\alpha\right)}{4} \Big),
\end{equation*}
which can be set arbitrarily close to $1$ per small enough choice of $c_\alpha > 0$, verifying the prior mass condition.

In the private coin protocol case $Z^\top \Gamma Z \sim \chi^2_{\lceil 2^L / 2 \rceil}$ and by applying again Lemma \ref{lem : Chernoff-Hoeffding bound chisq} (with $d=\lceil 2^L / 2 \rceil$) we get by similar computations as above that the rhs of \eqref{eq:help:thm6.1} is arbitrarily close to one for small enough choice of $c_{\alpha}$.

\section{Public coin protocols for estimation}

Consider the distributed signal-in-Gaussian-white-noise model as described in Section \ref{sec : nonparametric part}, i.e. local $X=(X^{1},\dots,X^{m})$ observations satisfying the dynamics of \eqref{eq : nonparametric model dynamics} and $b$-bit transcripts $Y=(Y^{1},\dots,Y^{m})$ communicated to a central machine taking values in a space $\cY^m$ with $|\cY|=b$. Let $\cE_{pub}(b)$ denote the class of all distributed estimation protocols generating transcripts that may depend on a public coin $U$. That is, $\cE_{pub}(b)$ consists of pairs $(\hat{f},\cL(Y,U|X))$ where $\hat{f}: \cY \to L_2[0,1]$ and $\cL((Y,U)|X)$ is such that 
\begin{equation*}
\P_f^Y(y) = \int \int \P^{Y|(X,U)=(x,u)}(y) d\P^X_f(x) d\P^U(u),
\end{equation*}
$X$ is independent of $U$ and $Y^{1},\dots,Y^{m}$ are independent given $(X,U)$. Let $\cE_{priv}(b)$ denote the class of all distributed estimation protocols that do not depend on a public coin. This is equivalent to the definition of $\cE_{pub}(b)$ above with $U$ set to a degenerate random variable. Below, we shall write $\hat{f} \equiv (\hat{f},\cL(Y,U|X))$ when no confusion can arise.

\begin{theorem}\label{thm : public coin estimation}
The distributed minimax estimation rates under communication constraints are the same in the public and private coin protocols, i.e.
\begin{equation*}
\underset{\hat{f} \in \cE_{pub}(b)}{\inf} \; \underset{f \in \cH^{s,R}}{\sup} \E_f^{(Y,U)} \| \hat{f}(Y) - f \|^2_{L_2} \asymp \underset{\hat{f} \in \cE_{priv}(b)}{\inf} \; \underset{f \in \cH^{s,R}}{\sup} \E_f^Y \| \hat{f}(Y) - f \|^2_{L_2}.
\end{equation*}
\end{theorem}
\begin{proof}
Since a private coin protocol can be seen as a public coin protocol with a degenerate random variable $U$, it remained to deal with the ``$\gtrsim$'' inequality. To that extend, it is sufficient to show that the same lower bound as for the private coin case holds. 

Following the proof of Theorem 3.1 of \cite{pmlr-v80-zhu18a}, there exists a distribution $\pi$ on $\cH^{s,R}$ such that
\begin{equation}
\underset{\hat{f} \in \cE_{priv}(b)}{\inf} \; \underset{f \in \cH^{s,R}}{\sup} \E_f^Y \| \hat{f}(Y) - f \|^2_{L_2} \geq \underset{\hat{f}  \in \cE_{priv}(b)}{\inf} \int \E_f^Y \| \hat{f}(Y) - f \|^2_{L_2} d \pi(f)(1 + o(1)),
\end{equation}
where the $o(1)$ term is concerned with asymptotics in $n$ only. The particular choice of $\pi$ considered in \cite{pmlr-v80-zhu18a} does not depend on the law of $Y$ and satisfies 
\begin{equation}\label{eq : estimation bayes lower bound}
 \int \E_f^Y \|\hat{f}(Y) - f \|^2_{L_2} d \pi(f) \gtrsim \begin{cases} 
 n^{-\frac{2s}{2s+1}} &\text{ if } b \gtrsim n^{\frac{1}{2s+1}}, \\
{(b n)}^{-\frac{2s}{2s+2}} &\text{ if } b \lesssim n^{\frac{1}{2s+1}} \text{ and } b \gtrsim (n/m^{2s+2})^{\frac{1}{2s+1}},  \\
 \left(bm \right)^{-2s} &\text{ if } b \lesssim  (n/m^{2s+2})^{\frac{1}{2s+1}}.
 %\\ \frac{1}{n} &\text{ if } \rho > 1.
\end{cases}
\end{equation}
Furthermore, this lower bound is tight, see Section 4 in \cite{pmlr-v80-zhu18a}. By the Markov chain structure of \eqref{eq : markov chain structure}, we have that 
\begin{equation*}
\int \P^Y_f(y) d\pi(f) = \int \int \int \P_f^{Y|(X,U)=(x,u)}(y) d\P^X_f(x) d\P^U(u) d\pi(f).
\end{equation*}
Consequently, for the same choice of prior and any $(\hat{f},\cL(Y,U|X)) \in \cE_{pub}(b)$, $f \in \cH^{s,R}$,
\begin{align*}
 \E_f^Y \| \hat{f}(Y) - f \|^2_{L_2} &\geq \int \int \E_f^{Y|U=u} \| \hat{f}(Y) - f \|^2_{L_2} d\pi(f)(1+o(1)) d\P^U(u) \\
&\geq \int \underset{\hat{f} \in \cE_{pub}(b)}{\inf}  \int \E_f^{Y|U=u} \| \hat{f}(Y) - f \|^2_{L_2} d\pi(f)(1+o(1)) d\P^U(u) \\
&= \int \underset{\hat{f} \in \cE_{priv}(b)}{\inf}  \int \E_f^{Y} \| \hat{f}(Y) - f \|^2_{L_2} d\pi(f)(1+o(1)) d\P^U(u) \\
 &=  \underset{\hat{f} \in \cE_{priv}(b)}{\inf} \int \E_f^{Y} \| \hat{f}(Y) - f \|^2_{L_2} d\pi(f)(1+o(1)).
\end{align*}
Here, the second to last equation follows from the fact that for any $(\hat{f},\cL(Y,U|X)) \in \cE_{pub}(b)$, it holds that $(\hat{f},\P_f^{Y|X,U=u}) \in \cE_{priv}(b)$. By \eqref{eq : estimation bayes lower bound}, the private coin lower bound also holds in the public coin case and the result follows.
\end{proof}

\section{{Proof of the lower bounds in Theorems \ref{thm:adapt1} and \ref{thm:adapt2}}}\label{sec : adaptation lower bounds}

{Let $f^L$ and $\tilde{X}^{j}_{L':L}$ as defined in \eqref{eq : wavelet projection f} and \eqref{eq : wavelet coefficients X}, respectively.} Let $T = (T,K,\P^U)$ be a given distributed testing protocol (with $U$ degenerate in the case it is a private coin protocol) and fix $\alpha \in (0,1)$. For given $s_{\min} < s_{\max}$, consider for $s \in [s_{\min} , s_{\max}]$ the map $s \mapsto \rho_s$.

% WAVELET BASIS INTRODUCTION
Recall that for $\Psi_L$ as defined in \eqref{eq : inverse wavelet transform} and any distribution $\pi_L$ on $\R^{\nu({L})}$, $\pi_L \circ \Psi^{-1}_L$ defines a probability measure on the Borel sigma algebra of $L_2[0,1]$. Define the mixture of the above probability measures by
\begin{equation}\label{eq : adaptation adverserial prior}
\Pi = \frac{1}{|\cC_0|} \underset{L \in \cC_0}{\overset{}{\sum}}  \pi_L \circ \Psi^{-1}_L
\end{equation}
where $\cC_0 \subseteq \cC$. There exists a grid of points $\cS \subset [s_{\min} , s_{\max}]$ such that the map $s \mapsto L_s$ is a one-to-one map from $\cS$ to $\cC$. Let $L \mapsto s_L$ denote its inverse. 

By the same steps as in \eqref{eq : type II error minus mass lb}, 
\begin{equation}
\underset{f \in H_{c_\alpha \rho_{s_L}}^{s_L,R}}{\sup} \P_f^Y ( T = 0 ) 
\geq \P_{\pi_L}^Y ( T = 0 ) - \pi_L \circ \Psi^{-1}_L \left( f \notin H_{c_\alpha \rho_{s_L}}^{s_L,R} \right),
\end{equation}
for all $L \in \cC$. Using the above display, we can bound the risk in the adaptive setting from below:
\begin{align}
\underset{s \in [s_{\min} , s_{\max}]}{\sup} \; \cR (H_{c_\alpha \rho_s}^{s,R}, T ) &\geq \frac{1}{|\cC|} \underset{L \in \cC}{\overset{}{\sum}}  \cR (H_{c_\alpha \rho_{s_L}}^{s_L,R}, T ) \nonumber \\
&\geq \P_{0}^Y ( T = 1 ) + \P_{\Pi}^Y ( T = 0 ) - \frac{1}{|\cC_0|} \underset{L \in \cC_0}{\overset{}{\sum}}  \pi_L \circ \Psi^{-1}_L \left( f \notin H_{c_\alpha \rho_{s_L}}^{s_L,R} \right). \label{eq : adaptive risk hyperprior}
\end{align}
Taking $\pi_L$ as in the proof of Theorem \ref{thm : nonparametric SNWN minimax rate}, then by the same reasoning as in proof the proof of Theorem \ref{thm : nonparametric SNWN minimax rate} that the third term in the above display can be made arbitrarily small per choice of $c_\alpha$ for $\rho_s$ satisfying \eqref{eq : rho condition public coin adaptive lb}-\eqref{eq : rho condition private coin adaptive lb}. For the first two terms, define 
\begin{equation*}
\mathcal{L}^{Y|u}_{\pi_L} := \int \frac{d\P_f^{Y|U=u}}{d\P_0^{Y|U=u}} d\pi_L(f) 
\end{equation*}
and note that
\begin{align*}
\P_{0}^Y ( T = 1 ) + \P_{\Pi}^Y ( T = 0 ) &= \frac{1}{|\cC_0|} \underset{L \in \cC_0}{\overset{}{\sum}}  \int \P^{Y|U=u}_0 \left( T + \mathcal{L}^{Y|u}_{\pi_L}(1 - T) \right) d\P^U(u) \\ &\geq \frac{1}{|\cC_0|} \underset{L \in \cC_0}{\overset{}{\sum}}  \int \E^{Y|U=u}_0 \left( \gamma T + \mathcal{L}^{Y|u}_{\pi_L}(1 - T) \right) \mathbbm{1}\left\{ \mathcal{L}^{Y|u}_{\pi_L} > \gamma \right\} d\P^U(u) \\
&\geq \gamma \frac{1}{|\cC_0|} \underset{L \in \cC_0}{\overset{}{\sum}} \int \P^{Y|U=u}_0 \left(  \mathcal{L}^{Y|u}_{\pi_L} > \gamma \right) d\P^U(u),
\end{align*}
where the conditioning follows from the Markov chain structure \eqref{eq : markov chain structure} and the inequality holds for $0 < \gamma < 1$. We can conclude that it suffices to show that for all $\varepsilon > 0$,
\begin{equation}\label{eq : to show likelihood to 1 in prob}
\frac{1}{|\cC_0|} \underset{L \in \cC_0}{\overset{}{\sum}} \P_0^{(Y,U)} \left( \left|\mathcal{L}^{Y|U}_{\pi_L} - 1 \right| > \varepsilon \right) 
\end{equation}
can be made arbitrarily small per small enough choice of $c_\alpha$
%  to show that
% \begin{equation}\label{eq : to show likelihood to 1 in prob}
% \mathcal{L}^{Y|U}_{\pi_L} \to 1 \; \text{ as } c_\alpha \to 0 \text{ in } \; \frac{1}{|\cC|} \underset{L \in \cC}{\overset{}{\sum}} \delta_L \times \P_0^{(Y,U)}\text{--probability}
% \end{equation}
in order obtain the required lower bound in \eqref{eq : adaptive risk hyperprior}. Using $\P^{(Y,U)}_0 = d\P^U d \P^{Y|U}_0$, conditioning on the $\P^{Y|U}_0$-variance of $\mathcal{L}^{Y|u}_\Pi$ with Chebyshev's inequality and $\E_0^{Y|U=u} \mathcal{L}^{Y|u}_\Pi = 1$ lead to
\begin{equation*}
\frac{1}{|\cC_0|} \underset{L \in \cC_0}{\overset{}{\sum}} \P^{(Y,U)}_0 \left( \left(\mathcal{L}^{Y|U}_{\pi_L} - 1\right)^2 > \varepsilon^2 \right) \leq \frac{1}{|\cC_0|} \underset{L \in \cC_0}{\overset{}{\sum}} \P^U \left( \E^{Y|U} (\mathcal{L}^{Y|U}_{\pi_L})^2  > 1 + \zeta \right)  + \frac{\zeta}{\varepsilon^2}
\end{equation*}
 for all $\varepsilon > 0$ and $\zeta > 0$. Noting that $\E^{Y|U=u} (\mathcal{L}^{Y|U=u}_{\pi_L})^2 \geq 1$, sufficiently bounding \eqref{eq : to show likelihood to 1 in prob} follows from Markov's inequality and showing
 \begin{equation}\label{eq : to show pub int log-variance is o(1)}
 \frac{1}{|\cC_0|} \underset{L \in \cC_0}{\overset{}{\sum}} \int \log \left( \E^{Y|U=u} (\mathcal{L}^{Y|U=u}_{\pi_L})^2 \right) d\P^U(u) \lesssim c_\alpha.
 \end{equation}
Noting that $\E^{Y|U=u} (\mathcal{L}^{Y|U=u}_{\pi_L})^2 = D_{\chi^2}(\P^{Y|U=u}_{0,K} ; \P^{Y|U=u}_{\pi_L,K}) + 1$, we can apply the argument of the proof of Theorem \ref{thm : detection lb} (foregoing the bound of \eqref{eq : factorized problem ip}) for bounding the Chi-square divergence and we obtain that for some fixed $C > 0$,
\begin{equation}
\log \left( \E^{Y|U=u} (\mathcal{L}^{Y|U=u}_{\pi_L})^2 \right) \leq \begin{cases}
{C c_\alpha \frac{n^4 \rho^4_{s_L}}{m^4 2^{3L}} \text{Tr} \left( \Xi_{L,u} \right)^2  } + A_{L,u}, &\mbox{ if } U \text{ is degenerate,}\\
{C c_\alpha \frac{n^3 \rho^4_{s_L}}{m^2 2^{2L}} \text{Tr} \left( \Xi_{L,u} \right)} + A_{L,u}, &\mbox{ otherwise,}
\end{cases} 
\end{equation}
where 
\begin{equation*}
A_{L,u} = \underset{j=1}{\overset{m}{\sum}} \log \left( \E_0^{Y^{j}|U=u} \left( \E_0 \left[  \int \frac{d\P^{\tilde{X}^{j}}_f}{d \P^{\tilde{X}^{j}}_0} ( \tilde{X}^{j}_L) d\pi_L(f) \bigg| Y^{j}, U=u \right]^2 \right) \right)
\end{equation*}
and $\Xi_{L,u} = \sum_{j=1}^{m} \Xi_{L,u}^j$ with $\Xi_{L,u}^j = \E_0 \E_0 \left[ \tilde{X}^{j}_{L} \big| Y^{j} , U=u \right] \E_0 \left[ \tilde{X}^{j}_{L} \big| Y^{j}, U=u \right]^\top$. Via a data processing argument (Lemma \ref{lem : small budget adaptation processing} in the supplement), 
\begin{equation*}
\frac{1}{|\cC_0|} \underset{L \in \cC_0}{\overset{}{\sum}}  \int A_{L,u} d\P^U(u) \lesssim \max_{L \in \cC_0} \frac{c_\alpha n^2 \rho_{s_L}^4 (b \wedge |\cC_0|)}{m 2^L |\cC_0|}.
\end{equation*}
When $U$ is degenerate, Lemma \ref{lem : strong data processing Fisher info adaptive I} implies that there exists a choice for $\cC_0 \subset \cC_0$ such that for all $L\in\cC_0$,
\begin{equation*}
\text{Tr} \left( \Xi_{L,u} \right)^2 \lesssim \left( \frac{b}{|\cC| } \wedge 2^L \right)^2 \frac{m^4}{n^2}.
\end{equation*}
When $U$ is not degenerate, Lemma \ref{lem : public coin adaptation processing} implies that taking $\cC_0=\cC$,
\begin{equation*}
\frac{1}{|\cC|} \underset{L \in \cC}{\overset{}{\sum}} \frac{n^3 \rho^4_{s_L}}{m^2 2^{2L}} \text{Tr} \left( \Xi_{L,u} \right) \lesssim \max_{L \in \cC} \frac{n^2 \rho^4_{s_L}}{ 2^{2L}} \left( \frac{b}{|\cC| } \wedge 2^L \right).
\end{equation*}
Combining the above with the fact that $s \mapsto L_s = \lfloor s^{-1}\log (1/\rho_s)\rfloor\vee 1$ maps a grid $\cS \subset [s_{\min}, s_{\max}]$ one-to-one to $\cC_0$ with inverse map $L \mapsto s_L$ on $\cC_0$, we obtain 
\begin{equation*}
\frac{1}{|\cC_0|} \underset{L \in \cC_0}{\overset{}{\sum}} \int \log \left( \E^{Y|U=u} (\mathcal{L}^{Y|U=u}_{\pi_L})^2 \right) d\P^U(u) \lesssim c_\alpha \cdot \begin{cases}
\underset{L \in \cC}{\max} \, \frac{n^2 \rho_{s_L}^4 \left( \frac{b}{\log(n) } \wedge 2^L \right)^2}{ 2^{3L} } \bigvee \frac{ n^2 \rho_{s_L}^4 (b \wedge \log(n))}{m 2^L \log(n)}, \\
{ \underset{L \in \cC}{\max} \, \frac{n^2 \rho_{s_L}^4 \left( \frac{b}{\log(n) } \wedge 2^L \right)}{ 2^{2L} } \bigvee \frac{ n^2 \rho_{s_L}^4 (b \wedge \log(n))}{m 2^L \log(n)}},
\end{cases} 
\end{equation*}
where the first case corresponds to a degenerate $U$, the latter to the general (public coin) case. The conditions \eqref{eq : rho condition public coin adaptive lb}-\eqref{eq : rho condition private coin adaptive lb} for $\rho_{s_L}$ yield \eqref{eq : to show pub int log-variance is o(1)}, which in turn finishes the proof.

\section{Lemmas concerning the adaptation upper and lower bounds}\label{sec : adaptation lemmas}

The following lemma controls the Type I error of the adaptive tests defined in Section \ref{sec : adaptation}. 
\begin{lemma}\label{lem : petrov iterated logarithm bound}
Consider for $L \in \N$ and a nonnegative positive integer sequence $K_n$,
\begin{equation*}
S_n(L) := \frac{1}{\sqrt{K_n}} \underset{i=1}{\overset{K_n}{\sum}} \zeta_{i,L}
\end{equation*}
where $(\zeta_{1,L},\dots,\zeta_{K_n,L})$ independent random variables with mean $0$ and unit variance.

Assume that the random variables satisfy Cram\'er's condition, i.e. for some $\epsilon > 0$ and all $t \in (-\epsilon,\epsilon)$, $i=1,\dots,K_n$ and $L \in \cC$, for some set  $\cC  \subset \N$ satisfying $| \cC | \asymp \log(n)$,
\begin{equation*}
\E e^{t \zeta_{i,L}} < \infty.
\end{equation*} 
Then for $K_n\gg( \log \log n )^6$, it holds that
\begin{equation*}
\text{Pr} \left( \underset{L \in \cC}{\max} \; |S_n(L)| \geq  c \sqrt{\log \log (n) }  \right) \to 0
\end{equation*}
for all $c>\sqrt{2}$ as $n \to \infty$. 

If  the random variables are iid Rademacher or are of the form
\begin{equation*}
\zeta_{i,L} = \frac{1}{4Q} \left[\left(\underset{q=1}{\overset{Q}{\sum}} R_{qL} \right)^2 - Q\right]
\end{equation*}
with $R=(R_{1L},\dots,R_{QL})$ independent Rademacher random variables and $Q \in \N$, the statement holds for any sequence $K_n$ as $n\rightarrow\infty$.
\end{lemma}
\begin{proof}
By using union bounds,
\begin{align*}
\text{Pr} &\left( \underset{L \in \cC}{\max} S_n(L) \geq c\sqrt{\log \log (n) }  \right) \leq  \underset{L \in \cC}{\overset{}{\sum}} \text{Pr} \left(|S_n(L)| \geq c \sqrt{\log \log (n) }  \right) \leq  \\
&\underset{L \in \cC}{\overset{}{\sum}} \left[ \text{Pr} \left(S_n(L) \geq c \sqrt{\log \log (n) }  \right) + \text{Pr} \left(- S_n(L) \geq c \sqrt{\log \log (n) }  \right) \right].
\end{align*}
The proof follows by showing that $S_n(L)$ and $-S_n(L)$ are or tend to sub-Gaussian variables with sub-Gaussianity constant less than or equal to $1$, since this allows for bounding the above display by
\begin{align*}
 2\underset{L \in \cC}{\overset{}{\sum}} e^{- \frac{c^2}{2} \log \log(n) } \lesssim \frac{1}{(\log(n))^{c^2/2 -1}}
\end{align*}
and the result follows. 

For the first statement, by Cram\'er's theorem (see e.g. Theorem 7 in Section 8.2 of \cite{petrov2022sums}),
\begin{equation*}\label{eq : law of iterated log lem condition to verify}
\frac{\text{Pr} \left(S_n(L) \geq c\sqrt{\log \log (n) }  \right)}{1 - \Phi( c\sqrt{\log \log (n) }) } = \exp \left( O(1) \cdot \frac{(\log \log n )^3}{\sqrt{K_n}} \right) \left(1 + O \left(\frac{\log \log n }{ \sqrt{ K_n}} \right) \right)  \to 1.
\end{equation*}
Note that the above statement holds for $-S_n(L)$ also. The statement now follows by using $1 - \Phi(x) \leq e^{-x^2/2}$. 

For the second statement, note that by symmetry of the Rademacher distribution, it suffices to consider only $S_n(L)$. In case the $\zeta_{i,L}$'s are iid Rademacher, note that a Chernoff bound yields
\begin{align}
\text{Pr} \left(S_n(L) \geq c \sqrt{\log \log (n) }  \right) \leq \underset{t > 0}{\inf}  e^{
\frac{t^2}{2} - c t \sqrt{\log \log (n)}} = e^{ - \frac{c^2}{2} \log \log (n)}. \label{eq : chernoff bound Rademacher}
\end{align} 
Similarly, for the sum of Rademacher random variables, we have
\begin{align*}
\E \exp \left( \frac{t}{ \sqrt{K_n}} \zeta_{i,L}  \right) &= \E \exp \left( \frac{t}{ 4 Q\sqrt{K_n}} \left[ \underset{q \neq q'}{\overset{Q}{\sum}} R_{qL} R_{q'L}  \right] \right)  \\
&\leq \E \exp \left( \frac{t}{Q \sqrt{K_n}} \left[ \underset{q \neq q'}{\overset{Q}{\sum}} R_{qL} R_{q'L}'  \right] \right),
\end{align*}
where the inequality follows from e.g. Theorem 6.1.1 in \cite{vershynin_high-dimensional_2018} with $R'=(R_{1L}',\dots,R_{QL}')$ independent of $R$. The latter implies that $(R_{qL} R_{q'L}')_{(q,q') \in \{1,\dots,Q\}^2}$ itself is a vector of independent Rademacher random variables, and consequently the above display is further bounded by 
\begin{equation*}
\exp\left( \frac{t^2 Q(Q-1)}{2K_nQ^2} \right) \leq \exp \left( \frac{t^2 }{2 K_n} \right).
\end{equation*}
The proof of the last statement now follows via Chernoff bound as in \eqref{eq : chernoff bound Rademacher}. 
\end{proof}

{The next lemma controls the Type 2 error of the adaptive test in the high-budget case under public coin protocol.}

{
\begin{lemma}\label{lem:adaptive_public_coin_proof}
Consider $S_{\text{II}}(L_s)$ as in \eqref{eq : pub coin adaptive test stat} in the paper. It holds that
\begin{equation*}
\E_f \mathbbm{1} \left\{  S_{\text{II}}(L_s) < 2 \sqrt{\log \log n}  \right\} \leq \alpha/2
\end{equation*}
whenever $f \in H_{C_\alpha \rho_s}^{s,R}$ with $\rho^2 \geq C_0 \sqrt{\log\log (n)} {\frac{2^{L_s}  }{n \sqrt{ \frac{b}{\log(n)} \wedge 2^{L_s}}}}$ for $C_0$ large enough, depending only on $R$. 
\end{lemma}
\begin{proof}
The proof is similar in spirit to that of the risk bound in the finite dimensional, non-adaptive, public coin setting given in Lemma \ref{lem : public coin test risk ub}.

We show below that the event
\begin{equation*}
A = \Big\{ \frac{m'-1}{2\sqrt{b'}} \underset{i=1}{\overset{b'}{\sum}}  (Y_{\text{II}}^{j}(L))_{i} - 1/2)^2 \geq 2\sqrt{\log\log n} \Big\},
\end{equation*}
occurs with $\P_f$-probability greater than $1 - \alpha/4$. {Since on $A$ the condition of Lemma \ref{lem : binomial testing supplement} is satisfied with $c_{\alpha,n}=2\sqrt{\log\log n}$ and consequently, by the conclusion of Lemma \ref{lem : binomial testing supplement}, $\E_f \mathbbm{1} \left\{  S_{\text{II}}(L_s) < 2 \sqrt{\log \log n}  \right\}$ is bounded by $\alpha/2$.}

Following the proof of Lemma \ref{lem : public coin test risk ub} (with $d=\nu_{L_s}$, considering the $\nu_{L_s}$ dimensional vector $f^{\nu_{L_s}}$, and taking $N_{\alpha}=2\sqrt{\log\log n}$), and noting that for $C_0^2>4R^2$
\begin{align*}
\|\tilde{f}^{L_s}\|_2^2\geq \|f\|_2^2/2 -R^2 2^{-2L_s s} \gtrsim \frac{C_0 2^{L_s}\sqrt{\log\log (n)}}{ 2n\sqrt{\frac{b}{\log(n)} \wedge 2^{L_s}}} 
\gtrsim  \frac{C_0 2^{L_s}\sqrt{\log\log (n)}}{ n\sqrt{b' \frac{m'}{m}}},
\end{align*}
 we get that
\begin{align}
\E_f\mathbbm{1}_{A^c} \leq \text{Pr} \left( \frac{m'-1}{24\sqrt{b'}} \underset{i=1}{\overset{b'}{\sum}}  \min \left\{  \frac{C_0\sqrt{\log\log n}2^{L_s} Z_i^2}{2 
 m' \sqrt{b'}\|Z\|_2^2} , 1 \right\}  \leq 2\sqrt{\log\log n} \right).\label{eq:help1}
\end{align}
Considering the intersection with the event $\{\|Z\|_2^2 \leq k2^{L_s} \}$ for some large enough $k > 0$, and noting that by Lemma \ref{lem : Gaussian maximum}, 
\begin{equation*}
\text{Pr} \left( \underset{1 \leq i \leq b'}{\max} Z_i^2 \geq \frac{2m' 
 \sqrt{b'}k}{C_0\sqrt{\log\log n}} \right) \leq 2b'\exp\left(-\frac{m' \sqrt{b'} k}{2C_0\sqrt{\log\log n}}\right)=o(1),
\end{equation*}
the right hand side of $\eqref{eq:help1}$ is further bounded by
$$\text{Pr} \left( \underset{i=1}{\overset{b'}{\sum}}  Z_i^2  \leq  \frac{96 b' m' k }{ C_0(m'-1)}   \right)+o(1)+\alpha/8\leq \alpha/4, $$
where the last inequality holds for large enough choices $m':=\frac{m (b \wedge \log(n))}{\log(n)}$, $b':= \frac{mb }{m'|\mathcal{C}|}\wedge \nu_L$ and large enough choice of $C_0$ (depending on $k$), {see e.g. \eqref{eq:help:testII} in the proof of Lemma \ref{lem : public coin test risk ub}, which finishes the proof of our statement.}
\end{proof}}

{Next we provide the lemmas for the lower bound. From now on in this section,} we consider the setting of Section \ref{sec : adaptation}. That is, let $\tilde{X}^{j}_L$, $\tilde{X}^{j}_{1:L}$ denote the wavelet coefficients of $X^{j}$ as in \eqref{eq : wavelet coefficients X}. Define in addition the matrices
\begin{align*}
\Xi_{L,u}^j &= \E_0 \E_0 \left[ \tilde{X}^{j}_{L} \big| Y^{j} , U=u \right] \E_0 \left[ \tilde{X}^{j}_{L} \big| Y^{j}, U=u \right]^\top, \\  \Xi_{L':L, u}^j &= \E_0 \E_0 \left[ \tilde{X}^{j}_{L':L} \big| Y^{j}, U=u  \right] \E_0 \left[ \tilde{X}^{j}_{L':L} \big| Y^{j}, U=u \right]^\top,
\end{align*}
$\Xi_{L,u} := \sum_{j=1}^m  \Xi_{L,u}^j$ and $\Xi_{u} = \sum_{j=1}^m \Xi_{L_{\min}:L_{\max}, u}^j$. The lemma below allows for extending the data processing inequality of Lemma \ref{lem : trace of fisher info} to the adaptive private coin case, in which extra demands are placed on the communication budget in terms of the budget needing to cover the coordinates corresponding to each resolution level.

\begin{lemma}\label{lem : strong data processing Fisher info adaptive I}
Suppose $Y^{j}$ takes values in a space with cardinality at most $2^b \in \N$, for $j=1,\dots,m$ and let $\cC=\{L_{\min},...,\L_{\max}\}$, for some $L_{\min}< L_{\max} \in \N$ .
There exists $\cC_0 \subset \cC$ such that 
% minimum satisfies
\begin{equation*}
\text{Tr}\left( \Xi_{L,u} \right)  \lesssim \left( \frac{b}{|\cC| } \wedge 2^L \right) \frac{m^2}{n}
\end{equation*}
for all $L \in \cC_0$.
\end{lemma}
\begin{proof}
Define $\Delta_L = \text{Tr}\left( \Xi_{L,u} \right)$ and let $\ell: \{ 1, \dots, L_{\max} - L_{\min} + 1 \} \to \cC$ a map that respects the ordering of the $\Delta_L$'s in the sense that 
\begin{equation*}
\Delta_{\ell(i)} \leq \Delta_{\ell(k)} \; \text{ if } \; i\leq k.
\end{equation*}
Let $\cC_0$ denote the first $\lfloor \frac{L_{\max} - L_{\min}+1}{2} \rfloor$ elements of the collection $\{ \Delta_{\ell1}, \Delta_{\ell(2)}, \dots, \Delta_{\ell(L_{\max} - L_{\min}+1)} \}$. For all $L^\circ \in \cC$, 
\begin{align*}
\text{Tr}\left( \Xi_{L^\circ,u} \right) \leq \frac{2}{|\cC|} \; \underset{ L \in \cC\backslash \cC_0 }{\overset{}{\sum}} \; \text{Tr}\left( \Xi_{L,u} \right).
\end{align*}
By definition of the trace of a matrix, $\sum_L \text{Tr}( \Xi_{L,u} ) = \text{Tr}( \Xi_{L_{\min}:L_{\max},u} )$. By Lemma \ref{lem : trace of fisher info},
\begin{equation*}
\text{Tr}\left( \Xi_{L_{\min}:L_{\max},u} \right) = \underset{j=1}{\overset{m}{\sum}}  \text{Tr}\left( \Xi_{L_{\min}:L_{\max},u}^j \right) \leq \frac{2 \log(2) m^2 b}{n}.
\end{equation*}
Combining the above two displays, we obtain that
\begin{equation*}
\text{Tr}\left( \Xi_{L^\circ,u} \right) \lesssim \frac{m^2 b}{n (|\cC|}.
\end{equation*}
By an application of Lemma \ref{lem : strict DPI} and a straightforward computation as in the proof of Lemma \ref{lem : trace of fisher info}, 
\begin{equation}
\text{Tr}\left( \Xi_{L^\circ,u} \right) \leq \frac{m^2}{n }2^{L^\circ}.
\end{equation}
Combining the two bounds for $\text{Tr}\left( \Xi_{L^\circ,u} \right)$ gives the result.
\end{proof}

The next lemma applies to the adaptive public coin setting. The bound below is slightly more relaxed than the previous one, which relates to the private coin setting. The reason for this is the fact that in the public coin setting, the hyperprior cannot be chosen in an adversarial way because the public coin draw.

\begin{lemma}\label{lem : public coin adaptation processing}
With the notation as in the proof of Theorem \ref{thm:adapt1}, it holds that
\begin{equation*}
\frac{1}{|\cC|} \underset{L \in \cC}{\overset{}{\sum}} \frac{n^3 \rho^4_{s_L}}{ m^2 2^{2L}} \text{Tr} \left( \Xi_{L,u} \right) \lesssim \max_{L \in \cC} \frac{n^2 \rho^4_{s_L}}{ 2^{2L}} \left( \frac{b}{|\cC| } \wedge 2^L \right).
\end{equation*}
\end{lemma}
\begin{proof}
Similarly to the proof of Lemma \ref{lem : strong data processing Fisher info adaptive I}, we note that by the linearity of the trace,
\begin{equation*}
\underset{L \in \cC}{\overset{}{\sum}}  \text{Tr} \left( \Xi_{L,u} \right) =   \text{Tr} \left( \Xi_{u} \right),
\end{equation*}
where $\Xi_{u} = \sum_{j=1}^m \Xi_{L_{\min}:L_{\max}, u}^j$. Lemma \ref{lem : trace of fisher info} yields $\text{Tr} \left( \Xi_{u} \right) \leq 2 \log(2) \frac{b m^2}{n}$. Otherwise, applying Lemma \ref{lem : strict DPI} yields $\text{Tr} \left( \Xi_{L,u} \right) \leq  \frac{2^L m^2}{n}$. Combining these two inequalities yields the result:
\begin{align*}
 \frac{1}{|\cC|} \underset{L \in \cC}{\overset{}{\sum}} \frac{n^2 \rho^4_{s_L}}{ 2^{2L}} \text{Tr} \left( \Xi_{L,u} \right) &\leq  \frac{1}{|\cC|} \underset{L \in \cC}{\overset{}{\sum}} \frac{n^2 \rho^4_{s_L}}{ 2^{2L}} \left( \frac{n}{m^2} \text{Tr} \left( \Xi_{L,u} \right) \bigwedge 2^{L} \right) \\
 &\leq  \max_{L^*} \frac{n^2 \rho^4_{s_L^*}}{ 2^{2L^*}} \left( \frac{n}{m^2} \frac{1}{|\cC|} \underset{L \in \cC}{\overset{}{\sum}} \text{Tr} \left( \Xi_{L,u} \right) \bigwedge  2^{L^*} \right) \\
&\leq  \max_{L^*} \frac{n^2 \rho^4_{s_L^*}}{ 2^{2L^*}} \left( \frac{b}{|\cC|} \bigwedge  2^{L^*} \right).
\end{align*}

\end{proof}

Whereas in the nonadaptive setting of Theorem \ref{thm : detection lb} and Theorem \ref{thm : nonparametric SNWN minimax rate} the local ``chi-square'' based terms need no special data processing treatment, it does in the adaptive case. For each of the $\log(n)$ resolution levels $L$, information on the norm of $\tilde{X}^{j}_L$ is communicated. Using $b \asymp \log(n)$ to this without loss (compared to Theorem \ref{thm : nonparametric SNWN minimax rate}) turns out to be fundamental, as is the content of the lemma below. The proof of the lemma is based on exploiting the fact that even though $2^{-L/2}(\| \sqrt{n/m}\tilde{X}^{j}_L\|_2^2 - 2^L)$ is sub-exponential, the fact that it tends to a sub-Gaussian random variable can be exploited whenever the communication budget is small enough.  % subgaussianiaty subexponentiality technique

\begin{lemma}\label{lem : small budget adaptation processing}
Let $\pi_L$ as in the proof of Theorem \ref{thm:adapt1}, with $\rho_s = \rho_{s_L}$ satisfying \eqref{eq : rho condition private coin adaptive lb} or \eqref{eq : rho condition public coin adaptive lb}. Furthermore, let
\begin{equation*}
A_{L,u} = \underset{j=1}{\overset{m}{\sum}} \log \left( \E_0^{Y^{j}|U=u} \left( \E_0 \left[  \int \frac{d\P^{\tilde{X}^{j}}_f}{d \P^{\tilde{X}^{j}}_0} ( \tilde{X}^{j}_L) d\pi_L(f) \bigg| Y^{j}, U=u \right]^2 \right) \right).
\end{equation*}
Then for arbitrary $\cC\subset \mathbb{N}$,
\begin{equation*}
\frac{1}{|\cC|} \underset{L \in \cC}{\overset{}{\sum}}  \int A_{L,u} d\P^U(u) \lesssim \max_{L \in \cC} \frac{c_\alpha n^2 \rho_{s_L}^4 (b \wedge |\cC|)}{m 2^L |\cC|}.
\end{equation*}
\end{lemma}
\begin{proof}
Recalling the notation from Section \ref{sec: proof:nonparam}, we shall write $\mathscr{L}_{\pi_L}(\tilde{X}^{j}_L) = \int \mathscr{L}_{f}(\tilde{X}^{j}_L) d\pi_L(f)$ with
\begin{equation*}
\mathscr{L}_{f}(\tilde{X}^{j}_L) := \frac{d\P_{f}^{\tilde{X}^{j}}}{d\P_{0}^{\tilde{X}^{j}}}(\tilde{X}^{j}_L) = e^{ \frac{n}{m} f^\top \tilde{X}^{j}_L  - \frac{n}{2m} \| f \|_{2}^2}.
\end{equation*}
Note that, using $\log(x) \leq x - 1$, $\E_0 \mathscr{L}_{\pi_L}(\tilde{X}^{j}_L) = 1$ and the fact that by the law of total probability
\begin{equation*}
\E_0^{Y^{j}|U=u}  \E_0 \left[  \mathscr{L}_{\pi_L}(\tilde{X}^{j}_L) \bigg| Y^{j}, U=u \right] = 1,
\end{equation*}
we obtain that
\begin{equation}
A_{L,u} \leq \underset{j=1}{\overset{m}{\sum}}  \E_0^{Y^{j}|U=u} \left( \E_0 \left[  \mathscr{L}_{\pi_L}(\tilde{X}^{j}_L) - 1 \bigg| Y^{j}, U=u \right]^2 \right). \label{eq:UB_A_L}
\end{equation}

We work out the case where $\pi = N(0,\epsilon_s^2 I_{2^L})$, the case where $\pi = N(0,\epsilon_s^2 \Gamma)$ with $\| \Gamma \| \asymp 1$ follows similarly with additional bookkeeping. Since $f \sim N(0,\epsilon_s^2 I_{2^L})$ with $\epsilon_s = c_\alpha^{1/4} \rho_s/2^{L/2}$,
\begin{align}
\mathscr{L}_{\pi_L}(\tilde{X}^{j}_L) = \underset{i=0}{\overset{2^L-1}{\Pi}} \, \int \frac{e^{\frac{n}{m} {f}_{i} \tilde{X}_{Li}^{j} - \frac{1}{2}(\frac{n}{m} + \epsilon_{s}^{-2}) {f}_{i}^2}}{\sqrt{2\pi \epsilon_s^2}}  d{f}_{i} 
=  \frac{e^{ \frac{n}{m}\epsilon_{s}^{2} \frac{\| \sqrt{\frac{n}{m}} \tilde{X}^{j}_L\|_2^2}{2(1 + \frac{n}{m}\epsilon_{s}^{2})} }}{ ({1 + \frac{n}{m}\epsilon_{s}^{2}})^{2^L/2}} \label{eq : Gaussian MGF locally}
\end{align}
 where the last equality follows by the substitution $u = {f}_{i} \sqrt{1 + \frac{n}{m}\epsilon_{s}^{2}}$ and completing the square. Taking the logarithm and using that $ \frac{(1+x) \log(1+x)}{x} > 1$ for $x > 0$, we find 
\begin{equation}\label{eq : local MGF expansion log inequality}
V^{j}_L := \frac{n}{m}\epsilon_{s}^{2} \frac{\|\sqrt{\frac{n}{m}}\tilde{X}^{j}_L\|_2^2}{2(1 + \frac{n}{m}\epsilon_{s}^{2})} - 2^{L-1} \log(1 + \frac{n}{m}\epsilon_{s}^{2}) \leq \frac{\frac{n}{m}\epsilon_{s}^{2}}{2}  \left( \|\sqrt{\frac{n}{m}}\tilde{X}^{j}_L\|_2^2 - 2^L \right)
\end{equation}
Therefore, using \eqref{eq : Gaussian MGF locally}, Taylor expanding, $(a+b)^2 \leq 2a^2 + 2b^2$ and \eqref{eq : local MGF expansion log inequality}, we can upper bound \eqref{eq:UB_A_L} by
\begin{equation}\label{eq : expansion adaptive local likelihood}
2 \underset{j=1}{\overset{m}{\sum}}   \E_0^{Y^{j}|U=u} \left( \E_0 \left[ V^{j}_L  \bigg| Y^{j}, U=u \right]^2 \right) + 2 \underset{j=1}{\overset{m}{\sum}}  \E_0^{Y^{j}|U=u} (D^j)^2,
\end{equation}
 with
\begin{equation*}
D^j = \E_0 \left[ \underset{k=2}{\overset{\infty}{\sum}}  {\frac{n^k\epsilon_{s}^{2k}}{{2^k}m^k k!}}    \left| \|\sqrt{\frac{n}{m}}\tilde{X}^{j}_L\|_2^2 - 2^L \right|^k  \bigg| Y^{j}, U=u \right].
\end{equation*}

We deal with the two terms in \eqref{eq : expansion adaptive local likelihood} separately. Since conditional expectation contracts the $L_2$-norm, 
\begin{equation*}
\underset{j=1}{\overset{m}{\sum}} \E_0^{Y^{j}|U=u} (D^j)^2 \lesssim m \cdot \underset{k=2}{\overset{\infty}{\sum}} \; \underset{i=2}{\overset{\infty}{\sum}} \frac{n^k c_\alpha^{k/2}\rho_{s}^{2k}  }{{2^k}m^k 2^{kL_s/2 }k!} \frac{n^i c_\alpha^{i/2} \rho_{s}^{2i}  }{{2^i}m^i 2^{iL_s/2} i!} \E W^{i+k} 
\end{equation*}
where $W \overset{d}{=} \left( \|\sqrt{\frac{n}{m}}\tilde{X}^{j}_L\|_2^2 - 2^L \right)$. Furthermore, since $ \|\sqrt{\frac{n}{m}}X^{j}_L\|_2^2 \sim \chi^2_{2^L}$ is sub-exponential, $\E W^{i+k} \leq C^{k+i} (i+k)^{i+k}$, where $C>0$ is a constant (see e.g. Proposition 2.7.1 in \cite{vershynin_high-dimensional_2018}). Then in view of $(i+k)^{i+k} \leq 2^{i+k} i! k!$, we the above display is $O(\frac{c_{\alpha}^2 n^4 \rho^8_s}{m^3 2^{2L_s}})$ whenever $\frac{c_\alpha^2 n^4 \rho_s^8}{C^2 m^4 2^{2L_s}} < 1$. This is certainly the case when $\rho_s^2 \lesssim \left( \frac{\sqrt{m \log(n)}}{n \sqrt{b \wedge \log(n)}} \right)^{\frac{2s}{2s+1/2}}$ and $mb \gtrsim \log(n)$, which yields that 
\begin{equation*}
\underset{j=1}{\overset{m}{\sum}} \E_0^{Y^{j}|U=u} (D^j)^2 \lesssim \frac{c_\alpha^2 n^2 \rho^4_s}{m 2^{L_s/2}} \cdot O\left( \frac{\log(n)}{m (b \wedge \log(n))} \right).
\end{equation*}

It remained to deal with the first term in \eqref{eq : expansion adaptive local likelihood}, where we proceed by a data processing argument. When $b \geq \log(n)$, 
\begin{equation*}
2 \underset{j=1}{\overset{m}{\sum}} \E_0^{Y^{j}|U=u} \left( \E_0 \left[ V^{j}_L  \bigg| Y^{j}, U=u \right]^2 \right) \leq 2 \underset{j=1}{\overset{m}{\sum}} \E_0^{\tilde{X}^{j}} \left(V^{j}_L \right)^2 \leq  \frac{c_\alpha n^2 \rho_s^4}{m 2^{L_s}},
\end{equation*}
in which case the result follows. 

We continue with the case where $b < \log(n)$, which implies $|\cY^{j}| \leq 2^{\log(n)}$. We bound the average of the first terms in \eqref{eq : expansion adaptive local likelihood} over $\cC$, by
\begin{align}\label{eq : continue with this local adaptivity}
\frac{1}{|\cC|} \underset{L \in \cC}{\overset{}{\sum}} \underset{j=1}{\overset{m}{\sum}} \frac{n^2 \rho^4_s}{m^2 2^{L_s}}  \E_0^{Y^{j}|U=u} \left( \E_0 \left[ G^{j}_L  \bigg| Y^{j}, U=u \right]^2 \right) \leq \\ \underset{L \in \cC}{\max}  \frac{n^2 \rho^4_{s_L}}{m^2 2^{L}|\cC|}  \underset{j=1}{\overset{m}{\sum}} \E_0^{Y^{j}|U=u} \text{Tr}( M^{j}(Y^{j})) \nonumber,
\end{align}
where $M^{j}(y) = \E_0 \left[ G^{j}_{\cC}  \bigg| Y^{j}=y, U=u \right] \E_0 \left[ G^{j}_{\cC}   \bigg| Y^{j}=y, U=u \right]^\top$,  $G^{j}_{\cC} =  (G^{j}_{L})_{L \in \cC}$, and $G^{j}_L = \left( \frac{n \rho^2_s}{m 2^{L_s/2}}\right)^{-1} V^{j}_L$. We show below that for all $v=(v_{L})_{L\in \cC}$ of unit norm
\begin{align}
 \E_0^{Y^{j}|U=u}  \langle v_\cC, G^{j}_\cC \rangle^2\leq b\label{eq:help:new1} ,
\end{align}
which by taking $v= G^{j}_\cC/\| G^{j}_\cC\|_2$ yields that \eqref{eq : continue with this local adaptivity} is $O(\max_s \frac{n^2 \rho^4_s}{m 2^{L_s}|\cC|} b)$ as required.

Therefore, it remained to verify \eqref{eq:help:new1}.  For any $\lambda \in \R$, independence and \eqref{eq : local MGF expansion log inequality} yield
\begin{equation*}
\E_0^{X^{j}} e^{\lambda v^\top G^{j}_{\cC}} \leq \underset{L \in \cC}{\overset{}{\Pi}}   \E_0^{X^{j}} e^{ \frac{\lambda}{ 2^{L_s/2}} v_{L} \underset{i=0}{\overset{2^L-1}{\sum}} (\tilde{X}_{Li}^2 - 1)}. 
\end{equation*}
When $\frac{|\lambda|}{2 \cdot 2^{L_s/2}} v_{L} \leq \frac{1}{4}$, the latter can be further bounded by 
\begin{equation*}
\underset{L \in \cC}{\overset{}{\Pi}}  \exp \left( {\lambda^2}v_{L}^2 \right) = \exp \left( {\lambda^2}  \right),
\end{equation*}
see e.g. Lemma 12 in \cite{szabo2022optimal}. In view of $0 \leq K(y|X^{j},u) \leq 1$ and the previously shown sub-exponential behaviour of $\langle v_\cC, G^{j}_\cC \rangle$, we get that
\begin{align*}
&\P^{Y^{j}|U=u}(y)\E_0 \left[ \langle v_\cC, G^{j}_\cC \rangle \bigg| Y^{j} = y, U=u \right]\\
&\quad = \E^{X^{j}}_0 \langle v_\cC, G^{j}_\cC \rangle K(y|X^{j},u)  = 
    \E^{X^{j}}_0  \int_0^\infty \mathbbm{1}\left\{|\langle v_\cC, G^{j}_\cC \rangle| > t \right\} K(y|X^{j},u) dt  \\
&\quad\leq   \int_0^\infty \min \left\{ \P^{X^{j}}_0 \left(|\langle v_\cC, G^{j}_\cC \rangle| > t \right), \P^{Y^{j}|U=u}(y) \right\} dt \leq
  e^{ - t_0} + t_0 \P^{Y^{j}|U=u}(y).
\end{align*}
Taking $t_0 = - \log(\P^{Y^{j}|U=u}(y))$ yields
\begin{equation}
\E_0 \left[ \langle v_\cC, G^{j}_\cC \rangle \bigg| Y^{j} = y, U=u \right] \leq   - 2 \log(\P^{Y^{j}|U=u}(y)).\label{eq:help:new2}
\end{equation}

Furthermore, for $\lambda_y \in \R$ and $y$ satisfying 
\begin{equation}\label{eq : reasonable conditional}
 -{2^{L_s/2 + 2}}\leq \lambda_y = \E_0 \left[ \langle v_\cC, G^{j}_\cC \rangle \bigg| Y^{j} = y, U=u \right] \leq {2^{L_s/2 + 2}},
 \end{equation}
 the argument of Lemma \ref{lem : trace of fisher info} yields
\begin{equation}\label{eq : subGaussian adaptive bound chi-square stat}
\E_0 \left[ \langle v_\cC, G^{j}_\cC \rangle \bigg| Y^{j} = y, U=u \right]^2 \leq - \log \left( \P^{Y^{j}|U=u}(y) \right).
\end{equation}
Note, that if \eqref{eq : reasonable conditional} does not hold, then in view of \eqref{eq:help:new2}, $- \log(\P^{Y^{j}|U=u}(y)) \geq 2^{L_s/2+1}$.

Let us write $p_y=\P^{Y^{j}|U=u}(y)$ and define $\cY^{j}_* = \{y\in \cY^{j}:\, \log(1/p_y) \leq 2^{L_s/2+2}\}$. Since $x \mapsto x \log^2(1/x)$ is increasing on $(0,e^{-2})$, it holds that $p_y \log^2(1/p_y) \leq e^{- 2^{L_s/2+2} + (L_s+4)\log(2)}$ for $y \in (\cY^{j}_*)^c$. Then, in view of \eqref{eq:help:new2} and \eqref{eq : subGaussian adaptive bound chi-square stat} we get that
\begin{align*}
\underset{y \in \cY^{j}}{\overset{}{\sum}} p_y \E_0 \left[ \langle v_\cC, G^{j}_\cC \rangle \bigg| Y^{j} = y, U=u \right]^2 &\leq \underset{y \in \cY^{j}_*}{\overset{}{\sum}}  p_y \log(1/p_y) + 4 \underset{y \in (\cY^{j}_*)^c}{\overset{}{\sum}} p_y \log^2(1/p_y)\\
& \lesssim \log |\cY^{j}|+ e^{- 2^{L_s/2+2} + (L_s+4)\log(2)} \lesssim b,
\end{align*}
concluding the proof of \eqref{eq:help:new1} and hence the lemma.
\end{proof}

\section{Proof of Lemma \ref{lem : key gaussian kernel maximizer}}\label{sec : supplement Brascamp-Lieb details}

We start by introducing some short hand notations for convenience. Write, for $x \in \R^{vk}$, $v \in \{1,m\}$,
\begin{equation*}
\phi_v( x) = \E^H \frac{p_H^v}{p_0^v}(x) = \E^H e^{  H^\top (\sum_{j=1}^{v} \Lambda^{-1} x^j) - \frac{ v}{2} \|\Lambda^{-1/2} H\|_2^2},
\end{equation*}
with $\phi_m(x) p_0^m(x) = \E^H p_H^m(x)$, $x=(x^1,...,x^m)$, and $\Pi_{j=1}^{m}  \phi_1(x^j) = \Pi_{j=1}^{m} \E^H p_H(x^j)$. Let $P_0^m$ denote the measure corresponding to the Lebesgue density $p_0^m$. Furthermore, recall that
\begin{align*}
\cQ\equiv\cQ(M,\Sigma) := \bigg\{ q \in & L_1(\R^{mk},P_0^m) : \; \; q \geq 0, \; \; \frac{q}{\int q(x) dP_0^m (x)} \leq M \; \; P_0^m-a.e.,  \\  & \; \int x \, q(x) dP_0^m (x) = 0, \text{ and }  \; \; \frac{\int x x^\top \, q(x) dP_0^m (x)}{\int q(x) d P_0^m(x)} = \Sigma \bigg\},
\end{align*}
where $L_1(\R^{mk},P_0^m)=\{f: \R^{mk}\mapsto \R,\,\text{such that}\, \int f  d P_0^m(x)<\infty \}$.

Let $\lambda \equiv \lambda_{mk}$ denote the Lebesgue measure on $\R^{mk}$, define for $r \in L_1(\R^{mk},\lambda)$ nonnegative,
\begin{equation}\label{eq : sup to attain gaussian kernel constant def}
  F(r) := \frac{\int \phi_m(x) \,  r(x) dx}{\int \underset{j=1}{\overset{m}{\Pi}}  \phi_1(x^j) \,  r(x) dx} \in [0,\infty],
\end{equation}
and set $G(q) := F(q p_0^m)$. Since $G( c q) = G(q)$ for any constant $c \in \R$, it suffices to show that
\begin{equation*}
\bar{G} = \underset{q \in \cQ}{\sup} \, G(q) \leq \frac{\int \phi_m(x) \,  dN(0,\Sigma)(x)}{\int \underset{j=1}{\overset{m}{\Pi}}  \phi_1(x^j) \, dN(0,\Sigma)(x)}.
\end{equation*}
We will proceed through the following steps.

\begin{enumerate}
    \item First, we show that the supremum $\bar{G}$ is finite and attained in $\cQ$, i.e. by the Banach-Alaoglu theorem there exists $q \in \cQ$ such that $G(q) = \bar{G}$.
    \item We will then consider $\cQ_2$, the class of all $Q \in L_1(\R^{2km}, \lambda)$ such that $x_1 \mapsto Q(x_1, x_2)$ is in $\cQ$ for $P_0^m$-almost every $x_2 \in \{ x_1 \mapsto Q(x_1, x_2) \nequiv 0 \}$ and $x_2 \mapsto Q(x_1, x_2)$ is in $\cQ$ for $P_0^m$-almost every $x_1 \in \{x_2 \mapsto Q(x_1, x_2)  \nequiv 0 \}$. It holds that
    \begin{equation*}
        G_{2}(Q) := \frac{\int \phi_m(x_1) \phi_m(x_2) \,  p_0^m(x_1) p_0^m(x_2) Q(x_1,x_2) d(x_1,x_2)}{\int \underset{j=1}{\overset{m}{\Pi}}  \phi_1(x^j_1) \phi_1(x^j_2) \,   p_0^m(x_1) p_0^m(x_2) Q(x_1,x_2) d(x_1,x_2)}
    \end{equation*}
    satisfies $ \underset{Q \in \cQ_2}{\sup} \, G_2(Q) = \bar{G}^2$.% and that for a maximizer $q \cQ$, $(x_1,x_2) \mapsto q(x_1)q(x_2)$ is a maximizer of $G_2$.
    \item Next, we show that $(x_1,x_2) \mapsto q(\frac{x_1- x_2}{\sqrt{2}})q(\frac{x_1 + x_2}{\sqrt{2}})$ is a maximizer of $G_2$ whenever $q \in \cQ$ is a maximizer of $G$. This is a consequence of the conjugacy between the observation and the distribution of the parameter $H$.
    \item Then it will be shown that for any maximizer $Q$ of $G_2$, $x_1 \mapsto Q(x_1,x_2)$ maximizes $G$ for  $P_0^m$-almost every $x_2$.
    \item Combining the above steps, we obtain that for any maximizer $q$, an appropriately rescaled convolution of $q$ with itself is also a maximizer, i.e. 
    \begin{equation*}
        F( \sqrt{2} (qp_0^m) \ast (qp_0^m)(\sqrt{2} \, \cdot)) = \bar{G},
    \end{equation*}
    where $\ast$ denotes convolution.
   \item By repeated application of Step 5 and the central limit theorem, the result follows.
\end{enumerate}

\emph{Step 1.} For $q \in \cQ$, define the normalizing constant as $C_q := (\int q dP_0^m)^{-1}$. As linear combinations and products of nonnegative convex functions are convex, the mapping
\begin{equation*}
x \mapsto \underset{j=1}{\overset{m}{\Pi}}  \E^H e^{  H^\top \Lambda^{-1} x^j - \frac{1}{2} \|\Lambda^{-1/2}H\|_2^2}
\end{equation*}
is convex. Then Jensen's inequality gives
\begin{align*}
\frac{\int \E^H e^{  H^\top (\sum_{j=1}^{m} \Lambda^{-1} x^j) - \frac{1}{2} \| \Lambda^{-1/2}H\|_2^2}  q(x) dP_0^m(x)}{\int \underset{j=1}{\overset{m}{\Pi}}  \E^H e^{  H^\top \Lambda^{-1} x^j - \frac{1}{2} \|\Lambda^{-1/2}H\|_2^2} \, q(x) dP_0^m(x)} \leq 
\frac{ C_q \int \E^H e^{  H^\top ( \Lambda^{-1} \sum_{j=1}^{m} x^j) - \frac{1}{2} \|\Lambda^{-1/2}H\|_2^2}  q(x) dP_0^m(x)}{ \underset{j=1}{\overset{m}{\Pi}}  \E^H e^{ C_q \int  H^\top \Lambda^{-1}  x^j  q(x) dP_0^m(x) - \frac{1}{2} \|\Lambda^{-1/2}H\|_2^2} \,  }.
\end{align*}
Since $X = (X_1,\dots,X_m) \sim q dP_0^m$ has mean $0$, the denominator on the lhs is equal to  $(\E^H e^{-\frac{1}{2} \|\Lambda^{-1/2}H\|_2^2 })^m > 0$. This means that the denominator in the above display is bounded away from $0$ over $q$. Since $q C_q \leq M$ a.e., the numerator is bounded above by $M \int \E^H p_H^m(x)dx = M$. We can conclude that the supremum of \eqref{eq : sup to attain gaussian kernel constant def} over $qp_0^m$, $q\in\cQ$ is finite. It is easy to construct a $q^* \in \cQ$ such that $G(q^*) > 0$, so we can conclude that $0 < \bar{G} < \infty$.

Let $q_t$ be a maximizing sequence for $G$, rescale $q_t$ such that $\int q_t P^m_0 = 1$ and note that $q_t \in \cQ$ and $q_t$ is contained in the $L_\infty(\R^{mk})$ ball of radius $M$. Since $L_\infty(\R^{mk})$ is the dual of $L_1(\R^{mk}, \lambda)$, by the Banach-Alaoglu theorem the $L_\infty(\R^{mk})$ ball of radius $M$ is weak-$\ast$-compact. Therefore, there exists a subsequence, again denoted by $q_t$, along which $q_t \overset{\text{wk}-\ast}{\to} q$ for some $q$ in the $L_\infty(\R^{mk})$ ball of radius $M$. Since $x=(x^1,\dots,x^m) \mapsto \phi_m(x)$ is in $L_1(\R^{mk}, P_0^m)$, the weak-$\ast$-convergence implies that
\begin{equation*}
\int \phi_m\left(x\right) \, q_t(x) dP_0^m(x) \to \int \phi_m\left(x\right) \, q(x) dP_0^m(x).
\end{equation*}
Similarly, 
\begin{equation*}
\int \Pi_{j=1}^{m} \phi_1\left( x^j\right) \, q_t(x) dP_0^m(x) \to \int \Pi_{j=1}^{m}  \phi_1\left( x^j\right) \, q(x) dP_0^m(x) \in (0,
\infty),
\end{equation*}
where the boundedness away from $0$ has been concluded earlier on in the proof. We have now obtained that
\begin{equation}\label{eq : q is a maximizer}
\bar{G} = \underset{t \to \infty}{\lim} \frac{\int \phi_m\left(x\right) \, q_t(x) dP_0^m(x)}{\int \Pi_{j=1}^{m} \phi_1\left( x^j\right) \, q_t(x) dP_0^m(x)} = \frac{\int \phi_m\left(x\right) \, q(x) dP_0^m(x)}{\int \Pi_{j=1}^{m} \phi_1\left( x^j\right) \, q(x) dP_0^m(x)}.
\end{equation}
Since $q_t \in \cQ$, we have
\begin{equation*}
\int x \, q_t(x) dP_0^m(x) = 0 \; \text{ and } \; \int x x^\top \, q_t(x) dP_0^m(x) = \Sigma \; \text{ for all }t.
\end{equation*}
As $x \mapsto 1$, $x \mapsto x$ and $x \mapsto x x^\top$ are all $P_0^m$ integrable, the weak-$\ast$-convergence yields that $\int q(x) dP_0^m(x) = 1$, $\int x \, q(x) dP_0^m(x) = 0$ and $ \Sigma =  \int x x^\top \, q(x) dP_0^m(x) $. Since we have that $\int \zeta(x) q_t(x) dP_0^m(x) \to \int \zeta(x) q_t(x) dP_0^m(x)$ for every continuous and bounded function $\zeta : \R^{mk} \to \R^{mk}$, the Portmanteau lemma yields that $\int_B q dP_0^m \geq 0$ for all open sets $B$ so $q \geq 0$ almost everywhere. We conclude that $G(q) = \bar{G}$ and $ q \in \cQ$.

\emph{Step 2.} Let $Q \in \cQ_2$ be given. By definition, the marginals $x_1 \mapsto Q(x_1,x_2), \; x_2 \mapsto Q(x_1,x_2)$ are in $\cQ$ $P_0^m$-a.e. and $\E^H p_H(x)dx = \phi_m(x) p_0^m(x)dx$ is equivalent to the Lebesgue measure, hence
\begin{align*}
 G_2(Q) &= \int \phi_m(x) p_0^m(x_1) \int \phi_m(x) \,   p_0^m(x_2) Q(x_1,x_2) dx_2 dx_1\\
&\leq \bar{G} \int \phi_m(x) p_0^m(x_1) \int \Pi_{j=1}^{m} \phi_1( x^j_2) p_0^m(x_2) Q(x_1,x_2) dx_2 dx_1 \\
 &\leq  \bar{G}^2 \int \Pi_{j=1}^{m} \phi_1( x^j_2) p_0^m(x_2) \int \Pi_{j=1}^{m} \phi_1( x^j_1) p_0^m(x_1) Q(x_1,x_2) dx_1 dx_2.
\end{align*}
Let $q \in \cQ$ be a maximizer of $G$. Then, the above steps hold with equality for $Q(x_1,x_2) := q(x_1)q(x_2)$. For almost every $x_1 \in \{ q \neq 0 \} \equiv \{ x_2 \mapsto Q(x_1, x_2) \nequiv 0 \}$, 
\begin{equation*}
\frac{Q(x_1,x_2)}{\int Q(x_1,x_2) dP_0^m(x_2)} = \frac{q(x_2)}{\int q(x_2) dP_0^m(x_2) } \leq M.
\end{equation*}
By similar calculations, the rescaled marginal has the correct mean and covariance. By symmetry, we conclude that the marginals of $(x_1,x_2) \mapsto q(x_1)q(x_2)$ belong to $\cQ$ and it is a maximizer of $G_2$ over $\cQ_2$.

\emph{Step 3.} % Need p to also be Gaussian for this step!!
Consider a maximizer $q \in \cQ$ of $G$. By a change of variables $w_1 = (x_1-x_2)/\sqrt{2}$ and $w_2= (x_1 + x_2)/\sqrt{2}$,
\begin{align*}
&\int \phi_m (x_1) \phi_m (x_2) q\left(\frac{x_1- x_2}{\sqrt{2}}\right)q\left(\frac{x_1 + x_2}{\sqrt{2}}\right) p_0^m(x_1) p_0^m(x_2) d(x_1,x_2) = \\ &\int \phi_m \Big(\frac{w_1 + w_2}{\sqrt{2}} \Big) \phi_m \Big( \frac{w_1 - w_2}{\sqrt{2}} \Big) q(w_1)q(w_2) p_0^m\left(\frac{w_1 - w_2}{\sqrt{2}}\right) p_0^m\left(\frac{w_1 + w_2}{\sqrt{2}}\right) d(w_1,w_2).
\end{align*}

Since $p_0^m$ is a Gaussian density, $p_0^m\left(\frac{w_1 - w_2}{\sqrt{2}}\right) p_0^m\left(\frac{w_1 + w_2}{\sqrt{2}}\right) = p_0^m(w_1) p_0^m(w_2)$. This follows from direct computation, but it characterizes Gaussian functions in general, see e.g. Theorem 1 in \cite{carlen_superadditivity_1991}. Likewise, for $H'$ an independent copy of the centered Gaussian random vector $H$, $\frac{H - H'}{\sqrt{2}}$ and $\frac{H + H'}{\sqrt{2}}$ are independent and furthermore equal in distribution to $H$. Therefore,
\begin{align*}
&\phi_m \Big( \frac{w_1+ w_2}{\sqrt{2}} \Big) \phi_m \Big(\frac{w_1 - w_2}{\sqrt{2}} \Big)\\
 &\qquad=  \E^{(H,H')} e^{H^\top \Lambda^{-1}\sum_{j=1}^{m} \frac{w_1^j + w_2^j}{\sqrt{2}} + (H')^\top \Lambda^{-1} \sum_{j=1}^{m} \frac{w_1^j - w_2^j}{\sqrt{2}} - \frac{ m }{2} \|\Lambda^{-1/2}H\|_2^2 - \frac{ m }{2} \|\Lambda^{-1/2}H'\|_2^2}  \\
&\qquad=\E^{(H,H')} e^{\left(\frac{H + H'}{\sqrt{2}}\right)^\top \Lambda^{-1} \sum_{j=1}^{m} w_1^j - \frac{ m }{2} \|\Lambda^{-1/2}\frac{H + H'}{\sqrt{2}}\|_2^2 + \left( \frac{H - H'}{\sqrt{2}}\right)^\top \Lambda^{-1} \sum_{j=1}^{m} w_2^j  - \frac{ m }{2} \|\Lambda^{-1/2}\frac{H - H'}{\sqrt{2}}\|_2^2}  \\
&\qquad=\phi_m (w_1) \phi_m (w_2). 
\end{align*}
Since $(x_1,x_2) \mapsto q(x_1)q(x_2)$ was established to be a maximizer of $G_2$ in the second step, the above establishes that $(x_1,x_2) \mapsto q(\frac{x_1- x_2}{\sqrt{2}})q(\frac{x_1 + x_2}{\sqrt{2}})$ is a maximizer of $G_2$ also.

\emph{Step 4.} Next, we will show that for a maximizer $Q \in \cQ_2$ of $G_2$, $x \mapsto Q(x,w)$ is in $\cQ$ and is a maximizer of $G$ for almost every $w$. We prove this by contradiction. Take an arbitrary measurable set $A \subset \R^{mk}$ s.t. $\lambda(A) > 0$. Note that Gaussian measures are equivalent to the Lebesgue measure, so both $\E^H P_H^m (A)$ and $\Pi_{j=1}^{m} \E^H P_H^1 (A)$ are bounded away from zero. Suppose that for $Q \in \cQ_2$  a maximizer of $G_2$ it holds that
\begin{align}
&\int_A \phi_m  (w)  \int \phi_m(x)   \, Q(x,w)   dP_0^m(x)dP_0^m(w) \nonumber \\
&\qquad<  \bar{G} \int_A \phi_m(w)  \int \Pi_{j=1}^{m} \phi_1\left(  x^j\right)   \, Q(x,w)  dP_0^m(x)dP_0^m(w).  \label{eq : strict step 4}
\end{align}
Since the marginal $w \mapsto Q(x,w)$ is in $\cQ$ for almost every $x \in \{w \mapsto Q(x,w) \nequiv 0 \}$,
\begin{align*}
%\int \phi_m\left( \sum_{j=1}^{m} x^j_1\right) \phi_m\left( \sum_{j=1}^{m} x^j_{2}\right) \, Q(x_1,x_2)  p(x_1) p(x_2) d(x_1,x_2) &= \\ 
&\bar{G}^2 \int \Pi_{j=1}^{m} \phi_1\left(  w^j\right) \Pi_{j=1}^{m} \phi_1\left( x^j\right) \, Q(x,w)  (dP_0^m \times P_0^m)(x,w)  \\
&\qquad \geq \bar{G}  \int \Pi_{j=1}^{m} \phi_1\left( x^j\right) \int \phi_m(w) \,  Q(x,w) dP_0^m(w)  dP_0^m (x).  
\end{align*}
Likewise, $x \mapsto Q(x,w)$ is in $\cQ$ for almost every $u \in A^c \cap \{ x \mapsto Q(x,w) \nequiv 0 \}$, so
\begin{align*}
& \bar{G} \int \Pi_{j=1}^{m} \phi_1\left( x^j\right) \int_{A^c} \phi_m(w) \,  Q(x,w) dP_0^m(w)  dP_0^m (x)   \\
&\qquad\geq \int_{A^c} \phi_m(w) \int \phi_m(x) \, Q(x,w) dP_0^m (x) dP_0^m(w).
\end{align*}
Together with \eqref{eq : strict step 4} and the second to last display, we obtain that
\begin{align*}
&\bar{G}^2 \int \Pi_{j=1}^{m} \phi_1\left(  w^j\right) \Pi_{j=1}^{m} \phi_1\left( x^j\right) \, Q(x,w)  (dP_0^m \times P_0^m)(x,w)  \\
 &\qquad>\int \int  \phi_m \left(x\right)  \phi_m(w) \,  Q(x,w) dP_0^m(w)  dP_0^m (x),
\end{align*}
which contradicts with $Q$ maximizing $G_2$. 

\emph{Step 5.} Let $q \in \cQ$ be a maximizer of $G$ over $\cQ$, where $q$ is normalized such that $\int q dP_0^m =1$. Define $q_2$ as
\begin{equation*}
q_2(x) :=  \int q\left(\frac{x- w}{\sqrt{2}}\right)q\left(\frac{x + w}{\sqrt{2}}\right) dP_0^m(u).
\end{equation*}
The map $x \mapsto q\left(\frac{x- w}{\sqrt{2}}\right)q\left(\frac{x + w}{\sqrt{2}}\right):=Q(x,w)$ is in $\cQ$ for almost all $w$ s.t. $Q(x,w) \nequiv 0$ and as a consequence of the previous step, it is a maximizer of $G$ for such $w$. Hence, $q_2(x)$ is a maximizer of $G$:
\begin{align*}
\int  \phi_m(x) q_2(x) dP_0^m(x) &= 
 \int \int \phi_m(x) q\left(\frac{x- w}{\sqrt{2}}\right)q\left(\frac{x + w}{\sqrt{2}}\right)  dP_0^m(x) \, dP_0^m(w)   \\
 &= \bar{G} \int \Pi_{j=1}^{m} \phi_1\left(  x^j\right) q_2(x)  dP_0^m(x).
\end{align*}
Let $h \in L_1(\R^{mk},p_0^m)$. Using again that $p_0^m\left(\frac{w_1 - w_2}{\sqrt{2}}\right) p_0^m\left(\frac{w_1 + w_2}{\sqrt{2}}\right) = p_0^m(w_1) p_0^m(w_2)$ and applying a change of variable $w=\sqrt{2}w-x$, we get
\begin{align*}
\int  h(x) q_2(x)  p_0^m(x) dx &= 
 \int \int h(x) q\left(\frac{x- w}{\sqrt{2}}\right)q\left(\frac{x + w}{\sqrt{2}}\right)\, p_0^m\left(\frac{x- w}{\sqrt{2}}\right)  \, p_0^m\left(\frac{x + w}{\sqrt{2}}\right) 
\, dx dw \\ 
&= 
 \int \int h(x) q\left(\sqrt{2}x - w\right)q\left(w\right)\, p_0^m\left(\sqrt{2}x - w\right)  \, p_0^m\left(w\right) 
\, dx \sqrt{2} dw \\
&= \int h(x) \sqrt{2} (qp_0^m) \ast  (qp_0^m) (\sqrt{2}x) dx,
\end{align*}
where $f \ast g$ denotes convolution. Therefore, $qp_0^m$ being a probability density with mean $0$ and covariance $\Sigma$ implies that $q_2 p_0^m$ is too. So, $q_2\in\cQ$ and maximizes $G$.

\emph{Step 6.} Consider now 
$q_4 \in \cQ$ defined by $q_4(x) := \int q_2\left(\frac{x- w}{\sqrt{2}}\right)q_2\left(\frac{x + w}{\sqrt{2}}\right) dP_0^m(w)$. Since $q_2 \in \cQ$ is a maximizer, the above steps imply that $G(q_4) = \bar{G}$ and by a similar computation as above,
\begin{align*}
q_4(x)  p_0^m(x) &=   \sqrt{4} \underset{}{\overset{4}{\Asterisk}} (q p_0^m)  (\sqrt{4}x),
\end{align*}
where $\underset{}{\overset{4}{\Asterisk}} r$ denotes $r \ast r \ast r \ast r$. Repeating the above steps, we obtain a maximizer $q_{2^N} \in \cQ$ of $G$ for $N\in\mathbb{N}$ which satisfies
\begin{align*}
r_{2^N} (x) := q_{2^N}(x) p_0^m (x) &= \int q_{2^{N-1}}\left(\frac{x- w}{\sqrt{2}}\right)q_{2^{N-1}}\left(\frac{x + w}{\sqrt{2}}\right) p_0^m(x) p_0^m(w) dx dw \\
&= \sqrt{2} \int q_{2^{N-1}} \left(\sqrt{2}x - w\right) p_0^m\left(\sqrt{2}x - w\right) q_{2^{N-1}} \left(w\right) dP_0^m(w) \\
&= \sqrt{2} \left( q_{2^{N-1}}  p_0^m \right) \ast \left( q_{2^{N-1}}  p_0^m \right) ( \sqrt{2} x ).
\end{align*}
We conclude that 
\begin{equation*}
r_{2^N}(x)={2^{N/2}}  \underset{}{\overset{{2^N} }{\Asterisk}} (q p_0^m)  ({2^{N/2}} x)
\end{equation*}
and
\begin{equation*}
\frac{\int \phi_m(x) r_{2^N} (x) dx}{\int \Pi_{j=1}^{m}\phi_1( x^j) r_{2^N} (x) dx} = G(q_{2^N} ) = \bar{G}
\end{equation*}
for all $N\in\mathbb{N}$. Let $r = qp_0^m$. The characteristic function of $r_{2^N} $ equals, for $s \in \R^{mk}$,
\begin{align*}
\cF r_{2^N} (s) &:= \int e^{- i s^\top x} r_{2^N} (x) dx = \int e^{-i \frac{s^\top}{{2^{N/2}}  } x} \, \underset{}{\overset{{2^N} }{\Asterisk}} r \left( x \right) dx = \Big(\int e^{- i \frac{s^\top}{{2^{N/2}}  } x} \, r(x) dx\Big)^{{2^N} } \\
&= \Big(\int \Big(1 - i \frac{s}{{2^{N/2}}  } x  -  \frac{(s^\top x)^2}{2^{N+1} } + O\Big( \frac{(s^\top x)^3}{{2^{3N/2}} } \Big) \Big)  \, r(x) dx\Big)^{2^{N}}.
\end{align*}
Since $r$ has mean $0$, covariance $\Sigma$ and bounded third moment (by the boundedness of $q$ and $p_0^md\lambda$ possessing a third moment), $\cF r_{2^N}(s) \to e^{ - \frac{1}{2}s^\top \Sigma s}$. Consequently, $r_{2^N} d\lambda$ converges weakly to a Gaussian distribution with mean $0$ and covariance $\Sigma$. In particular, $\int \phi r_{2^N} d\lambda \to \int \phi dN(0,\Sigma) $ for all $ \phi \in C^\infty(\R^{mk})$, so
\begin{equation*}
\bar{G} = \underset{N \to \infty}{\lim}  \frac{\int \phi_m (x) \, r_{2^N}(x) dx }{\int \Pi_{j=1}^{m} \phi_1( x^j) \,r_{2^N}(x) dx } = \frac{\int \phi_m (x) \, dN(0,\Sigma)(x) }{\int \Pi_{j=1}^{m} \phi_1( x^j) \, dN(0,\Sigma)(x) },
\end{equation*}
which finishes the proof.

\section{Definitions and notations for wavelets}\label{sec: wavelets}
In this section we briefly introduce wavelets and collect some properties used in the article. For a more detailed and elaborate introduction of wavelets we refer to \cite{hardle:2012, gine:nickl:2016}.

In our work we consider the Cohen, Daubechies and Vial construction of compactly supported, orthonormal, $N$-regular wavelet basis of $L_2[0,1]$, see for instance \cite{cohen:1993}. First for any $N\in\mathbb{N}$ one can follow Daubechies' construction of the father $\phi(.)$ and mother $\psi(.)$ wavelets with $N$ vanishing moments and bounded support on $[0,2N-1]$ and $[-N+1,N]$, respectively, see for instance \cite{daubechies:1992}. The basis functions are then obtained as
\begin{align*}
\big\{ \phi_{j_0m},\psi_{jk}:\, m\in\{0,...,2^{j_0}-1\},\quad j> j_0,\quad k\in\{0,...,2^{j}-1\} \big\},
\end{align*}
with $\psi_{jk}(x)=2^{j/2}\psi(2^jx-k)$, for $k\in [N-1,2^j-N]$, and $\phi_{j_0 k}(x)=2^{j_0}\phi(2^{j_0}x-m)$, for $m\in [0,2^{j_0}-2N]$, while for other values of $k$ and $m$, the basis functions are specially constructed, to form a basis with the required smoothness property.  For notational convenience we take $j_0=0$ and denote the father wavelet by $\psi_{00}$. Then the function $f\in L_2[0,1]$ can be represented in the form
\begin{align*}
f=\sum_{j=j_0}^{\infty}\sum_{k=0}^{2^{j}-1}f_{jk}\psi_{jk},
\end{align*}
with $f_{jk}=\langle f,\psi_{jk}\rangle$. Note that in view of the orthonormality of the wavelet basis the {$L_2$-norm} of the function $f$ is equal to
\begin{align*}
\|f\|_2^2=\sum_{j=j_0}^{\infty}\sum_{k=0}^{2^{j}-1}f_{jk}^2.
\end{align*}

Next we give an equivalent definition of Sobolev spaces using wavelets. Let us define the norm for $s\in(0,N)$ as
\begin{align*}
\|f\|_{\cH^s}^2=\sum_{j\geq j_0} 2^{2js}\sum_{k=0}^{2^j-1}f_{jk}^2.
\end{align*}
 Then the Sobolev space $\cH^{s}([0,1])$ and Sobolev ball $\cH^{s,R}([0,1])$ of radius $R>0$ are defined as
\begin{align*}
\cH^{s}=\{f\in L_2[0,1]:\, \|f\|_{\cH^s}<\infty \},\quad\text{and}\quad \cH^{s,R}([0,1])=\{f\in L_2[0,1]:\, \|f\|_{\cH^s}<R \},
\end{align*}
respectively. The above definition of the Sobolev space and norm is equivalent to the classical one based on the weak derivatives of the function.

\end{document}